\documentclass[11pt,a4paper]{amsart}[2000]
\usepackage{geometry}
\usepackage{amsmath}
\usepackage{amssymb}
\usepackage{amsthm}
\usepackage{mathrsfs}
\usepackage[hang]{footmisc}
\usepackage[all,cmtip]{xy}
\usepackage[linktocpage]{hyperref}
\usepackage{tikz, tikz-cd}
\usepackage{color}
\usepackage{mathtools}
\usepackage{enumitem}
\usepackage{stackengine}

\thanks{}

\allowdisplaybreaks

\theoremstyle{plain}

\newtheorem{Thm}{Theorem}[section]

\theoremstyle{definition}

\theoremstyle{plain}
\newtheorem{thm}[Thm]{Theorem}
\newtheorem{lem}[Thm]{Lemma}
\newtheorem{cor}[Thm]{Corollary}
\newtheorem{prop}[Thm]{Proposition}

\theoremstyle{definition}
\newtheorem{defn}[Thm]{Definition}

\newtheorem{eg}[Thm]{Example}
\newtheorem{rmk}[Thm]{Remark}

\newtheorem{innercustomthm}{Theorem}
\newenvironment{customthm}[1]
{\renewcommand\theinnercustomthm{#1}\innercustomthm}
{\endinnercustomthm}

\newtheorem{innercustomcor}{Corollary}
\newenvironment{customcor}[1]
{\renewcommand\theinnercustomcor{#1}\innercustomcor}
{\endinnercustomcor}

\newtheorem{innercustomrmk}{Remark}
\newenvironment{customrmk}[1]
{\renewcommand\theinnercustomrmk{#1}\innercustomrmk}
{\endinnercustomrmk}

\newcommand{\B}{B}
\newcommand{\A}{A}
\newcommand{\J}{J}
\newcommand{\K}{\mathcal{K}}
\newcommand{\D}{D}
\newcommand{\Ch}{D}
\newcommand{\Zh}{\mathcal{Z}}
\newcommand{\E}{E}
\newcommand{\Oh}{\mathcal{O}}

\newcommand{\T}{{\mathbb T}}

\newcommand{\N}{{\mathbb N}}
\newcommand{\Z}{{\mathbb Z}}
\newcommand{\C}{{\mathbb C}}
\newcommand{\Q}{{\mathbb Q}}
\newcommand{\F}{{\mathbb F}}

\newcommand{\aut}{\mathrm{Aut}}
\newcommand{\supp}{\mathrm{supp}}
\newcommand{\out}{\mathrm{Out}}

\newcommand{\hnn}{\mathrm{HNN}}

\newcommand{\eps}{\varepsilon}
\numberwithin{equation}{section}

\newcommand{\id}{\mathrm{id}}

\newcommand{\halpha}{\widehat{\alpha}}
\newcommand{\calpha}{\widehat{\alpha}}

\newcommand{\tih}{\widetilde {h}}

\newcommand\set[1]{\left\{#1\right\}}  
\newcommand\mset[1]{\left\{\!\!\left\{#1\right\}\!\!\right\}}

\newcommand{\IG}[0]{\mathbb{G}}

\newcommand{\CA}[0]{\mathcal{A}} \newcommand{\CB}[0]{\mathcal{B}}
\newcommand{\CC}[0]{\mathcal{C}} \newcommand{\CD}[0]{\mathcal{D}}
 
\newcommand{\CG}[0]{\mathcal{G}} \newcommand{\CH}[0]{\mathcal{H}}

 \newcommand{\CN}[0]{\mathcal{N}}
 \newcommand{\CP}[0]{\mathcal{P}}
\newcommand{\CQ}[0]{\mathcal{Q}} 
 \newcommand{\CT}[0]{\mathcal{T}}
 
\newcommand{\CW}[0]{\mathcal{W}}

\newcommand{\Ra}[0]{\Rightarrow}
\newcommand{\La}[0]{\Leftarrow}
\newcommand{\LRa}[0]{\Leftrightarrow}

\newcommand{\quer}[0]{\overline}
\newcommand{\eins}[0]{\mathbf{1}}			
\newcommand{\diag}[0]{\operatorname{diag}}
\newcommand{\ad}[0]{\operatorname{Ad}}
\newcommand{\ev}[0]{\operatorname{ev}}
\newcommand{\fin}[0]{{\subset\!\!\!\subset}}
\newcommand{\diam}[0]{\operatorname{diam}}
\newcommand{\Hom}[0]{\operatorname{Hom}}
\newcommand{\dst}[0]{\displaystyle}
\newcommand{\spp}[0]{\operatorname{supp}}
\newcommand{\lsc}[0]{\operatorname{Lsc}}
\newcommand{\del}[0]{\partial}

\newcommand{\fix}[0]{\operatorname{Fix}}
\newcommand{\GU}[0]{\CG^{(0)}}

\theoremstyle{definition}

\numberwithin{equation}{Thm}

\title[Boundary actions]{Boundary actions of Bass-Serre Trees and the applications to $C^*$-algebras}	

\begin{document}
\global\long\def\floorstar#1{\lfloor#1\rfloor}
\global\long\def\ceilstar#1{\lceil#1\rceil}	

\global\long\def\B{B}
\global\long\def\A{A}
\global\long\def\J{J}
\global\long\def\K{\mathcal{K}}
\global\long\def\D{D}
\global\long\def\Ch{D}
\global\long\def\Zh{\mathcal{Z}}
\global\long\def\E{E}
\global\long\def\Oh{\mathcal{O}}

\global\long\def\T{{\mathbb{T}}}
\global\long\def\BR{{\mathbb{R}}}
\global\long\def\N{{\mathbb{N}}}
\global\long\def\Z{{\mathbb{Z}}}
\global\long\def\C{{\mathbb{C}}}
\global\long\def\Q{{\mathbb{Q}}}

\global\long\def\aut{\mathrm{Aut}}
\global\long\def\supp{\mathrm{supp}}

\global\long\def\eps{\varepsilon}

\global\long\def\id{\mathrm{id}}

\global\long\def\halpha{\widehat{\alpha}}
\global\long\def\calpha{\widehat{\alpha}}

\global\long\def\tih{\widetilde{h}}

\global\long\def\opFol{\operatorname{F{\o}l}}

\global\long\def\opRange{\operatorname{Range}}

\global\long\def\opIso{\operatorname{Iso}}
\global\long\def\opisom{\operatorname{Isom}}
\global\long\def\dimnuc{\dim_{\operatorname{nuc}}}

\global\long\def\set#1{\left\{  #1\right\}  }


\global\long\def\mset#1{\left\{  \!\!\left\{  #1\right\}  \!\!\right\}  }

\global\long\def\Ra{\Rightarrow}
\global\long\def\La{\Leftarrow}
\global\long\def\LRa{\Leftrightarrow}

\global\long\def\quer{\overline{}}
\global\long\def\eins{\mathbf{1}}
\global\long\def\diag{\operatorname{diag}}
\global\long\def\ad{\operatorname{Ad}}
\global\long\def\ev{\operatorname{ev}}
\global\long\def\fin{{\subset\!\!\!\subset}}
\global\long\def\diam{\operatorname{diam}}
\global\long\def\Hom{\operatorname{Hom}}
\global\long\def\dst{{\displaystyle }}
\global\long\def\spp{\operatorname{supp}}
\global\long\def\spo{\operatorname{supp}_{o}}
\global\long\def\del{\partial}
\global\long\def\lsc{\operatorname{Lsc}}
\global\long\def\GU{\CG^{(0)}}
\global\long\def\HU{\CH^{(0)}}
\global\long\def\AU{\CA^{(0)}}
\global\long\def\BU{\CB^{(0)}}
\global\long\def\CUU{\CC^{(0)}}
\global\long\def\DU{\CD^{(0)}}
\global\long\def\QU{\CQ^{(0)}}
\global\long\def\TU{\CT^{(0)}}
\global\long\def\CUUU{\CC'{}^{(0)}}
\global\long\def\dom{\operatorname{dom}}
\global\long\def\ran{\operatorname{ran}}
\global\long\def\AUl{(\CA^{l})^{(0)}}
\global\long\def\BUl{(B^{l})^{(0)}}
\global\long\def\HUp{(\CH^{p})^{(0)}}
\global\long\def\sym{\operatorname{Sym}}
\global\long\def\stab{\operatorname{Stab}}
\newcommand{\cat}[0]{\operatorname{CAT}(0)}
\global\long\def\properlength{proper}
\global\long\def\deg{\operatorname{deg}}
\global\long\def\isom{\operatorname{Isom}}
\global\long\def\interior#1{#1^{\operatorname{o}}}
	\global\long\def\ln{\operatorname{ln}}

\author{Xin Ma}
	
	\address{X. Ma:  Institute for Advanced Study in Mathematics, Harbin Institute of Technology, Harbin, China, 150001}
  \email{xma17@hit.edu.cn}


\date{\today}

\author{Daxun Wang}
	\address{D. Wang: Yau Mathematical Sciences Center, Tsinghua University, Beijing, China}
\email{wangdaxun@mail.tsinghua.edu.cn}

\author{Wenyuan Yang}
\address{W. Yang: Beijing International Center for Mathematical Research, Peking University, Beijing 100871, China P.R.}
\email{wyang@math.pku.edu.cn}

\begin{abstract}
In this paper, we study Bass-Serre theory from the perspectives of $C^*$-algebras and topological dynamics. In particular, we investigate the actions of fundamental groups of graphs of groups on their Bass-Serre trees and the associated boundaries, through which we identify new families of $C^*$-simple groups including certain tubular groups, fundamental groups of certain graphs of groups with one vertex group acylindrically hyperbolic and outer automorphism groups $\operatorname{Out}(BS(p, q))$ of Baumslag-Solitar groups. In addition, we study $n$-dimensional Generalized Baumslag-Solitar ($\text{GBS}_n$) groups. We first recover a result by Minasyan and Valiunas on the characterization of $C^*$-simplicity for $\text{GBS}_1$ groups and identify new $C^*$-simple $\text{GBS}_n$ groups including the Leary-Minasyan group. These $C^*$-simple groups also provide new examples of $C^*$-selfless groups and highly transitive groups. Moreover, we demonstrate that natural boundary actions of these $C^*$-simple fundamental groups of graphs of groups give rise to the new purely infinite crossed product $C^*$-algebras.
\end{abstract}
\maketitle

\tableofcontents

\section{Introduction}
In recent years, there has been an increasing acknowledgment of the profound interplay between the fields of group theory, topological dynamics, and $C^*$-algebras. Groups and topological dynamical systems have emerged as valuable sources of examples and motivations for exploring $C^*$-algebras, particularly through the construction of group $C^*$-algebras and crossed product $C^*$-algebras. 

A countable discrete group $G$ is said to be $C^*$-\textit{simple} if its reduced $C^*$-algebra is simple. Numerous groups are already recognized as $C^*$-simple, including all Powers groups (\cite{Powers}) and we refer to the standard reference \cite{Harpe} for further background on this topic and refer to e.g. \cite{Su} and \cite{Raum} for constructions of non-discrete $C^*$-simple groups.  It was shown in \cite[Theorem 2.35]{D-G-O} that all acylindrically hyperbolic groups introduced by Osin in \cite{Osin} are $C^*$-simple, if and only if there are no non-trivial finite normal subgroups. Moreover, a new proof for this result is provided in \cite[Theorem 0.2]{A-Da} by verifying such acylindrically hyperbolic groups
have Property $P_{naive}$. Many groups admitting certain actions on trees are acylindrically hyperbolic by \cite[Theorem 2.1]{M-O}. Most recently, Minasyan and Valiunas have obtained in \cite{M-V} characterizations of $C^*$-simplicity of (finitely generated) one-relator groups and \textit{Generalized Baumslag-Solitar} (GBS) groups.

Kalantar and Kennedy provided a dynamical characterization of $C^*$-simplicity in \cite{K-K}, demonstrating that $G$ is $C^*$-simple if and only if the action of $G$ on its \textit{Furstenberg boundary} is topologically free. See also \cite{B-K-K-O}. Here, the Furstenberg boundary is a universal object within the category of all $G$-\textit{boundary} actions, as also introduced by Furstenberg. In applications, Furstenberg boundaries are often described in very abstract terms. By contrast, in the realm of geometric group theory, various concrete geometric boundaries associated with a group $G$ have been shown to be $G$-boundaries; for instance, the Gromov boundary for non-elementary hyperbolic groups, the end boundary of infinitely ended groups, and Floyd boundary of relatively hyperbolic groups. Consequently, the topological freeness on these boundaries can be transferred to the Furstenberg boundary due to its universal nature. Thus, it becomes highly demanding  to explore the topological freeness of the action on  geometric boundaries to detect the $C^*$-simplicity of groups. 

Furthermore, topological actions on geometric boundaries often exhibit paradoxicality in many aspects. This paradoxical nature usually leads to the corresponding crossed product $C^*$-algebra being \textit{purely infinite}, which is a significant regularity property in the structure theory of $C^*$-algebras (see \cite{K-Rord} and \cite{Kir-Rord}). Moreover, it plays a crucial role in the celebrated classification theorem established by Kirchberg and Phillips (see, e.g., \cite{Phillips} and \cite{Kir}). To be more precise, Anantharaman-Delaroche in \cite{A-D} as well as Laca and Spielberg in \cite{L-S}, independently introduced \textit{strong boundary} actions and \textit{local contractivity} for topological dynamical systems. Subsequently, in \cite{J-R}, Jolissaint and Robertson studied $n$-\textit{filling} actions, which is a generalization of strong boundary actions. These notions imply the crossed product of the action is purely infinite if the action is further assumed to be topologically free. Furthermore, the first author has generalized in \cite{M2} these results on pure infiniteness to minimal topological free systems with \textit{dynamical comparison} for actions of non-amenable groups. Then,  Gardella, Geffen, Kranz and Naryshkin have proven in \cite{G-G-K-N} that all minimal topologically free \textit{topologically amenable} actions of acylindrically hyperbolic groups have dynamical comparison and thus by \cite{M2}, the crossed products are nuclear, simple and purely infinite.

Motivated by this research,  we investigate in this paper the $C^*$-simplicity of groups that admit actions on trees and the pure infiniteness of the crossed product from the associated boundary actions.   In the paper \cite{M-W-Y}, we  study the case for groups acting on $\cat$ spaces. 

The theory of groups acting on trees is known as Bass-Serre theory, which has become an important tool in Geometric Group Theory. The fundamental theorem in Bass-Serre theory, recorded as Theorem \ref{thm fundamental theorem in Bass serre} in this paper, says that a group $G$ acting on a tree $T$ without inversions admits a decomposition as a graph of groups, denoted by $\IG=(\Gamma, \CG)$, where $\Gamma=T/G$ is the quotient graph with a family of subgroups $\CG$ consisting of appropriate stabilizers $G_v$ and $G_e$ for each (lifted) vertex $v$ and edge $e$ in $\Gamma$. Furthermore, $G$ as well as the tree $T$ can be recovered from these data $\IG=(\Gamma, \CG)$ as the so-called fundamental group $\pi_1(\IG, v)$ and  Bass-Serre tree $X_\IG$ of the graph of groups. We refer to \cite{Ba} and \cite{Ser} for these constructions, which are briefly recalled in Section \ref{sec: prelim} for their Definitions   \ref{defn: graph of groups} and \ref{defn: fundamental group}.

Amalgamated free products and HNN extensions provide basic examples of graphs of groups.   It also includes many other interesting groups such  as Baumslag-Solitar (BS) groups, and certain outer automorphism groups of BS groups in \cite{C} to be studied in this paper. 

An important large class of groups under our consideration is the family of  \textit{tubular} groups, named by Bridson. These are finitely generated groups that act on trees without inversions such that the vertex stabilizers are isomorphic to $\Z^2$ and the edge stabilizers are isomorphic to $\Z$. This class  includes some right-angled Artin groups, Wise's non-Hopfian $\cat$ group $W$ in \cite{Wise1}, simple curve group $G_d$ in \cite{Wise}, and Brady-Bridson group $\operatorname{BB}(p, r)$ in \cite{B-B} to study isoperimetric spectrum. The action of tubular groups on $\cat$ spaces have been studied in, e.g., \cite{Wise}, \cite{Button1}.  Quasi-isometric classification of tubular groups have been investigated in, e.g., \cite{Cashen}.  We refer to \cite{Wise},  \cite{Cashen} and \cite{Button1} for more information on this topic.

Another prevalent large class of groups investigated in this paper are $n$-\textit{dimensional Generalized Baumslag-Solitar} groups, abbreviated as $\text{GBS}_n$ groups for simplicity. These are fundamental groups of finite graph of groups whose vertex and edge groups are all isomorphic to $\Z^n$.  When $n=1$, $\text{GBS}_1$ groups are nothing but usual GBS groups. This class of groups have been studied in relation to JSJ decomposition of manifolds in \cite{Fo}. It has been shown in \cite{C-V} that all GBS groups possess the Haagerup property. 
For high dimension $\text{GBS}_n$ groups, a famous $\text{GBS}_2$ example, called \textit{Leary-Minasyan group} 
and denoted by $G_P$, was introduced in \cite{LM} serving as the first example disproving several conjectures for $\cat$ groups. We refer to \cite{L2}, \cite{M06}, \cite{C}, \cite{CR-K-Z}, \cite{W}, \cite{LM}, and \cite{Button2}  for more details on $\text{GBS}_n$ groups. 

The graph-of-groups decomposition has been proven useful in   the study of $C^*$-algebras. In the $C^*$-setting, de la Harpe and Pr\'{e}aux have investigated the actions of groups acting on trees in \cite{H-P} and provided useful criteria for the $C^*$-simplicity of amalgamated free products and HNN extensions. Subsequently, Ivanov and Omland in \cite{I-O}, as well as Bryder, Ivanov, and Omland in \cite{B-I-O}, established new criteria for $C^*$-simplicity for these two classes of groups. On the other hand, Brownlowe, Mundey, Pask, Spielberg, and Thomas introduced a graph $C^*$-algebraic method in \cite{B-M-P-S-T} to study the actions of fundamental groups on the boundaries of their Bass-Serre trees. In particular, they demonstrated pure infiniteness for a $C^*$-algebra $C^*(\IG)$ associated to a certain locally finite non-singular graph of groups $\IG$, which is stably isomorphic to the crossed product of boundary actions of the corresponding Bass-Serre tree. 

So far, in light of works \cite{H-P}, \cite{L-M}, \cite{I-O}, and \cite{B-I-O} on strong hyperbolicty and its relation to boundary actions in the sense of Furstenburg, the main difficulty in showing $C^*$-simplicity of a group acting on trees (and the pure infiniteness of the crossed product of its boundary action) is to establish the topological freeness of the boundary action. The cases on amalgamated free products, HNN extensions and (finitely generated) GBS groups have been studied in \cite{H-P},  \cite{I-O}, \cite{B-I-O}, and \cite{B-M-P-S-T}.

In this paper, we initiate a new approach in this direction combining the geometric method with a combinatorial interpretation of the boundary action of Bass-Serre trees in \cite{B-M-P-S-T}  based on \cite{Ba}. This allows us to establish several novel criteria determining the topological freeeness of the action of a fundamental group of a graph of groups on the boundary of its Bass-Serre trees (see Lemma \ref{lem: strongly faithful permanent}, Proposition \ref{prop: non locally finite topo free 1}, and Proposition \ref{prop: infinite GBS topo free}) and thus leads to new examples of $C^*$-simple groups and purely infinite crossed product $C^*$-algebras stemming from Bass-Serre theory.
The following are the applications of our key result, Theorem \ref{thm: boundary action},  
providing many new $C^*$-simple groups beyond the scope of acylindrically hyperbolic groups and their amenable actions in the literature mentioned above. Besides the Theorem \ref{thm: A} to \ref{thm: E}, we also remark that Theorem \ref{thm: boundary action} still has additional potential to detect $C^*$-simplicity and pure infiniteness of other classes of groups and their actions.

A graph of groups $\IG=(\Gamma, \CG)$ is said to be \textit{reduced} if there exists no \textit{collapsible} edges in the sense of Definition \ref{defn: reduced graph} in the graph $\Gamma$. Any graph of groups $\IG$ could be transformed to a reduced graph while the fundamental groups remain same by Remark \ref{rmk: reduced tree}. Therefore, in many cases in studying fundamental groups, it suffices to investigate reduced graphs of groups  (see Remark \ref{rmk: reduced advantage}). Applying Theorem \ref{thm: boundary action} mentioned above to the reduced graphs, we have the following stating that local $C^*$-simplicity determining the global $C^*$-simplicity of fundamental groups of graphs of groups as well as pure infiniteness of crossed product of boundary actions.

\begin{customthm}{A}[Theorem \ref{thm: reduced graph C star simple}]\label{thm: A}
     Let $\IG=(\Gamma, \CG)$ be a reduced graph of groups. Suppose
    \begin{enumerate}[label=(\roman*)]
        \item $\IG$ contains a non-degenerated edge $e$ (see Definition \ref{defn: non-degenerated}) with $o(e)\neq t(e)$ such that $G_{o(e)}$ and $G_{t(e)}$ are amenable and the group $G_{o(e)}*_{G_e}G_{t(e)}$ is $C^*$-simple; 
        \item or $\IG$ contains a non-ascending loop $e$ (see Definition \ref{defn:ascending}) with $o(e)=t(e)$ such that $G_{o(e)}$ is amenable and $G_{o(e)}*_{\alpha_e(G_e)}$ is $C^*$-simple. 
    \end{enumerate}
    Then $\pi_1(\IG, v)$ is $C^*$-simple and the crossed product $C(\overline{\partial_\infty X_\IG})\rtimes_r\pi_1(\IG, v)$ is a unital simple separable purely infinite $C^*$-algebra.
    \end{customthm}

A further application of Theorem \ref{thm: A} is to determine the $C^*$-simplicity of tubular groups. We have the following result.

\begin{customthm}{B}[Theorem \ref{thm: tubular}]\label{thm: B}
    Let $\IG=(\Gamma, \CG)$ be a tubular graph of groups. Suppose $\Gamma$ contains a loop $e$ with $o(e)=t(e)=v$. Denote by $(m_1, n_1)=\alpha_e(1)$ and $(m_2, n_2)=\alpha_{\bar{e}}(1)$. Suppose $|m_1|\neq |m_2|$ or $|n_1|\neq |n_2|$. Then the tubular group $\pi_1(\IG, v)$ is $C^*$-simple and the crossed product $A=C(\partial_\infty X_\IG)\rtimes_r \pi_1(\IG, v)$ is a unital Kirchberg algebra satisfying the UCT.
\end{customthm}

This particularly applies to a class $\CC_{t, 2}$ of tubular groups $G$ with the following presentation in which $(m_i, n_i)\neq (0, 0)$ and $(k_i, l_i)\neq (0, 0)$ for any $i=1,2$.
\[G=\langle a,b, x, y\ |\  [a, b]=1, x^{-1}a^{m_1}b^{n_1}x=a^{m_2}b^{n_2}, y^{-1}a^{k_1}b^{l_1}y=a^{k_2}b^{l_2}\rangle.  \] 
Note that this class contains the Wise's groups and the Brady-Bridson group above (see Example \ref{eg: tubular}) and so on. See more in \cite{Button1}.

\begin{customcor}{1}[Corollary \ref{cor: wise BB}]\label{cor: wiseBB 1}
    The group $G\in \CC_{t, 2}$ is $C^*$-simple if $(|m_1|, |n_1|)\neq (|m_2|, |n_2|)$ or  $(|k_1|, |l_1|)\neq (|k_2|, |l_2|)$. In particular,
the Wise's non-Hopfian $\cat$ group $W$, Brady-Bridson group $\operatorname{BB}(p, r)$ for $0<p<r$, and Wise's simple curve examples $G_d$ for $d\geq 2$ are $C^*$-simple.
\end{customcor}

Note that tubular groups are usually not one-relator groups by the presentation.
In addition, this class (even the subclass $\CC_{t,2}$) also includes many non-acylindrically hyperbolic groups (see Remark \ref{rmk: non ah tubular}). Therefore, our Theorem \ref{thm: B} and Corollary \ref{cor: wiseBB 1} have provided new examples of $C^*$-simple groups comparing to the results in \cite{D-G-O}, \cite{A-Da} and \cite{M-V}.
In addition, our Theorem \ref{thm: boundary action} can also be applied to another large class of reduced graphs which seems not covered by the combination of \cite[Theorem 2.1]{M-O} and \cite[Theorem 2.35]{D-G-O}.

\begin{customthm}{C}[Theorem \ref{thm: reduced graph c star simple 2}]\label{thm: C}
    Let $\IG=(\Gamma, \CG)$ be a reduced graph of groups.  Suppose there exists an edge $e\in E(\Gamma)$ satisfying that the vertex group $G_{o(e)}$ is an acylindrically hyperbolic group containing no non-trivial finite normal subgroup and the edge group $G_e\simeq \Z$. Then, setting $v=o(e)$, the group $\pi_1(\IG, v)$ is $C^*$-simple and the crossed product $C(\overline{\partial_\infty X_\IG})\rtimes_r \pi_1(\IG, v)$ is a unital simple separable purely infinite $C^*$-algebra.
\end{customthm}

The structure of outer automorphism groups of Baumslag-Solitar groups has been explored in \cite{Co}, \cite{G-H-M-R}, \cite{L}, and \cite{C}. When $p$ does not divide $q$ properly, $\out(BS(p, q))$ is amenable as shown in Remark \ref{rmk: amenable outer auto of BS}. Otherwise, $\out(BS(p, q))$ admits a non-singular (but not reduced) graph of groups decomposition, for which Theorem \ref{thm: boundary action} yields the following result. 

\begin{customthm}{D}[Theorem \ref{thm: out bs C simple}]\label{thm: D}
   $\out(BS(p, q))$ is $C^*$-simple if and only if $q=2p$ and $p>1$ 
\end{customthm}
Furthermore, the actions of $\out(BS(p, 2p))$ for $p > 1$ on the boundaries of the corresponding Bass-Serre trees produce unital Kirchberg algebras that satisfy the UCT and are thus classifiable by their K-theory. See more details in Theorem \ref{thm: out bs main}.
On the other hand, Combining \cite[Theorem 2.1]{M-O} with Lemma \ref{lem: minimal out bs}, Proposition \ref{prop: minimal on bdry} and Proposition \ref{prop: minimal 1}, the group $\out(BS(p, 2p))$ can be proven to be acylindrically hyperbolic and this leads to an alternative way to the $C^*$-simplicity of this group once the property $P_{naive}$, equivalently, \textit{infinite conjugacy class property} (ICC) for $\out(BS(p, 2p))$ is  verified directly. See Remark \ref{rmk: out bs acy hyperbolic}.

$\text{GBS}_n$ groups are not acylindrically hyperbolic (see Remark \ref{rmk: non ah gbs}). In addition, most of them are not one-relator groups. For this class, a consequence of Theorem \ref{thm: boundary action} is to recover in our framework a characterization of $C^*$-simplicity of (finitely generated) GBS groups first proven by Minasyan and Valiunas in \cite[Proposition 9.1]{M-V}. See Theorem \ref{thm: GBS C star simple}. Moreover, Theorem \ref{thm: boundary action} yields new infinitely generated $C^*$-simple GBS groups in Proposition \ref{prop: non locally finite GBS} as well as $\text{GBS}_n$ groups in Theorem \ref{thm: main GBSn} and Example \ref{eg: gbsn diagonal}. In particular, we have the following result that is applicable to the Leary-Minasyan group $G_P$ mentioned above. The definition of the matrices $A_e$ and $A_{\bar{e}}$ below can be found in Subsection \ref{subsec: gbsn}. 

\begin{customthm}{E}[Corollary \ref{cor: gbs2 including LM}]\label{thm: E}
    Let $\IG=(\Gamma, \CG)$ be a $\text{GBS}_2$ graph of groups containing a non-ascending loop $e$ such that   $M=A^{-1}_eA_{\bar{e}}$ is a unitary in $\operatorname{GL}(2, \Q)$ of the form
    \[\begin{bmatrix}
        \cos\theta & \sin\theta\\
        -\sin\theta & \cos\theta
    \end{bmatrix}
    \]
    in which $\theta$ is irrational (e.g., the loop $e$ as a subgraph yielding Leary-Minasyan group $G_p$ ).  Then $\pi_1(\IG, v)$ is $C^*$-simple. Morevoer, the crossed product $C(\overline{\partial_\infty X_\IG})\rtimes_r\pi_1(\IG, v)$ is a unital Kirchberg $C^*$-algebra satisfying the UCT.
    \end{customthm}

As a corollary, all $C^*$-simple groups mentioned above also have $2$-paradoxical towers in the sense of  \cite{G-G-K-N} because their actions on the boundary of the Bass-Serre trees are shown to be topologically free strong boundary actions, which are $2$-filling. Then by \cite[Theorem B]{G-G-K-N}, one has the following application to topological amenable actions.

\begin{customcor}{2}\label{cor: 2}
  Let $G$ be the $C^*$-simple groups mentioned above, i.e., 
  \begin{itemize}
      \item 
      certain fundamental groups of reduced graphs of groups in Theorem \ref{thm: boundary action}, \ref{thm: A}, and \ref{thm: C} 
      \item 
      certain tubular groups in Theorem \ref{thm: B}, 
      \item 
      outer automorphism groups of BS groups $\out(BS(p, 2p))$ for $p>1$ in Theorem \ref{thm: D}, and
      \item 
      all (finitely generated) GBS group that are non-elementary, non virtually $\F_n\times \Z$ and not isomorphic to $BS(1, n)$, Octopus GBS group in Proposition \ref{prop: non locally finite GBS}, $\text{GBS}_n$ groups in in Theorem \ref{thm: main GBSn} as well as Example \ref{eg: gbsn diagonal}, $\text{GBS}_2$ groups including Leary-Minasyan group in Theorem \ref{thm: E}.
  \end{itemize}Let $H$ be another countable discrete group. Suppose $G\times H\curvearrowright Z$ is a topological amenable minimal topologically free action on a compact metric space $Z$. Then its reduced crossed product is a unital Kirchberg algebra satisfying the UCT and thus classifiable by its Elliott invariant.
\end{customcor}

On the other hand, we remark that Theorem \ref{thm: boundary action}, Theorems \ref{thm: A} and  \ref{thm: C} cannot be deduced from Corollary \ref{cor: 2} because the action there may not be topologically amenable. Therefore, these crossed products could be non-nuclear.  

Finally, to the best knowledge of authors, these groups also provide new examples of \textit{highly transitive} groups. We refer to , e.g., \cite{H-O} and \cite{F-L-M-S} for the definition. Combining \cite[Theorem A]{F-L-M-S} with the results proven in this paper, it is straightforward to obtain the following result for all $C^*$-simple groups above, as we have shown that these groups admit minimal strongly hyperbolic actions on their Bass-Serre trees with a topologically free boundary action (the condition of strong hyperbolicity is called ``of general type'' in \cite{F-L-M-S}). 

\begin{customcor}{3}
     Let $G$ be the $C^*$-simple groups mentioned above, i.e., 
  \begin{itemize}
      \item 
      certain fundamental groups of reduced graphs of groups in Theorem \ref{thm: boundary action}, \ref{thm: A}, and \ref{thm: C} 
      \item 
      certain tubular groups in Theorem \ref{thm: B}, 
      \item 
      outer automorphism groups of BS groups $\out(BS(p, 2p))$ for $p>1$ in Theorem \ref{thm: D}, and
      \item 
      all (finitely generated) GBS group that are non-elementary, non virtually $\F_n\times \Z$ and not isomorphic to $BS(1, n)$, Octopus GBS group in Proposition \ref{prop: non locally finite GBS}, $\text{GBS}_n$ groups in in Theorem \ref{thm: main GBSn} as well as Example \ref{eg: gbsn diagonal}, $\text{GBS}_2$ groups including Leary-Minasyan group in Theorem \ref{thm: E}.
  \end{itemize}
    Then $G$ is highly transitive.
      \end{customcor}

\begin{customrmk}{F}
It was recently established in~\cite{O} that if a countable discrete group $G$ admits a topologically free strong boundary action, then $G$ is $C^*$-selfless; that is, its reduced group $C^*$-algebra $C_r^*(G)$ is selfless in the sense of~\cite{Ro}. Consequently, every group $G$ appearing in Corollary~\ref{cor: 2} is $C^*$-selfless, since, by Theorem~\ref{thm: boundary action}, their actions on the boundaries of the corresponding Bass--Serre trees have been shown to be topologically free strong boundary actions in this paper.
\end{customrmk}

\subsection*{Organization of the paper} In Section \ref{sec: prelim}, we recall the necessary background on Bass-Serre theory, topological dynamical systems, and $C^*$-algebras. In Section \ref{sec: main}, we primarily establish Theorem \ref{thm: boundary action} and its application to reduced graphs of groups to prove Theorem \ref{thm: A}. In Section \ref{sec: reduced graph}, we apply Theorem \ref{thm: boundary action} to prove Theorems \ref{thm: B} and \ref{thm: C}. In Section \ref{sec: Out(BS)}, we examine outer automorphism groups of BS groups and prove Theorem \ref{thm: D}. Finally, in Section \ref{sec: GBS}, we investigate the class of $\text{GBS}_n$ groups.

\section{Preliminaries}\label{sec: prelim}
In this section, we record several necessary backgrounds on Bass-Serre theory, topological dynamical systems, and $C^*$-algebras.

\subsection{Bass-Serre Theory}
The theory of graphs of groups, also known as Bass-Serre theory is a useful tool in geometric group theory and neighborhood areas. We first recall elements in this theory following the notations and terminologies in \cite{Ser} and \cite{Ba}.

\begin{defn}\label{defn: graph}
A (oriented) \textit{graph} $\Gamma$ consists of the \textit{vertex set} $V(\Gamma)$ and \textit{(oriented) edge set} $E(\Gamma)$, equipped with  two maps
$$E(\Gamma)\rightarrow V(\Gamma)\times V(\Gamma),\ \ \ \ e\mapsto (o(e),t(e))$$
and
$$E(\Gamma)\rightarrow E(\Gamma), \ \ \ \  e\mapsto \overline{e}$$
subject to the following condition: for each $e\in E(\Gamma)$ we have $\overline{\overline{e}}=e$, $\overline{e}\neq e$ and $o(e)=t(\overline{e})$.
An edge $e\in E(\Gamma)$ is refereed to as an \textit{oriented edge} of $\Gamma$, and $\overline{e}$ denotes the \textit{inverse} edge of $e$. The vertex $o(e)=t(\overline{e})$ is called the \textit{origin} of $e$, and the vertex $t(e)=o(\overline{e})$ is called the \textit{terminus} of $e$. The number of \textit{oriented} edges with origin at a vertex $v$ is called the \textit{valence} of  $v$.
\end{defn}

A \textit{path} $c$ in $\Gamma$ is either a vertex or a sequence $(e_1, \dots, e_n)$ of edges in $E(\Gamma)$ such that $t(e_i)=o(e_{i+1})$ for any $1\leq i\leq n-1$. The length of $c$, denoted by $\ell(c)$, is defined to be $0$ in the former case and $n$ in the latter. Moreover, we define the origin of $c$ to be $o(c)=o(e_1)$ and the terminus of $c$ to be $t(c)=t(e_n)$. 

\begin{defn}\label{defn: non singular and locally finite for graph}
    A graph (in particular a tree) $\Gamma$ is said to be \textit{locally finite} if the valence of any vertex is finite. It is called  \textit{non-singular} if there is no vertex in $\Gamma$ with valence one.
\end{defn}

\begin{defn}\label{defn: graph of groups}
    A \textit{graph of groups} $\IG=(\Gamma,\CG)$ consists of a connected graph $\Gamma$, and a family of groups $\CG=\{G_v, G_e: v\in V(\Gamma), e\in E(\Gamma)\}$ in which $G_v$ is referred as a \textit{vertex group} for each $v\in V(\Gamma)$ and $G_e$ is called an \textit{edge group}  satisfying $G_e=G_{\overline{e}}$    for each $e\in E(\Gamma)$, together with a monomorphism $\alpha_e: G_e\hookrightarrow G_{o(e)}$ for each $e\in E(\Gamma)$.    
    \end{defn}

For definiteness, we denote by $1_v$ and $1_e$ for the neutral elements of the vertex group $G_v$ and edge group $G_e$, respectively. Moreover, in this paper, we only consider countable graphs of groups, i.e., $\Gamma$ is countable, and all $G_v$ and $G_e$ are countable discrete. The following two examples are ``building blocks'' for general graph of groups.

\begin{eg}\label{example:edge of groups} One-edge graph of groups is pictured as follows. 
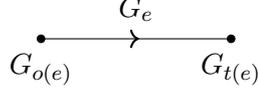
\begin{center}
\begin{tikzpicture}
      \draw (0,2) -- (2.5,2) node[midway, sloped, above] {};
      \draw[->, thick] (1.27,2) -- (1.29,2);

      \node[label={below, yshift=-0.3cm:}] at (0,2.8) {};
      \node[label={below, yshift=-0.3cm:}] at (1.25,2.4) {$G_e$};
      \node[label={below, yshift=-0.3cm:}] at (0,1.6) {$G_{o(e)}$};
      \node[label={below, yshift=-0.3cm:}] at (2.5,1.6) {$G_{t(e)}$};
      
      \tikzset{enclosed/.style={draw, circle, inner sep=0pt, minimum size=.1cm, fill=black}}
      
      \node[enclosed, label={right, yshift=.2cm:}] at (0,2) {};
      \node[enclosed, label={right, yshift=.2cm:}] at (2.5,2) {};
\end{tikzpicture}
\end{center}

The monomorphisms are $\alpha_e: G_e\hookrightarrow G_{o(e)} $ and $\alpha_{\overline{e}}: G_e\hookrightarrow G_{t(e)}$.
\end{eg}

\begin{eg}\label{example:loop of groups}
One-loop graph of groups is pictured as follows.

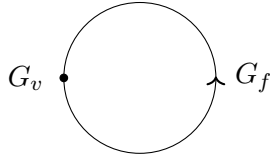
\begin{center}
\begin{tikzpicture}
      \node[label={below, yshift=-0.3cm:}] at (8.5,3.2) {};
      \node[label={below, yshift=-0.3cm:}] at (8.5,2) {$G_f$};
      \node[label={below, yshift=-0.3cm:}] at (5.5,2) {$G_v$};

      \tikzset{enclosed/.style={draw, circle, inner sep=0pt, minimum size=.1cm, fill=black}}

      \node[enclosed, label={right, yshift=.2cm:}] at (6,2) {};

      \draw(6, 2) arc(-180:180:1);
      \draw[->, thick] (8,1.98) -- (8,2.02);
\end{tikzpicture}
\end{center}
The monomorphisms are $\alpha_f: G_f\hookrightarrow G_v $ and $\alpha_{\overline{f}}: G_f\hookrightarrow G_v$ where $v=o(f)=t(f)$.
\end{eg}

\begin{defn}\label{defn: locally finite graph of groups}
A graph of groups $\IG=(\Gamma,\CG)$ is said to be \textit{locally finite} if 
\begin{enumerate}
	\item every vertex in $\Gamma$ has finite valence; and
	\item $[G_{o(e)}: \alpha_e(G_e)]<\infty$ for any $e\in E(\Gamma)$.
\end{enumerate} 
\end{defn}

\begin{defn}\label{defn: nonsingular}
A graph of groups $\IG=(\Gamma, \CG)$ is called \textit{non-singular} if for all $e\in E(\Gamma)$ such that $o(e)$ has valence 1, the monomorphism $\alpha_{e}$ is not surjective. In other words, if $e$ is the only edge adjacent to the vertex $o(e)$, then one must have $[G_{o(e)}:\alpha_e(G_e)]>1$.
\end{defn}

We remark that a graph of groups is locally finite (resp. non-singular) if and only if the associated Bass-Serre tree defined in Definition \ref{rmk: construct bass serre tree} below is locally finite (resp. non-singular) as defined earlier in Definition \ref{defn: non singular and locally finite for graph}. See Remark \ref{rmK: locally finite and non singular}.

\begin{defn}\cite[Paragraph 1.5]{Ba}\label{defn: path group}
Let $\IG=(\Gamma, \CG)$ be a graph of groups. Let $F(\IG)$ be a group generated by all vertex groups $G_v$  for $v\in V(\Gamma)$, and the edge set $E(\Gamma)$ of $\Gamma$, subject to the following relations: 
$$\overline{e}=e^{-1}\ \ \ \ \mbox{and}\ \ \ \ \bar{e}\alpha_{e}(g)e=\alpha_{\bar{e}}(g) \ \ \ \ \mbox{for any}\ e\in E(\Gamma),\ g\in G_{e}$$
Precisely, $F(\IG)$ is obtained from the free product  of $G_v$ and free group with basis consisting of edges $e\in E(\Gamma)$ as follows
$$(*_{v\in V(\Gamma)}G_v)* F(E(\Gamma))$$
quotient by the normal subgroup generated by $\{e\overline{e},e\alpha_{\bar{e}}(g)\bar{e}\alpha_{e}(g)^{-1}:e\in E(\Gamma), \ g\in G_e\}$, where $F(E(\Gamma))$ is the free group with the basis $E(\Gamma)$. We call $F(\IG)$ the \textit{path group} of $\IG$.
\end{defn}

We are particularly interested in elements in $F(\IG)$ that can be described using words in the following sense.

\begin{defn}(\cite[Definition 5.1.9]{Ser} and \cite[Paragraph 1.6]{Ba})
Let $c$ be a path in $\Gamma$ with length $n\ge 0$, i.e., $c$ is a vertex $v$ or $c=(e_1,\dots,e_n)$. 
A word $w$ in $(*_{v\in V(\Gamma)}G_v)* F(E(\Gamma))$ is said to be  \textit{of type $c$} if   
\begin{enumerate}[label=(\roman*)]
    \item 
    either  $w=g\in G_v$ where $c$ is a vertex $v$, 
    \item 
    or $w$ has the form
\[w=g_1e_1g_2e_2\dots g_ne_ng_{n+1}\]
where $g_i\in G_{o(e_i)}$ for $1\leq i\leq n$ and $g_{n+1}\in G_{t(e_n)}$.
\end{enumerate} We declare the \textit{length} of $w$ to be $\ell(w)=\ell(c)$.
Similarly, we define $o(w)=o(e_1)$ and $t(w)=t(e_n)$. We denote by
\[\CW_v=\{w: w\text{ is a word of type } c \text{ such that } o(w)=v\}\]
and write $[w]$ for the element in $F(\IG)$ represented by $w$.
\end{defn}

Denote by $\pi[v, x]$ the set of all elements $\gamma$ in $F(\IG)$ such that $\gamma=[w]$   for some word $w$ of type $c$ such that $o(c)=v$ and $t(c)=x$. It is direct to verify that the multiplication in $F(\IG)$ yields a subjective map $\pi[v, x]\times \pi[x, u]\to \pi[v, u]$.

\begin{defn}\label{defn: fundamental group}
Let $\IG=(\Gamma, \CG)$ be a graph of groups, and let $v$ be a base vertex of $\Gamma$. The \textit{fundamental group of $\IG=(\Gamma,\CG)$ at $v$}, denoted by $\pi_1(\IG, v)$, is the subgroup of $F(\IG)$ consists of elements $\gamma=[w]$ where $w$ is a word of type $c$ satisfying $o(c)=t(c)=v$. In other words, $\pi_1(\IG, v)=\pi[v, v]$.
\end{defn}

It was demonstrated in \cite[Proposition 5.1.20]{Ser}
that the fundamental group $\pi_1(\IG, v)$  does not depend on the choice of the base vertex $v$.  As examples,
Fundamental groups of  one-edge graph and one-loop graph of groups as described in  (\ref{example:edge of groups}) and (\ref{example:loop of groups}) give the well-known free product amalgamation and HNN extension we shall describe below.

\begin{eg}\label{example:free amalgamation}
In Example \ref{example:edge of groups}, setting $v=o(e)$,
the presentation of the fundamental group is given by
$$\pi_1(\IG, v)=\langle G_{o(e)}, G_{t(e)}\ |\  \alpha_{e}(g)=\alpha_{\overline{e}}(g)\ \mbox{for any}\ g\in G_e\rangle,$$
which is the free amalgamated product of the two vertex groups over the edge group, i.e. $\pi_1(\IG, v)= G_{o(e)}\ast_{G_e} G_{t(e)}$
\end{eg}

\begin{defn}\label{defn: ascending HNN}
    Let $H, K\leq G$ be subgroups of a group $G$ and $\theta: H\to K$ an isomorphism. We denote by $G*_H$ the HNN extension $\hnn(G, H, \theta)=\langle G, s| s^{-1}hs= \theta(h) \mbox{ for any } h\in H\rangle$. A HNN extension $\hnn(G, H, \theta)$ is said to be \textit{ascending} if $H$ or $\theta(H)$ equals $G$.
    \end{defn}
    
\begin{eg}\label{example:HNN extension}
In Example \ref{example:loop of groups}, 
the presentation of the fundamental group is given by
$$\pi_1(\IG, v)=\langle G_v, f\ |\  \overline{f}=f^{-1},\ \bar{f}\alpha_{f}(g)f=\alpha_{\bar{f}}(g)\ \mbox{for any}\ g\in G_{f}\rangle,$$
which is the HNN extension $G_v*_{\alpha_f(G_f)}$. In this case, the related isomorphism $\theta=\alpha_{\bar{f}}\circ \alpha^{-1}_{f}$.
\end{eg}


We now give a normalized form of words that could be used to identify elements in $F(\IG)$ and in $\pi_1(\IG, v)$. To that end, we need to fix a choice of \textit{left transversal} in $G_{o(e)}$ for each  $e\in E(\Gamma)$, denoted by $\Sigma_e\subset G_{o(e)}$, that is a section for $G_{o(e)}\to G_{o(e)}/\alpha_{e}(G_e)$ such that $1_{o(e)}\in \Sigma_e$. 
\begin{defn}\label{defn: normalized word}\cite[Paragraph 1.7 and 1.12]{Ba}
 Let $\IG=(\Gamma, \CG)$ be a graph of groups.  Let $c$ be a path in $\Gamma$. A word  $w$ of type $c$ in $(*_{v\in V(\Gamma)}G_v)* F(E(\Gamma))$
 is called \textit{reduced} if it satisfies the following conditions:
\begin{enumerate}
    \item if $c$ is a vertex $v$ (i.e. $w=g$), then $w\neq 1_v$;
    \item if $c=(e_1, \dots, e_n)$ (i.e. $w=g_1e_1g_2e_2\dots g_ne_ng_{n+1}$), then one has $g_{i+1}\notin \alpha_{e_{i+1}}(G_{e_{i+1}})$ for every index $1\leq i\leq n-1$ whenever $e_{i+1}=\overline{e}_i$.
\end{enumerate}
Further, we say $w$ is \textit{normalized} if $g_i\in \Sigma_{e_{i}}$ for every $1\leq i\leq n$ and $g_{n+1}\in G_{t(e_n)}$. 
A word $w$ is called  a \textit{path word} if either $w=1_v$ or $w$ is a nomalized word such that the length $\ell(w)>0$ and $g_{n+1}=1_{t(e_n)}$ In the latter case, we also write the path word $w$ in the form $w=g_1e_1\dots g_ne_n$.  
\end{defn}

\begin{defn}\label{defn: path map}
Let $v\in V(\Gamma)$. Denote by
\[\CN_v=\{w\in \CW_v: w\text{ is a normalized word}\}\]
and
\[\CP_v=\{w\in \CN_v: w\text{ is a path word}\}\cup \{1_v\}.\]
We define a map $P: \CN_v\to \CP_v$ as follows.
 Let $w=g_1e_1\dots g_ne_ng_{n+1}\in \CN_v$. Define $P(w)=g_1e_1\dots g_ne_n=wg^{-1}_{n+1}$, which is exactly the path word by canceling the final $g_{n+1}\in G_{t(c)}$  from $w$. 
\end{defn}

\begin{rmk}[normalization process]\label{rmk: transfer to reduced and normalized word}
    Let $w=g_1e_1\dots g_ne_ng$ be a word of type $c=(e_1, \dots, e_n)$. A subword of the form $e_ig_{i+1}e_{i+1}$ is said to be a \textit{reversal} in $w$ if $e_{i+1}=\bar{e}_i$ and $g_{i+1}=\alpha_{e_{i+1}}(h)$ for some $h\in G_{e_{i+1}}$  (\cite[Paragraph 1.7]{Ba}). By definition, $w$ is not reduced 
 if and only if $w$ contains a reversal. Indeed, if $w$ contains a reversal $e_ig_{i+1}e_{i+1}$, then  the word 
 \[w_1=g_1e_1\dots e_{i-1}(g_i\alpha_{\bar{e}_{i+1}}(h)g_{i+2})e_{i+2}\dots g_ne_ng\]
 represents the same element with $w$ in $F(\IG)$ and $\ell(w_1)=\ell(w)-2$. Proceeding this process by induction, one obtains a sequence $w=w_0, w_1, \dots, w_m$ such that there is no reversal in $w_m$, which is reduced. 

 Furthermore, if   $w=g_1e_1\dots g_ne_ng$ is a reduced word, one may convert it to a normalized word (with a fixed choice of left transversal). Write $g_1=s_1\cdot \alpha_{e_1}(h_1)$ for some $s_1\in \Sigma_{e_1}$ and $h_1\in G_{e_1}$ and one has $w=s_1e_1(\alpha_{\bar{e}_1}(h_1))g_2e_2\dots g_ne_ng$. Inductively, there exists $s_i\in \Sigma_{e_i}$ for $1\leq i\leq n$ and a $h\in G_{e_n}=G_{\bar{e}_n}$ such that 
\[w=s_1e_1\dots s_ne_n(\alpha_{\bar{e}_n}(h)g).\]
is a normalized word. This process thus leads to
a well-defined map $N: \CW_v\to \CN_v$ for each $v\in V(\Gamma)$ such that $N(w)$ is the normalized word obtained from the normalization process above.
\end{rmk}



\begin{rmk}\label{rmk: minimum length reduced words}

By \cite[Corollary 1.10]{Ba}, one may identify any $\gamma\in \pi_1(\IG, v)$ by its unique normalized representative $w$. 
In addition, according to Remark \ref{rmk: transfer to reduced and normalized word} and \cite[Corollary 1.10]{Ba},  a word $w$ is reduced if and only if $\ell(w)$ is the minimum among all $\ell(w')$ such that $w'$ representing the same element with $w$ in $F(\IG)$.  
\end{rmk}



We now are going to recall the so-called Fundamental Theorem of Bass-Serre theory. Roughly speaking,  it gives a duality between graphs of groups and group actions on trees. We first explain how to get a group action on the so-called Bass-Serre tree $X_\IG$ starting from  a graph of groups $\IG=(\Gamma, \CG)$. Let $G=\pi_1(\IG, v)$ for a base vertex $v$. 
The explicit description of the tree $X_\IG$ using path words from the base vertex $v$ are recorded in \cite[Theorem 1.17]{Ba} and also in \cite[Definition 2.13]{B-M-P-S-T}. 

\begin{defn}\cite[Definition 2.13]{B-M-P-S-T}\label{rmk: construct bass serre tree}
The \textit{Bass-Serre tree} $X_\IG$ for $\IG=(\Gamma, \CG)$ is defined as follows. The vertex set $V(X_\IG)$ is given by 
    \[V(X_\IG)=\bigsqcup_{u\in V(\Gamma)}\pi[v, u]/G_u=\{\gamma G_u: \gamma\in \pi[v, u], u\in V(\Gamma)\}.\]
    Let $\gamma_1=[w_1]\in \pi[v, u_1]$ and $\gamma_2=[w_2]\in \pi[v, u_2]$. Then there exists an edge $f\in E(X_\IG)$ between   $o(f)=\gamma_1G_{u_1}$ and $t(f)=\gamma_2G_{u_2}$ if and only if $N(w_1^{-1}w_2)\in G_{u_1}eG_{u_2}$, where $e\in E(\Gamma)$ with $o(e)=u_1$ and $t(e)=u_2$. 
    
 The action of $\pi_1(\IG, v)=\pi[v, v]$ on $X_\IG$ is given on vertices by $s\cdot \gamma G_u=s\gamma G_u$ for any $s\in \pi_1(\IG, v)$ and $\gamma G_u\in V(X_\IG)$. Then it is direct to see such actions extend to edge set $E(X_\IG)$.
\end{defn}

\begin{rmk}\label{rmK: locally finite and non singular}
By Definition \ref{rmk: construct bass serre tree}, it is easy to read off the graph structure of Bass-Serre tree from the graph of groups. 
\begin{enumerate}[label=(\roman*)]
    \item For a vertex $\gamma G_u$ in $X_\CG$, it is direct to verify its valence 
    \[\mathrm{val}(\gamma G_u)=\sum_{o(e)=t(\gamma)=u}[G_{o(e)}: \alpha_e(G_e)].\]
    Thus, a graph of groups $\IG$ is locally finite if and only if its Bass-Serre tree $X_\IG$ is locally finite in the sense of Definition \ref{defn: non singular and locally finite for graph}. 
    \item From the equation about the valence of vertices on $X_\IG$ above, it is also straightforward to see that $\IG$ is non-singular if and only if $X_\IG$ is non-singular in the sense of Definition \ref{defn: non singular and locally finite for graph}. 
    \item\label{path vertices} Let $\gamma G_u\in V(X_\IG)$ such that $\gamma=[w]$. Then $\gamma G_u$ can also be written as $[N(w)]G_u$ or $[P(N(w))]G_u$, where the maps $N$ and $P$ are defined in Remark \ref{rmk: transfer to reduced and normalized word} and Definition \ref{defn: path group}.
    \end{enumerate}
Moreover, it was shown in \cite[Paragraph 1.22]{Ba} that there exists an equivariant isomorphism between $\pi_1(\IG, v_1)\curvearrowright X_\IG$ and $\pi_1(\IG, v_2)\curvearrowright X_\IG$ for different base vertices $v_1\neq v_2\in V(\Gamma)$.    
\end{rmk}


On the other hand, there is a canonical way to build a graph of groups from an isometric action on trees without reversions. Such a process can be found in \cite[Section 5.4]{Ser} and the paragraph before \cite[Theorem 2.16]{B-M-P-S-T}. In conclusion, one has the following theorem stating the one-to-one correspondence between group acting on trees without inversions and the construction of graph of groups. See also \cite[Theorem 2.16]{B-M-P-S-T}.



\begin{thm}\cite[Section 5.4]{Ser}\label{thm fundamental theorem in Bass serre}
Let $G$ be a group acting on a tree $X$ without inversions. Then $G$ can be identified with the fundamental group of a certain graph of groups $\IG=(\Gamma, \CG)$, where $\Gamma$ is the quotient graph $X/G$. Moreover, there is an equivariant isomorphism of the Bass-Serre tree $X_\IG$ to $X$.
\end{thm}

A particular family of graph of groups of interests are \textit{reduced} graph of groups. Every graph of groups can be reduced to a reduced one without changing the fundamental group. Therefore, in many cases, it suffices to look at reduced graph of groups to investigate fundamental groups. See Remark \ref{rmk: reduced tree} below.  The following definition can be found in, e.g., \cite{M06}.
\begin{defn}\cite[Definition 1.2]{M06}\label{defn: reduced graph}
Let $\IG=(\Gamma, \CG)$ be a graph of groups. An edge $e\in E(\Gamma)$ is \textit{collapsible} if $o(e)\neq t(e)$ and $G_{t(e)}=\alpha_{\overline{e}}(G_e)$. If one collapses $\{e,\overline{e}\}$ to the vertex $o(e)$, the resulted graph of groups $\IG'=(\Gamma', \CG)$ is said to be obtained from $\IG=(\Gamma,G)$ by a \textit{collapse move}. The reverse of this move is called an \textit{expansion move}. A graph of groups is called \textit{reduced} if there is no collapsible edge.
\end{defn}

\begin{rmk}\label{rmk: reduced advantage}
From the above definition, we can see that a reduced graph of groups is always non-singular and satisfies $|\Sigma_{e}|\geq 2$ for any edge $e$ that is not a loop. But there exists non-singular graph of groups that is not reduced. See the following example in which the monomorohism $\alpha_e$ is surjective while $\alpha_{\bar{e}}$, $\alpha_{f}$ and $\alpha_{\bar{f}}$ are not surjective.
\end{rmk}

\begin{figure}[ht]
    \centering
\begin{tikzpicture}
      \draw (0,2) -- (4,2) node[midway, sloped, above] {};
      \draw[<-, thick] (0.97,2) -- (0.99,2);
      \draw[->, thick] (3.07,2) -- (3.09,2);

      \node[label={below, yshift=-0.3cm:}] at (0,2.8) {};
      \node[label={below, yshift=-0.3cm:}] at (2,2.4) {$v_2$};
      \node[label={below, yshift=-0.3cm:}] at (1,1.6) {$e$};
      \node[label={below, yshift=-0.3cm:}] at (0,2.4) {$v_1$};
      \node[label={below, yshift=-0.3cm:}] at (4,2.4) {$v_3$};
      \node[label={below, yshift=-0.3cm:}] at (3,1.6) {$f$};
      
      \tikzset{enclosed/.style={draw, circle, inner sep=0pt, minimum size=.1cm, fill=black}}
      
      \node[enclosed, label={right, yshift=.2cm:}] at (0,2) {};
      \node[enclosed, label={right, yshift=.2cm:}] at (2,2) {};
      \node[enclosed, label={right, yshift=.2cm:}] at (4,2) {};
\end{tikzpicture}
    \caption{A non-singular but not reduced graph of groups}
    \label{fig: non-singular not reduced}
\end{figure}
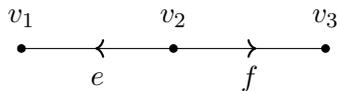

The following is a standard  well-known observation.

\begin{rmk}\label{rmk: reduced tree}
We note that if an edge $e$ (not a loop) has $G_{t(e)}=\alpha_{\bar{e}}(G_e)$, then $e$ contributes to a free amalgamated product $G_{o(e)}\ast_{G_e} G_{e}$ in the fundamental group of $\IG$. Note that $G_{o(e)}\ast_{G_e} G_{e}\cong G_{o(e)}$ holds if $e$ is collapsible. Thus collapsing such an edge in $\Gamma$ will not change the fundamental group. See Figure \ref{fig:collapse} below. In other words, every graph of groups $\IG$ can be collapsed into a reduced graph of groups $\IG'$ such that their fundamental groups satisfy  $\pi_1(\IG, v)=\pi_1(\IG', v)$.
\end{rmk}

\begin{figure}[ht]
    \centering
\begin{tikzpicture}
      \tikzset{enclosed/.style={draw, circle, inner sep=0pt, minimum size=.1cm, fill=black}}
      
      \node[enclosed, label={right, yshift=.2cm: \small $A$}] at (2.25,3.25) {};
      \node[enclosed, label={left, yshift=.2cm: \small$C$}] at (4.5,3.25) {};
      \node[label={below, yshift=-0.3cm:}] at (3.5,3) {\small$C$}; 
      \node[enclosed] at (11,3.25) {};
       \node[label={below, yshift=-0.3cm:}] at (11,3.6) {\small$A$};

      \draw[line width=0.5mm] (2.25,3.25) -- (4.5,3.25) node[midway, sloped, above] {};
      \draw (2.25,3.25) -- (1.75, 3.75) node[midway, right] {};
      \draw (2.25,3.25) -- (1.75, 2.75) node[midway, right] {};
      \draw (4.5,3.25) -- (5, 3.75) node[midway, right] {};
      \draw (4.5,3.25) -- (5, 2.75) node[midway, right] {};

      \draw[->, thick] (6.5, 3.3) -- (9, 3.3) node[midway, above] {\footnotesize collapse};
      \draw[<-, thick] (6.5, 3.1) -- (9, 3.1) node[midway, below] {\footnotesize expansion};
      
      \draw (11,3.25) -- (10.5, 3.75) node[midway, left] {};
      \draw (11,3.25) -- (10.5, 2.75) node[midway, left] {}; 
      \draw (11,3.25) -- (11.5, 3.75) node[midway, right]{};
      \draw (11,3.25) -- (11.5, 2.75) node[midway, right]{};
\end{tikzpicture}
    \caption{Example of collapse and expansion moves.}
    \label{fig:collapse}
\end{figure}
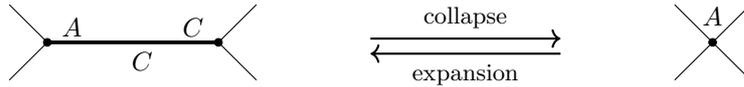

\subsection{Boundary actions} 
\label{SSecTDS}
In this subsection, we recall basic backgrounds on topological dynamics and boundary actions in the sense of Furstenburg.
Let $G$ be a countable discrete group and $Z$ a compact Hausdorff space. Let $\alpha: G\curvearrowright Z$ be an action of $G$ on $Z$ by homeomorphisms.  

We say an action $\alpha: G\curvearrowright Z$ is \textit{minimal} if all orbits are dense in $Z$. An action is said to be \textit{topologically free} provided that the set $\{z\in Z: \stab_G(z)=\{e\}\}$, is dense in $Z$. Equivalently, this amounts to saying that the fixed point set $\fix_Z(g)=\{z\in Z: gz=z\}$ is nowhere dense for any $g\in G\setminus\{1_G\}$. Since we will also discuss actions on non-compact spaces in this paper, we still adopt the definitions for minimality and topological freeness in the context that the underlying space $Z$ is not compact. 

A type of topological dynamical system of particular interest is $G$-boundary actions in the sense of Furstenburg. Now, denote by $P(Z)$ the set of all probability measures on $Z$. Furstenberg introduced the following definition in \cite{F}.
\begin{defn}\label{defn: boundary}
	\begin{enumerate}[label=(\roman*)]
		\item A $G$-action $\alpha$ on $Z$ is called \textit{strongly proximal} if for any probability measure $\eta\in P(Z)$, the closure of the orbit $\{g\eta: g\in G\}$ contains a Dirac mass $\delta_z$ for some $z\in Z$ 
		\item A $G$-action $\alpha$ on a compact Hausdorff space $Z$ is called a $G$-\textit{boundary action} if $\alpha$ is minimal and strongly proximal.
	\end{enumerate}
\end{defn}

The following definition appeared in \cite{L-S}.  See also \cite{Gla2}.

\begin{defn}\cite[Definition 1]{L-S}\label{defn: strong boundary}
We say an action $G\curvearrowright Z$ is a \textit{strong boundary action} (or \textit{extreme proximal}) if for any compact set $F\neq Z$ and non-empty open set $O$ there is a $g\in G$ such that $gF\subset O$.  
\end{defn}

\begin{rmk}\label{rmk: strong bdry is bdry}
Let $G\curvearrowright Z$ be a strong boundary action. Then in \cite{Gla1}, Glasner showed that $Z$ is a $G$-boundary in the sense of Definition \ref{defn: boundary} unless $Z$ contains exactly two points.
\end{rmk}

The following definition is widely used in the study of various boundaries by Proposition \ref{prop: north-south visual}. 

\begin{defn}\label{defn: ns dyna}
	Let $Z$ be a topological space and $g$ a homeomorphism of $Z$. We say $g$ has \textit{north-south dynamics} with respect to two fixed points $x, y\in Z$ if for any open neighborhoods $U$ of $x$ and $V$ of $y$, there is an $m\in \N$ such that $g^m(Z\setminus V)\subset U$ and $g^{-m}(Z\setminus U)\subset V$. The points $x,y$ shall be referred to as  \textit{attracting} and \textit{repelling} fixed points of $g$, respectively.
\end{defn}

A similar argument for the following proposition also appeared in, e.g., \cite{A-D} and \cite{L-S}. 
\begin{prop}\label{prop: north-south visual}
	Let $\alpha:G\curvearrowright Z$ be a continuous minimal action of a discrete group $G$ on an infinite compact Hausdorff space $Z$. Suppose there is a $g\in G$  performing the north-south dynamics. Then $\alpha$ is a strong boundary action and thus $\alpha$ is a $G$-boundary action by Remark \ref{rmk: strong bdry is bdry}.
\end{prop}
\begin{proof}
	Let $F$ be a proper compact set in $Z$ and $O$ a non-empty open set in $Z$.
 Write $O_1=Z\setminus F$, and $O_2=O$, which are non-empty open sets in $Z$. Suppose $x, y$ are attracting and repelling fixed points of $g$, respectively. First, by minimality of the action, one can find two open neighborhoods $U, V$ of $x, y$, respectively, small enough such that  there are $g_1, g_2\in G$ such that $g_1V\subset O_1$ and $g_2U\subset O_2$.
	Now our assumption on $g$ implies that there is an $m\in \N$ such that $g^m(Z\setminus V)\subset U$, which implies $g_2g^m(Z\setminus V)\subset O_2$. Then one observes that $Z=(g_2g^m)^{-1}O_2\cup g_1^{-1}O_1$. Write $h=g_2g^mg_1^{-1}$ for simplicity. This shows that $hF=h(Z\setminus O_1)\subset O$. Therefore $\alpha$ is a strong boundary action and thus is a $G$-boundary action.
\end{proof}

The study of boundary actions has a strong connection with $C^*$-simplicity of discrete groups and pure infiniteness of crossed product $C^*$-algebras as we will explain next.

\subsection{$C^*$-algebras of groups and dynamical systems}
Let $G$ be a countable discrete group, we say $G$ is $C^*$-\textit{simple} if its reduced $C^*$-algebra $C^*_r(G)$ is simple in the sense that it has no proper closed ideal.  On the other hand, for any action $G\curvearrowright Z$ on a compact Hausdorff space $Z$, one may define a (reduced) crossed product $C(Z)\rtimes_r G$. We refer to \cite{B-O} for the detailed definitions for the $C^*$-algebras $C_r^*(G)$ and $C(Z)\rtimes_r G$. 
So far, we have characterizations of  $C^*$-simplicity of a group $G$ using $G$-boundaries. See, e.g., \cite{B-K-K-O} and \cite{K-K}. The following theorem will be used in this paper.

\begin{thm}\cite[Theorem 6.2]{K-K}\label{thm: C simple equivalence}
Let $G$ be a discrete group. Denote by $\partial_F G$ the \textit{Furstenberg boundary} for $G$, which is the unique universal $G$-boundary in the sense that every $G$-boundary is a $G$-equivariant continuous image of $\partial_F G$. The following are equivalent.
\begin{enumerate}[label=(\roman*)]
    \item $G$ is $C^*$-simple.
    \item The reduced crossed product $C(\partial_F G)\rtimes G$ is simple.
    \item The reduced crossed product $C(B)\rtimes G$ is simple for any $G$-boundary $B$.
    \item The action of $G\curvearrowright \partial_F G$ is topologically free.
    \item The action of $G$ on some $G$-boundary $B$ is topologically free.
\end{enumerate}
\end{thm}

We now turn to characterize structure properties of crossed product $C(Z)\rtimes_r G$ by dynamical properties. The following result is fundamental in this direction.

\begin{rmk}\label{rmk: equivalence of dynamical properties}
It is well known that if the action $G\curvearrowright Z$ is topologically free and minimal then the reduced crossed product $C(Z)\rtimes_r G$ is simple (see \cite{A-S}) and it is also known that the crossed product $C(Z)\rtimes_r G$ is nuclear if and only if the action $G\curvearrowright Z$ is (topological) amenable (see \cite{B-O}). In this setting, Archbold and Spielberg \cite{A-S} further showed $C(Z)\rtimes_r G$ is simple and nuclear if and only if the action is minimal, topologically free, and (topologically) amenable.
\end{rmk}

\begin{rmk}\label{rmk: equivalence of C simple and topo free}
Let $\alpha: G\curvearrowright Z$ be a $G$-boundary action. Suppose $\alpha$ is (topologically) amenable.  Then Theorem \ref{thm: C simple equivalence} and Remark \ref{rmk: equivalence of dynamical properties} imply that $G$ is $C^*$-simple if and only if $\alpha$ is topologically free. The only non-trivial part is the direction of ``only if''. If $G$ is $C^*$-simple, then $C(Z)\rtimes_r G$ is simple by Theorem \ref{thm: C simple equivalence}(iii) and nuclear since $\alpha$ is also amenable. Then Remark \ref{rmk: equivalence of dynamical properties} implies that $\alpha$ is topologically free.
\end{rmk}

The notion of pure infiniteness for $C(Z)\rtimes_r G$ is an important regularity property for $C^*$-algebras, which we shall describe now.  Let $A$ be a unital $C^*$-algebra. We denote by $\precsim$ the \textit{Cuntz subequivalence relation} on $M_\infty(A)\coloneqq \bigcup_{n=1}^\infty M_n(A)$ defined by saying $a\precsim b$ if there exists a sequence of matrices $x_n$ with proper dimensions such that $x^*_nbx_n\to a$ as $n\to \infty$. we refer to \cite{A-P-T} as a standard reference for this topic.
A non-zero positive element $a$ in $A$ is said to be \textit{properly infinite} if
$a\oplus a\precsim a$, where $a\oplus a$ denotes the diagonal matrix with entry $a$ in $M_2(A)$. A $C^\ast$-algebra $A$ is said to be \textit{purely infinite} if there are no characters
on $A$ and if, for every pair of positive elements $a,b\in A$ such that $b$ belongs to the closed ideal in $A$ generated by $a$, one has $b\precsim a$. A very useful characterization of pure infiniteness was provided in \cite{K-Rord} that a $C^\ast$-algebra $A$ is purely infinite if and
only if every non-zero positive element $a$ in $A$ is properly infinite.
Recall that a $C^*$-algebra $A$ is called a \textit{Kirchberg algebra} if it is simple nuclear separable and purely infinite.  The pure infiniteness plays an important role in the Kirchberg-Philips classification theorem of Kirchberg algebras satisfying the UCT. (see, e.g., \cite{Kir} and \cite{Phillips}). \textit{Uniform coefficient theorem}, written as the UCT for simplicity, is a technical condition posed on nuclear $C^*$-algebras to be classified. By the classical result of \cite{Tu} and Remark \ref{rmk: equivalence of dynamical properties}, a crossed product $C(Z)\rtimes_r G$ satisfies the UCT if the action is (topological) amenable.
In general, The following was established in \cite{L-S}. See also \cite{A-D}.

\begin{thm}\cite[Theorem 5]{L-S}\label{thm: strong bdry action pure infinite}
    Let $G\curvearrowright Z$ be a topological free strong boundary action. Then $C^*$-algebra $A=C(Z)\rtimes_r G$ is purely infinite and simple. In addition, if $Z$ is metrizable, then $A$ is separable.
\end{thm}
We conclude this section with the following remark.

\begin{rmk}\label{rmk: n-filling}
A topological action $G\curvearrowright Z$ on a Hausdorff space $Z$ is said to be $n$-filling if for any $n$ non-empty open sets $O_1,\dots, O_n$, there exists $g_1\dots, g_n\in G$ such that $\bigcup_{i=1}^ng_iO_i=Z$. It was shown in \cite{J-R} that if $Z$ is also compact, then the action $G\curvearrowright Z$ is $2$-filling if and only if it is a strong boundary action.
We remark that Theorem \ref{thm: strong bdry action pure infinite} has been generalized to $n$-filling actions in \cite{J-R} and minimal actions having dynamical comparison in \cite{M2}. 
\end{rmk}

\section{Boundary actions on Bass-Serre trees}\label{sec: main}
This section studies the general theory of the action of the fundamental group of a graph of groups on the boundary of the Bass-Serre tree. We begin by recalling the actions on trees by automorphisms, which can be found in e.g., \cite[Section 3]{H-P}, and then focus on the action on the boundary of the Bass-Serre tree. 

Recall that this paper only considers \textbf{countable} trees (i.e. with countably many vertices and edges). 

\subsection{Group actions on trees and their boundaries}
Let $X$ be a  (possibly locally infinite) simplicial tree, equipped with the path metric $d$. A finite sequence $\mu=(x_1,\dots,x_n)$ of vertices in $X$ is said to be a \textit{geodesic segment} if $d(x_i, x_j)=|i-j|$ for any $1\leq i, j\leq n$. An infinite sequence $(x_n)_{n\in \N}$ of vertices on $X$ is said to be a \textit{geodesic ray} if $d(x_n, x_m)=|m-n|$ for any $m, n\in \N$. A bi-infinite sequence $(x_n)_{n\in \Z}$ is said to be \textit{geodesic line} if $d(x_n, x_m)=|m-n|$ for any $m, n\in \Z$.
From \cite{Tits}, an automorphism $s\in \aut(X)$ can be of three different types:
\begin{enumerate}
    \item[$\circ$] It is an \textit{inversion}, i.e. $s(e)=\bar{e}$ for some $e\in E(X)$.
    \item[$\circ$] It is \textit{elliptic} if it fixes a vertex $v\in V(X)$.
    \item[$\circ$] It is \textit{hyperbolic} if otherwise.
\end{enumerate}
A hyperbolic automorphism $s\in \aut(X)$   has a unique  \textit{axis}, which is a geodesic line $T_s=(x_n)_{n\in \Z}$ satisfying  that  $d(x, s(x))=\min\{d(z, s(z)): z\in V(X)\}$ for any $x\in T_s$.

Two geodesic rays $(x_n)$ and $(y_n)$ are said to be \textit{cofinal} if they coincide eventually, i.e. there exists $m_0, n_0\geq 0$ such that $x(m_0+n)=y(n_0+n)$ holds for any $n\geq 0$. It is direct to verify that being cofinal is an equivalence relation. Then the \textit{boundary} of the tree $X$ is defined as the set of all such equivalence classes, usually denoted by $\partial X$.   
Fixing a base vertex $x_0$ of $X$, it is a standard fact that each equivalence class contains exactly one geodesic ray starting from $x_0$. Therefore, the boundary $\partial X$ may be identified with the set of all geodesic rays starting from $x_0$. We warn that, unlike some literature, we follow \cite{H-P} that the notation $\partial X$ is just a set without being equipped with any topology. We will describe two possible topologies on the set $\partial X$.

First, we equip $\partial X$ with the inverse limit topology by identifying $\partial X$ as an inverse limit space of finite sets $S_n=\{x\in X: d(x, x_0)=n\}$. To be more specific, the neighborhood system for $X$ is given by all
\[U(\mu)=\{\xi\in \partial X: \mu \text{ is an initial segment of }\xi\}\]
for all geodesic segment $\mu$. This topology is also called \textit{cone topology} (see Remark \ref{rmk: action on bdry of tree} below). Equip $\partial X$ with this topology, we denote by $\partial_\infty X$, the boundary of $X$. We remark that $\partial_\infty X$ is compact if and only if $X$ is locally finite. Moreover, if $X$ is locally finite, the $X\sqcup \partial_\infty X$ forms a compactification of $X$. We denote by $\tau_c$ for this cone topology.

Second, no matter whether $X$ is locally finite or not, there is always a way to define a topology $\tau_s$ on $X\sqcup \partial X$, called \textit{shadow topology} (see \cite[Section 4]{M-S}) such that $(X\sqcup \partial X, \tau_s)$ is a compact metrizable space, which is a compactification of $X$. There is a very nice connection between these two topologies.

\begin{prop}\cite[Proposition 4.4]{M-S}\label{prop: topologies coincide}
    The topologies $\tau_c$ and $\tau_s$ coincide on $X\sqcup \partial X$ if and only if $X$ is locally finite. In general, they always coincide by restricting on $\partial X$.
\end{prop}
Therefore, we may keep the notation $\partial_\infty X$ for the boundary of $X$ no matter which topology we put on it.
Now, let $G$ be a countable discrete group and $G\curvearrowright X$ by automorphisms. Note that this action is an isometric action with respect to the path metric $d$. The action $G\curvearrowright X$ is said to be \textit{minimal} if there exists no proper $G$-invariant subtree $X'\subset X$. The isometric action $G\curvearrowright X$ can be extended to a continuous action $G\curvearrowright (X\sqcup \partial X, \tau_s)$. Note that $\partial_\infty X$ is a $G$-invariant subset and the restricted action $G\curvearrowright \partial_\infty X$ can be described explicitly as follows.

\begin{rmk}\label{rmk: action on bdry of tree}
    We remark that two geodesic rays $\xi, \eta$ are cofinal in the tree $X$ if and only if $\xi$ and $\eta$ are \textit{asymptotic} in the sense of \cite[Definition II. 8.1]{B-H} by viewing $X$ as a complete $\cat$ space. 
        Therefore, the boundary $\partial_\infty X$ is exactly the visual boundary for $X$ as a complete $\cat$ space (see \cite[Section II. 8]{B-H}). It is well-known that $\partial_\infty X$ is metrizable totally disconnected space as $X$ is countable. The sets $U(\mu)$ defined above are clopen. 
Moreover, $\partial_\infty X$ is compact if and only if $X$ is proper, i.e., $X$ is locally finite. See, e.g, \cite[Section 3]{H-P}. 
Let $G$ act on $X$ by isometry with respect to the path metric. Then such an action induces a continuous action of $G$ on $\partial_\infty X$ by $g\cdot [(x_n)]=[(g\cdot x_n)]$ by \cite[Corollary 8.9]{B-H}, where $[(x_n)]$ is the equivalence class of being cofinal that contains the geodesic $(x_n)$. Then if we characterize the $\partial_\infty X$ by all geodesic rays emitting from one vertex $x$. Then the action is explicitly given by $h \cdot (x_n)=(y_n)$ such that $y_1=x_1=x$ and $(y_n)$ is cofinal with $(g\cdot x_n)$. The underlying topology for this action is thus exactly the topology generated by all $U(\mu)$ above.
\end{rmk}

\begin{rmk}
    Let $s\in \aut(X)$ be a hyperbolic automorphism of the tree $X$. Then its axis $T_s=(x_n)_{n\in \Z}$ yields two fixed points $s^\infty=[(x_n)_{n\in \N}]$ and $s^{-\infty}=[(x_n)_{n\in -\N}]$ in $\partial_\infty X$ about which, the automorphism $s$ has the north-south dynamics in the sense of Definition \ref{defn: ns dyna}. Note that admitting north-south dynamics is actually a characterization of the hyperbolicity of $s$. See, e.g., \cite[Section 3, pp. 461]{H-P}. We also remark that $s^{\infty}$ and $s^{-\infty}$ are only fixed points for $s$ on $\partial_\infty X$ (see, e.g., \cite[Remarks 12(i)]{H-P}).
    \end{rmk}

    We say an action $G\curvearrowright X$ is \textit{strongly hyperbolic} if it contains two hyperbolic elements $h, g$ such that their fixed points set $\{h^\infty, h^{-\infty}\}$ and $\{g^\infty, g^{-\infty}\}$ are disjoint. We now turn to the freeness of the boundary actions. The following is a general observation.

\begin{prop}\label{prop: topo free permanant}
    Let $\alpha: G\curvearrowright Z$ be a continuous action on a compact Hausdorff space $Z$ by a countable discrete group $G$. Let $A\subset Z$ be a $G$-invariant dense subset. Then the restricted action $\alpha: G\curvearrowright A$ is topologically free if and only if $\alpha: G\curvearrowright Z$ is topologically free.
    \end{prop}
    \begin{proof}    
         Suppose the restricted action $\alpha: G\curvearrowright A$ is topologically free. Denote by $\fix_Z(g)=\{z\in Z: gz=z\}$ the fixed point set for $g\in G$.  Note the topological freeness for $\alpha: G\curvearrowright Z$ is equivalent to that $Z\setminus\fix_Z(g)$ is dense in $Z$ for any $g\in G\setminus\{1_G\}$. Let $g$ be a non-trivial element in $G$. We denote by $\fix_A(g)=\{z\in A: gz=z\}$ the fixed point set for $g$ in A, which satisfies $A\setminus \fix_A(g)\subset Z\setminus\fix_Z(g)$. Now because the restricted action $G\curvearrowright A$ is assumed to be topologically free, one has  $A\setminus \fix_A(g)$ is dense in $A$. Now let $O$ be open in $Z$. Since $A$ is dense in $Z$, the set $A\cap O$ is a non-empty open set in $A$, which implies that $A\cap O\cap (A\setminus \fix_A(g))\neq \emptyset$ and therefore $O\cap (Z\setminus \fix_Z(g))\neq \emptyset$ holds. This shows that $Z\setminus \fix_Z(g)$ is dense.

         For the converse, suppose $\alpha: G\curvearrowright Z$ is topologically free. Then $Z\setminus \fix_Z(g)$ is an open dense set in $Z$ for any $g\in G\setminus \{1_G\}$. Then observe $A\setminus \fix_A(g)=A\cap (Z\setminus \fix_Z(g))$. Now, let $O$ be a non-empty open set in $Z$ and one has 
         \[(A\setminus \fix_A(g))\cap O\cap A=A\cap O\cap  (Z\setminus \fix_Z(g))\neq \emptyset\]
         because $A$ is dense and $O\cap (Z\setminus \fix_Z(g))$ is a non-empty open set in $Z$.
          \end{proof}
This leads to the following result. 

\begin{cor}\label{cor: bdry topo free}
   Let $X$ be a tree. In the space $(X\sqcup \partial X, \tau_s)$, the restricted action $\alpha: G\curvearrowright \partial_\infty X$ is topologically free if and only if $\alpha: G\curvearrowright \overline{\partial_\infty X}$ is topologically free.
\end{cor}

In the case that the action $G\curvearrowright X$ contains hyperbolic elements, we have similar results for minimality as recorded in \cite{H-P}. 
\begin{prop}\cite[Remark, pp. 462]{H-P}\label{prop: minimal 1}
    Let $G\curvearrowright X$ be an action on a tree $X$ by automorphism and $X$ is non-singular in the sense of Definition \ref{defn: non singular and locally finite for graph}. Suppose the continuous action $G\curvearrowright \partial_\infty X$ is minimal and $G$ contains a hyperbolic element. Then the action $G\curvearrowright X$ by automorphism on the tree is minimal. 
\end{prop}

\begin{prop}\cite[Proposition 14]{H-P}\label{prop: minimal 2}
    Let $G\curvearrowright X$ be a minimal action by automorphism such that $X$ is not isomorphic to a line. Suppose there are no fixed points on $\partial_\infty X$ by $G$. Then $G\curvearrowright X$ is strongly hyperbolic.
\end{prop}

The following proposition was shown in \cite{L-M} and \cite{B-I-O} and we record the version in \cite{B-I-O}.

\begin{prop}\cite[Lemma 3.5]{B-I-O}\label{prop: minimal 3}
    Let $G\curvearrowright X$ be a minimal strongly hyperbolic action. Then the induced topological action $G\curvearrowright \overline{\partial_\infty X}$ is a strong boundary action and thus minimal.
\end{prop}

As a combination of these three propositions, one has the following result. 

\begin{prop}\label{prop: minimal 3 and strong bdry}
    Let $G\curvearrowright X$ be an action on a non-singular tree $X$ by automorphisms. Suppose $X$ is not isomorphic to a line. Consider the following statements.
    \begin{enumerate}[label=(\roman*)]
    \item $G\curvearrowright X$ by automorphism is minimal. 
    \item The continuous action $G\curvearrowright \partial_\infty X$ is minimal.
    \item The continuous action $G\curvearrowright \overline{\partial_\infty X}$ is minimal.
\end{enumerate}
If $G$ contains a hyperbolic element, then (ii)$\Leftrightarrow$(iii)$\Rightarrow$(i). If the action $G\curvearrowright X$ is strongly hyperbolic, then these three conditions are equivalent.
    \end{prop}
\begin{proof}
    Suppose there exists a hyperbolic automorphism $s\in G$. If (ii) holds, i.e., $G\curvearrowright \partial_\infty X$ is minimal, then Proposition \ref{prop: minimal 1} implies that $G\curvearrowright X$ is minimal. This has established (i).
    The minimality of $G\curvearrowright \partial_\infty X$ implies that there is no global fixed point on $\partial_\infty X$ for $G$. Then Proposition \ref{prop: minimal 2} shows that $G\curvearrowright X$ is in fact strongly hyperbolic. Then Proposition \ref{prop: minimal 3} shows that (iii) holds. 
On the other hand, (iii)$\Rightarrow$(ii) is trivial because $\partial_\infty X$ is a $G$-invariant subset of $\overline{\partial_\infty X}$.

Finally, suppose $G\curvearrowright X$ is strongly hyperbolic. Then Proposition \ref{prop: minimal 3} shows that (i)$\Rightarrow$(iii). Therefore, these three conditions are equivalent.
\end{proof}
 We remark that in general, (i) and (ii) in Proposition \ref{prop: minimal 3 and strong bdry} are not equivalent. See an example in Remark \ref{rmk: ascending loop non minimal}.

\subsection{Describing boundary of the Bass-Serre tree}\label{subsection: visual bdry of BS tree}
We now turn to the boundary of Bass-Serre trees.  
In this paper, $\IG=(\Gamma, \CG)$ always denotes a \textbf{countable} graph of groups. Then the Bass-Serre tree $X_\IG$ associated with $\IG$ is also countable. From now on, for each $e\in E(\Gamma)$, we always fix a left coset representative set $\Sigma_e\subset G_{o(e)}$ for $G_{o(e)}/\alpha_{e}(G_e)$ such that $1_{o(e)}\in \Sigma_e$. We begin with the following observation.

\begin{defn}\label{defn: non-degenerated}
    Let $\IG=(\Gamma, \CG)$ be a graph of groups. An edge $e\in E(\Gamma)$ is said to be \textit{non-degenerated} if $o(e)\neq t(e)$ and satisfies $|\Sigma_e|\geq 3$ and $|\Sigma_{\bar{e}}|\geq 2$. The amalgamated free product $G_{o(e)}*_{G_e} G_{t(e)}$ obtains from a non-degenerated edge (see Example \ref{example:free amalgamation}) is also said to be non-degenerated. 
    \end{defn}

\begin{rmk}\label{rmk: infinite bdry}
  Let $\IG=(\Gamma, \CG)$ be a graph of groups. A useful fact for the Bass-Serre tree $X_\IG$ is that in either of the following cases, the boundary $\partial_\infty X_\IG$ is infinite.
     \begin{enumerate}[label=(\roman*)]
         \item\label{nondegenerate} There exists a non-degenerated edge $e\in E(\Gamma)$.
         \item There exists a loop $e\in E(\Gamma)$ with $o(e)=t(e)$ such that $|\Sigma_e|\geq 2$ or $|\Sigma_{\bar{e}}|\geq 2$. 
         \end{enumerate}
    Indeed, in the first case, the family of infinite normalized words of the form
    \[h_{i_1}eg\bar{e}h_{i_2}eg\bar{e}\dots\]
    is infinite,
    where $h_{i_j}\in \Sigma_e\setminus \{1_{o(e)}\}$ for any $j\in \N$ and $g\in \Sigma_{\bar{e}}\setminus \{1_{t(e)}\}$. In the second case, the family of infinite normalized words of the form
    \[g_{i_1}eg_{i_2}e\dots\]
    is also infinite
    in which $g_{i_j}\in \Sigma_e$ for any $j\in \N$.
 \end{rmk}
 
Now, Recall the definition of normalized words, path words in Definition \ref{defn: normalized word} with respect to these $\Sigma_e$ and set $\CP_v$ of path words with the same origin $v$ in Definition \ref{defn: path map}.
Let $h_1, h_2\in \CP_v$. We say $h_2$ \textit{extends} $h_1$, denoted by $h_1\subset h_2$ if $h_2=h_1g_ke_k\dots g_ne_n$ holds without cancellation  for some non-trivial path word $g_ke_k\dots g_ne_n$. Take $\CP_v$ as a vertex set and we assign an edge between two path words $h_1$ and $h_2$ whenever $h_2=h_1ge$ holds without cancellation for an edge $e\in E(\Gamma)$ and a $g\in \Sigma_e$. Then, it is straightforward to verify that the graph $\CP_v$ becomes a tree.  See also \cite[Theorem 1.17, Remark 1.18]{Ba}.

\begin{rmk}\label{rmk: identify bass serre tree}
Recall the definition of Bass-Serre tree $X_\IG$ associated with a graph of groups $\IG$ in Definition \ref{rmk: construct bass serre tree} and the action of $\pi_1(\IG, v)$ on it. For each $\gamma G_u\in V(X_\IG)$ with $\gamma=[w]$ and $w$ is a $\IG$-word, Remark \ref{rmK: locally finite and non singular}\ref{path vertices} implies  $[P(N(w))]G_u=\gamma G_u$. This allows us to use path words to identify vertices of $X_\IG$. Indeed, as in \cite[Remark 1.18]{Ba},  there is a bijective map $\varphi: \CP_v\to X_\IG$ by $\varphi(w)=[w]G_u$ with $u=t(w)$, which can be verified to be a tree isomorphism without difficulty.

Moreover, through the isomorphism $\varphi$ one may define an action of $\pi_1(\IG, v)$ on $\CP_v$ by  $s\cdot w=P(N(w_0w))$, where $s=[w_0]\in \pi_1(\IG, v)$. Under this action, the map $\varphi: \CP_v\to X_\IG$ is a $\pi_1(\IG, v)$-equivariant isomorphism.
\end{rmk}

\begin{lem}\label{lem: path words permance}
    Let $h_1, h_2\in \CP_v$ such that $h_2$ extends $h_1$. Let $s=[w]\in \pi_1(\IG, v)$ such that $w$ is a normalized word with $\ell(w)<\ell(h_1)$. Then the path word $P(N(wh_2))$ also extends the path word $P(N(wh_1))$.
\end{lem}
\begin{proof}
     Write $w=g$ or $w=g_1e_1\dots g_ne_ng$ in normalized form in the sense of Definition \ref{defn: normalized word}. Write $h_1$ and $h_2$ explicitly by
    $h_1=r_1f_1\dots r_kf_k$ and $h_2=r_1f_1\dots r_lf_l$ such that $l>k$ and all $f_i\in E(\Gamma)$ with $o(f_1)=v$ and $r_i\in \Sigma_{f_i}$.
    We now proceed by induction on the lengths of $h_1$. Let $\ell(h_1)=1$, i.e., $h_1=r_1f_1$ and therefore in this case $\ell(w)=0$, i.e., $w=g\in G_v$.
    
    First, note that $wh_1=(gr_1)f_1$ and $wh_2=(gr_1)f_1r_2f_2\dots r_lf_l$ are reduced by definition. Then by the normalization procedure described in Remark \ref{rmk: transfer to reduced and normalized word}, there exists $r'_i\in \Sigma_{f_i}$ for $1\leq i\leq l$ and $g'\in G_{t(f_1)}$ and $g''\in G_{t(f_l)}$ such that $N(wh_1)=r'_1f_1g'$ and $N(wh_2)=r'_1f_1r'_2f_2\dots r'_lf_lg''$, which implies
$P(N(wh_2))=r'_1f_1\dots r'_lf_l$ extends $P(N(wh_1))=r'_1f_1$ by definition.

Now let $k\geq 2$ and suppose the statement holds for all path words $h_1$ and $h_2$ and normalized  $w$ with length $\ell(w)<\ell(h_1)<k$. Now let $w, h_1, h_2$ be given with $\ell(w)=n$ and $\ell(h_1)=k$, i.e., $s=g_1e_1\dots g_ne_ng$ and $h_1=r_1f_1\dots r_kf_k$. Then 
\[wh_1= g_1e_1\dots g_ne_n(gr_1)f_1\dots r_kf_k\] and 
\[wh_2= g_1e_1\dots g_ne_n(gr_1)f_1\dots r_lf_l.\]
Suppose $e_n(gr_1)f_1$ is a reversal in the sense of Remark \ref{rmk: transfer to reduced and normalized word}. In this case, $f_1=\bar{e}_n$.  Then there exists a $r'=\alpha_{e_n}(\alpha^{-1}_{\bar{e}_n}(gr_1))$ such that 
\[wh_1=g_1e_1\dots g_{n-1}e_{n-1}(g_nr'r_2)f_2\dots r_kf_k= (g_1e_1\dots g_{n-1}e_{n-1}(g_nr'))\cdot (r_2f_2\dots r_kf_k)\]
and similarly
\[wh_2=(g_1e_1\dots g_{n-1}e_{n-1}(g_nr'))\cdot (r_2f_2\dots r_lf_l).\]
We define normalized word $w'=g_1e_1\dots g_{n-1}e_{n-1}(g_nr')$ and path words $h'_1=r_2f_2\dots r_kf_k$ and $h'_2=r_2f_2\dots r_lf_l$. Observe that $\ell(w')=n-1<k-1=\ell(h'_1)$. Then the induction hypothesis implies that 
\[P(N(wh_2))=P(N(w'h'_2)) \text{ extends } P(N(w'h'_1))=P(N(wh_1))\]
since the normalization of a word is unique by Remark \ref{rmk: transfer to reduced and normalized word}.
Otherwise, there is no reversal in the presentation of $wh_1= g_1e_1\dots g_ne_n(gr_1)f_1\dots r_kf_k$ and $wh_2= g_1e_1\dots g_ne_n(gr_1)f_1\dots r_lf_l$ and thus they are already reduced. Then the normalization procedure described in Remark \ref{rmk: transfer to reduced and normalized word} shows that $P(N(wh_2))$ extends $P(N(wh_1))$.   This finishes the proof. 
 \end{proof}

 \begin{defn}\label{defn: infinite normalized word}
We say an infinite sequence $\xi=g_1e_1g_2e_2\dots$ is an \textit{infinite reduced word} if there is no reversal on $\xi$.
Furthermore, $\xi$ is said to be an \textit{infinite normalized word} if any initial segment $g_1e_1g_2e_2\dots g_ne_n$ is a path word in the sense of Definition \ref{defn: normalized word}. We set the origin of $\xi$ by $o(\xi)=o(e_1)$.
\end{defn}

\begin{rmk}\label{rmk: geodesic}
Therefore, let $u\in V(\Gamma)$ with $u\neq v$ and $\gamma\in \pi[v, u]$. Recall $\gamma G_u$ is a vertex on the tree $X_\CG$ with $\gamma=[w']$ in Remark \ref{rmk: construct bass serre tree}. Define $w=P(N(w'))$ and write $w=g_1e_1\dots g_ne_n$ explicitly. Define a path of vertices $([w_m]G_{v_m})_{m=0}^n$ in which $w_0=1_v$, $v_0=v$ and $w_m=g_1e_1\dots g_me_m$, $v_m=t(e_m)$ for $1\leq m\leq n$. Then the minimum length of each path word $w_m$ by Remark \ref{rmk: minimum length reduced words} shows that $([w_m]G_{v_m})_{m=0}^n$ is a geodesic from $[1_v]G_v$ to $\gamma G_u$. In the picture of $\CP_v$, the sequence $(w_m)_{m=0}^n$ is a geodesic from $1_v$ to $w$. In general, let $(w_i)_{i=0}^\infty$ be a sequence of path words such that  $1_v=w_0\subset w_1\subset\dots$. Then $(w_i)^\infty_{i=0}$ identifies a geodesic in the tree $\CP_v$.
\end{rmk}

The following combinatorial characterization of the boundary of a Bass-Serre tree and the action on it below have been formulated in \cite[Section 2.3.1]{B-M-P-S-T} under the assumption that the $\IG$ is locally finite and non-singular. However, working in a more general framework above on actions on the visual boundary of $\cat$ space, the following shows these assumptions are not necessary. In addition, to be self-contained, we choose to provide more geometric details.
\begin{rmk}\label{rmk: action on bdry of Bass-Serre tree}
In Remark \ref{rmk: identify bass serre tree}, one may identify $\pi_1(\IG, v)\curvearrowright X_\IG$ by the $ \pi_1(\IG, v)\curvearrowright\CP_v$. Note that these actions by automorphism preserve the path metric on the tree. Let $\xi=(\gamma_nG_{u_n})_{n=0}^{\infty}$ be a geodesic in $X_\IG$, where $\gamma_n=[w_n]$ such that $w_n$ are path words  with $o(w_n)=v$ and $w_0=1_v$. Then $(w_n)$ corresponds to a geodesic ray in $\CP_v$ starting from $1_v$ and by the definition of the tree $\CP_v$, 
one necessarily has $w_0\subset w_1\subset\dots$. This further determines an infinite normalized word $\xi=g_1e_1\dots$ such that $w_n=g_1e_1\dots g_ne_n$ for $n\geq 1$. Therefore, by Remark \ref{rmk: action on bdry of tree}, the boundary  $\partial \CP_v$, and thus $\partial X_\CG$, can be identified by the set of infinite normalized words $\xi$ starting from $1_v$, i.e. the origin $o(\xi)=v$. Equipped with the cone topology generated by 
\[U(\mu)=\{\gamma\in \partial_\infty X_\CG: \mu \text{ is an initial segment of } \gamma\}\]
 for all $n\in \N$ and all path words $\mu=g_1e_1\dots g_ne_n$ with $o(\mu)=v$, it is ready verifying that the map $\varphi: \CP_v\to X_\IG$ extends to a homeomorphism $\bar\varphi:\partial_\infty \CP_v\to \partial_\infty X_\CG$.
 
 Now, for the action $\pi_1(\IG, v)$ on $\partial_\infty X_\IG$, let $s=[w]\in \pi_1(\IG, v)$ be a non-trivial element. Then the geodesic ray $\eta=(\gamma_nG_{u_n})_{n=0}^{\infty}$ on $X_\IG$ will be translated to $(s\gamma_nG_{u_n})_{n=0}^{\infty}=([P(N(ww_n))]G_{u_n})_{n=0}^{\infty}$, which is still a geodesic ray that not necessarily starts from $[1_v]G_v$. Thus, regarding $\eta$ as an element in $\partial_\infty X_\CG$, Remark \ref{rmk: action on bdry of tree} implies that  $s\cdot \eta=\eta'= (\gamma'_nG_{v_n})_{n=0}^\infty$ such that $\gamma'_0G_v=[1_v]G_v$ and  $\eta'$ is cofinal with $[P(N(ww_n))]G_{u_n})_{n=0}^{\infty}$ on $X_\IG$. 

Using the homeomorphism $\bar{\varphi}: \partial_\infty\CP_v\to \partial_\infty X_\IG$, one defines an action $\pi_1(\IG, v)\curvearrowright \partial_\infty \CP_v$ by $s\cdot \xi=\bar\varphi^{-1}(s\cdot\bar\varphi(\xi))$, where $s\cdot \bar\varphi(\xi)$ is described above. Then this makes $\bar\varphi$ a topological conjugate map between these two boundary dynamical systems and thus one may identify $\pi_1(\IG, v)\curvearrowright \partial_\infty X_\CG$ by $\pi_1(\IG, v)\curvearrowright \partial_\infty \CP_v$, which will be described explicitly next.

 Let $s=[w]\in \pi_1(\IG, v)$  and $\xi=(w_n)_{n=0}^\infty$ with $w_0=1_v$ be a geodesic ray in the tree $\CP_v$. Then $s\cdot\xi$ is exactly the geodesic ray in $\CP_v$ starting from $1_v$ and cofinal with $(P(N(ww_n)))_{n=0}^\infty$. Denote by $n_0=\ell(w)$ and write $P(N(ww_{n_0}))$ explicitly by $1_v$ or $g_1e_1\dots g_me_m$. Then, Lemma \ref{lem: path words permance} implies that $P(N(ww_k))$ extends $P(N(ww_l))$ for any $k>l\geq n_0+1$ as $\ell(w)\leq \ell(w_{n_0})<\ell(w_l)$.
Define a sequence of vertices $(t_n)_{n=0}^\infty$ on $\CP_v$ by
 \begin{enumerate}[label=(\roman*)]
     \item  $t_0=1_v$,
     \item $t_i=g_1e_1\dots g_ie_i$ for $1\leq i\leq m$,
     \item $t_{m+k}=P(N(ww_{n_0+k}))$ for $k\geq 1$,
     \end{enumerate}
Note that $t_0\subset t_1\subset\dots$ holds. Remark \ref{rmk: geodesic} entails $(t_n)_{n=0}^\infty$ is a geodesic and thus the unique geodesic ray starting from $1_v$ and cofinal with $(P(N(ww_n)))_{n=0}^\infty$. 

 In $\partial_\infty \CP_v$,
the infinite normalized word $\sigma$ determined by $(t_n)_{n=0}^\infty$ is thus obtained by applying the reduction procedure in Remark \ref{rmk: transfer to reduced and normalized word} and then the normalization procedures (possibly infinite times) described in Remark \ref{rmk: transfer to reduced and normalized word} to the infinite word $w\xi$  from left to right (see an explicit example in Remark \ref{rmk: ascending loop non minimal}).  We denote by $N(w\xi)$ this resulting infinite normalized word $\sigma$.
As a summary,  in this picture, the boundary action $\pi_1(\IG, v)\curvearrowright \partial_\infty X_\IG$ can be identified by $s\cdot \gamma=N(w\gamma)$ where $s=[w]\in \pi_1(\IG, v)$.
 \end{rmk}

\subsection{Repeatable words, flowness, and minimal boundary actions}

The following notion, appeared as in \cite[Definition 5.14]{B-M-P-S-T}, plays an important role in the investigation of the boundary action of $\pi_1(\IG, v)$.

\begin{defn}\label{defn: repeatable}
Let $w=g_1e_1g_2e_2...g_{n}e_n$ be a path word in the sense of Definition \ref{defn: normalized word}. We say $w$ is \textit{repeatable} if $o(e_1)=t(e_n)$ and $g_1e_1\neq 1_{o(\overline{e}_n)}\overline{e}_n$.
\end{defn}

\begin{rmk}
Note that a repeatable word $w$ represents an element of the fundamental group $\pi_1(\IG,v)$ where $v=o(e_1)=t(e_n)$ is the base vertex in this context. We then denote by $w^m$ the concatenation of $w$ by itself $m$ times. We note that the word $w^m$ has no reversals. In particular, it is also a path word. We then denote by $w^\infty$ the point in $\partial_\infty X_\IG$ by concatenating $w$ by itself for infinite times under the identification in Remark \ref{rmk: action on bdry of Bass-Serre tree}.
\end{rmk}

The following concept  in \cite[Paragraph 6.7]{Ba} will show that the element $[w]$ in $\pi_1(\IG, v)$ representing by a repeatable word $w$ acts on the Bass-Serre tree $X_\IG$ hyperbolically.

\begin{defn}\label{defn: coherent}
    Two edges $e, f$ of a tree is said to be \textit{coherent} if the geodesic $[o(e), o(f)]$ from the origin $o(e)$ of $e$ to the origin $o(f)$ of $f$ contains exactly one of $e$ and $f$.
\end{defn}

\begin{lem}\label{lem: coherent hyperbolic}\cite[Lemma 6.8]{Ba}
    Let $e$ be an edge of a tree $X$ and $s\in \aut(X)$. If $e$ and $se\neq e$ are coherent, then $s$ is hyperbolic and $e$ lies on $T_s$, the axis of $s$.
\end{lem}

\begin{prop}\label{prop: repeatable hyperbolic}
   Let $\IG=(\Gamma, \CG)$ be a graph of groups and $w=g_1e_1\dots g_ne_n$ be a repeatable word with $o(e_1)=v$. Then the element  $\gamma=[w]\in\pi_1(\IG, v)$ is a hyperbolic element in $\aut(X_\IG)$.
\end{prop}
\begin{proof}
    Define $w_0=1_v$ and $w_i=g_1e_1\dots g_ie_i$ for $0<i\leq n$, which are path words. Then $\epsilon=([1_v]G_v, [w_1]G_{t(e_1)})$ is an edge on $X_\IG$ by Remark \ref{rmk: construct bass serre tree} and $o(\epsilon)=[1_v]G_v$. Define $\gamma=[w]$. Then $\gamma\cdot \epsilon=([w]G_v, [wg_1e_1]G_{t(e_1)})$ with $o(\gamma\cdot \epsilon)=[w]G_v$.  Then Remark \ref{rmk: geodesic} implies that the geodesic segment $[o(\epsilon), o(\gamma\cdot \epsilon)]=([w_i]G_{t(w_i)})_{i=0}^n$ consisting of edges $\epsilon_i=([w_i]G_{t(w_i)}, [w_{i+1}]G_{t(w_{i+1})})$ for $0\leq i\leq n-1$, none of which is $\gamma\cdot \epsilon$. This is because the repeatability of $w$ implies that $wg_1e_1$ is still a path word and therefore $[wg_1e_1]G_{t(e_1)}$ is none of the vertices $[w_i]G_{t(w_i)}$ for $0\leq i\leq n$. Then Lemma \ref{lem: coherent hyperbolic} shows that $[w]$ acts hyperbolicly.
\end{proof}

We now investigate the minimality of the boundary action. The following concept was introduced in \cite[Definition 5.3]{B-M-P-S-T} using a graph theoretical interpretation. However, we would like to formulate it using words.

\begin{defn}\label{defn:flowness}
Let $\IG=(\Gamma, \CG)$ be a graph of groups. Given two edges $e, f\in E(\Gamma)$, we say $f$ \textit{flows to} $e$ if there exists a word $\nu=g_1e_1\dots g_ne_ng$  with $g\in \Sigma_f$  such that the word
\[1_{o(e)}e\nu f=1_{o(e)}eg_1e_1\dots g_ne_ngf\]
is a path word in the sense of Definition \ref{defn: normalized word}.  Let $\xi=h_1f_1h_2f_2\dots$ be an infinite normalized word in the sense of Definition \ref{defn: infinite normalized word}. We say \textit{$\xi$ flows to $e$} if $f_i$ flows to $e$ for some edge element $f_i$ in $\xi$.
\end{defn}

The following criterion detecting the minimality of the action $\pi_1(\IG, v)\curvearrowright \partial_\infty X_\IG$ was shown in \cite[Theorem 5.5]{B-M-P-S-T} in the locally finite setting. However, the proof holds in general. To be self-contained and avoid groupoid languages in \cite{B-M-P-S-T}, we choose to provide the proof here. 

\begin{prop}\label{prop: minimal on bdry}
Let $\IG=(\Gamma, \CG)$ be a graph of groups. Choose a vertex $v\in V(\Gamma)$ as the base vertex. Suppose $\xi$ flows to $e$ for every infinite normalized word $\xi$ and edge $e\in E(\Gamma)$. Then the action  $\alpha: \pi_1(\IG, v)\curvearrowright \partial_\infty X_\IG$ is minimal.
\end{prop}
\begin{proof}
   Recall that in Remark \ref{rmk: action on bdry of Bass-Serre tree}, the boundary $\partial_\infty X_\IG$ is identified with the set of all infinite normalized words $\xi$ with the origin $o(\xi)=v$, equipped with the topology generated by $U(\mu)$ for path words $\mu=g_1e_1\dots g_ne_n$ with $o(\mu)=v$. Therefore, to show the minimality of  $\alpha: \pi_1(\IG, v)\curvearrowright \partial_\infty X_\IG$, it suffices to prove that for any infinite normalized word $\xi=h_1f_1\dots h_if_i \dots$  and path word $\mu=g_1e_1\dots g_ne_n$ with $o(\xi)=o(\mu)=v$, there exists a $\gamma\in \pi_1(\CG, v)$ such that $\gamma\cdot \xi\in U(\mu)$.

   To this end, since the infinite normalized word $\xi$ flows to $e_n$ by assumption, there exist an edge $f_i$ on $\xi$ and a  word $\nu=s_1e'_1s_2e'_2\dots s_me'_mg$ with $g\in \Sigma_{f_i}$ such that the following word
   \[1_{o(e_n)}e_n\nu f_i=1_{o(e_n)}e_ns_1e'_1\dots s_me'_mgf_i\]
is a path word. Define an element $\gamma=[y]\in \pi_1(\IG, v)$ such that 
\[y=\mu\nu h^{-1}_i\bar{f}_{i-1}\dots h^{-1}_2\bar{f}_1h^{-1}_1,\]
which implies
\[\gamma\cdot \xi=N(g_1e_1\dots g_ne_n\nu f_ih_{i+1}f_{i+1}\dots)=g_1e_1\dots g_ne_n\nu f_ih_{i+1}f_{i+1}\dots.\]
The last equality holds because $1_{o(e_n)}e_n\nu f_i$ is a path word and $g_n\in \Sigma_{e_n}$. This thus shows that $\gamma\cdot \xi\in U(\mu)$.
 \end{proof}

The following criterion for the topological freeness of boundary actions was demonstrated in, e.g., \cite[Proposition 3.8]{B-I-O}. An action $G\curvearrowright S$ on a set $S$ is said to be \textit{strongly faithful} if for any finite set $F\subset G\setminus \{1_G\}$, there exists an  $x\in S$ such that $g\cdot x\neq x$ for any $g\in F$.

\begin{prop}\cite[Proposition 3.8]{B-I-O}\label{prop: strongly faithful and topo free}    
Let $G\curvearrowright X$ be a strongly hyperbolic minimal action on a countable tree $X$ by automorphisms. Then the following are equivalent.
\begin{enumerate}
    \item $G\curvearrowright X$ is strongly faithful.
    \item The induced action $G\curvearrowright \partial_\infty X$ is strongly faithful.
    \item The induced action $G\curvearrowright \partial_\infty X$ is topologically free.
    \end{enumerate}
 \end{prop}

 Due to the complicated combinatorial nature of actions of the fundamental groups on the Bass-Serre trees in the framework of the graph of groups, it is usually not easy to characterize when the action is strongly faithful in a very general manner. Nevertheless, we provide the following.

\begin{lem}\label{lem: strongly faithful permanent}
    Let $\IG=(\Gamma, \CG)$ be a graph of groups. Let $\Gamma'\subset \Gamma$ be a nontrivial subgraph of $\Gamma$ containing at least one edge and $v\in V(\Gamma')$. Suppose $\pi_1(\IG', v)\curvearrowright X_{\IG'}$ is strongly faithful. Then $\pi_1(\IG, v)\curvearrowright X_{\IG}$ is also strongly faithful.
    \end{lem}
    \begin{proof}
        Let  $F\subset \pi_1(\IG, v)\setminus \{1\}$ be a finite set.  We choose a $\xi\in \partial_\infty X_{\IG'}$ depending on $F$. If $F\cap \pi_1(\IG', v)$ is non-empty, since $\pi_1(\IG', v)\curvearrowright X_{\IG'}$ is strongly faithful, choose $\xi\in \partial_\infty X_{\IG'}$ to be an infinite normalized word with origin $o(\xi)=v$ such that 
        $\gamma\cdot \xi\neq \xi$ for any $\gamma\in F\cap \pi_1(\IG', v)$.
        Otherwise, if $F$ is disjoint from $\pi_1(\IG', v)$, choose $\xi$ to be an arbitrary infinite normalized word in $X_{\IG'}$ with $o(\xi)=v$. 
        
        Write explicitly the chosen word $\xi$ as $\xi=h_1f_1h_2f_2\dots$ with all $f_i\in E(\Gamma')$.
        Then if $F\setminus \pi_1(\IG', v)$ is non-empty,       
        let $\gamma=[y]\in F\setminus \pi_1(\IG', v)$ with $y$ a normalized word.
        Note that $\ell(y)\geq 1$ holds necessarily and if one  writes $y$ into the explicit normalized form
        $y=g_1e_1\dots g_ne_ng$, then
        there is at least one $e_i\notin E(\Gamma')$.
         This implies
        \[\gamma\cdot \xi=N(g_1e_1\dots g_ie_ig_{i+1}\dots g_ne_n(gh_1)f_1h_2f_2\dots).\]
        By the normalization process in Remarks \ref{rmk: transfer to reduced and normalized word} and  \ref{rmk: action on bdry of Bass-Serre tree} and the fact that $y$ is already normalized, the standard induction argument shows that the edge $e_i$ appears in $\gamma\cdot \xi$.
         This implies $\gamma\cdot \xi\neq \xi$ because $e_i\notin E(\Gamma')$ while all edges $f_i$ on $\xi$ come from $E(\Gamma')$. Thus, for any $\gamma\in F$, one has $\gamma\cdot \xi\neq \xi$, which shows that $\pi_1(\IG, v)\curvearrowright X_\IG$ is still strongly faithful.
           \end{proof}

Then we have the following results as our main theorem in this section. Recall that  non-singular graph of groups in Definition \ref{defn: nonsingular} means that the Bass-Serre tree $X_\IG$ is non-singular (i.e. having no leaves) in the sense of Definition  \ref{defn: non singular and locally finite for graph}. 

\begin{thm}\label{thm: boundary action}    Let $\IG=(\Gamma, \CG)$ be a non-singular graph of groups. Suppose
    \begin{enumerate}[label=(\roman*)]
\item\label{nonlinear} $X_\IG$ is not isomorphic to a line;
       \item\label{word} there is a repeatable word $w=g_0e_1g_1e_2...g_{n-1}e_n$;
     \item\label{minimal} for any infinite normalized word $\xi$ and edge $e\in E(\Gamma)$, one has $\xi$ flows to $e$.
    \end{enumerate}
    Then the action $\alpha: \pi_1(\IG, v)\curvearrowright \overline{\partial_\infty X_\IG}$ is a strong boundary action. In particular, the action $\alpha$ is a $\pi_1(\IG, v)$-boundary action where $o(w)=v$. 
    
    Further, suppose there exists a nontrivial non-singular subgraph $\Gamma'\subset \Gamma$ containing at least one edge such that
    \begin{enumerate}
    \item the Bass-Serre tree $X_{\IG'}$ is not isomorphic to a line where $\IG'=(\Gamma', \CG)$;
    \item the word $w$ in  \ref{word} above is a $\IG'$-word, i.e., $w$ belongs to $F(\IG')$;
        \item\label{c simple}  each vertex group $G_u$ is amenable for any $u\in V(\Gamma')$ and $\pi_1(\IG', v)$ is $C^*$-simple where $v=o(w)$.
        \end{enumerate}
     Then the original action $\alpha$ is topologically free. Consequently, $\pi_1(\IG, v)$ is $C^*$-simple and the reduced crossed product $C(\overline{\partial_\infty X_\IG})\rtimes_r \pi_1(\IG, v)$ is a unital simple separable purely infinite $C^*$-algebra.
 \end{thm}
\begin{proof}
  By assumption \ref{word}, there exists a repeatable word $w$ with origin $o(w)=v$. Then Proposition \ref{prop: repeatable hyperbolic} implies that $[w]$ is a hyperbolic element acting on the tree $X_\IG$ with the base vertex $v$. On the other hand, by the assumption that $\IG$ is non-singular, the tree $X_\IG$ is non-singular by Remark \ref{rmK: locally finite and non singular}. Moreover, Proposition \ref{prop: minimal on bdry} proves that $ \pi_1(\IG, v)\curvearrowright \partial_\infty X_\IG$ is  minimal. Then Proposition \ref{prop: minimal 1} shows that the action $\pi_1(\IG, v)\curvearrowright X_\IG$ by automorphism is minimal. In addition, since $\pi_1(\IG, v)\curvearrowright \partial_\infty X_\IG$ is minimal, there are no global fixed points. Then using the assumption \ref{nonlinear} that $X_\IG$ is not isomorphic to a line, Proposition \ref{prop: minimal 2} entails that $\pi_1(\IG, v)\curvearrowright X_\IG$ is strongly hyperbolic.  Thus, Proposition \ref{prop: minimal 3 and strong bdry} entails that $\alpha: \pi_1(\IG, v)\curvearrowright \overline{\partial_\infty X_\IG}$ is a strong boundary action. Finally, strong hyperbolicity and minimality of the action  $\pi_1(\IG, v)\curvearrowright X_\IG$ imply that $|\partial_\infty X_\IG|$ is infinite (see, e.g., \cite[Proposition 14]{H-P}). Therefore, $\overline{\partial_\infty X_\IG}$ is a  $\pi_1( \IG, v)$-boundary action by Remark \ref{rmk: strong bdry is bdry}. This has established the first part of the theorem.
  
Now suppose there exists a non-trivial non-singular subgraph $\Gamma'\subset \Gamma$ satisfying the assumptions above.  First, the restricted action $\alpha': \pi_1(\IG', v)\curvearrowright \partial_\infty X_{\IG'}$ is still minimal by Proposition \ref{prop: minimal on bdry} because the assumption \ref{minimal} above implies that all $\IG'$-infinite normalized words $\eta$ flow to $e$ for any $e\in E(\Gamma')$. Then the same argument in the previous paragraph above shows that $\pi_1(\IG', v)\curvearrowright X_{\IG'}$ is a strongly hyperbolic minimal action by automorphism and  $\overline{\partial_\infty X_{\IG'}}$ is a $\pi_1(\IG', v)$-boundary.
Moreover, the action $\alpha'$ is topologically amenable by \cite[Proposition 5.2.1, Lemma 5.2.6]{B-O} since all vertex groups $G_u$ are assumed to be amenable for any vertex $u\in V(\Gamma')$ (see also \cite[Remark 4.5]{B-I-O}). In addition, since $\pi_1(\IG', v)$ is $C^*$-simple by assumption (\ref{c simple}), the action $\alpha'$ is topologically free by Remark \ref{rmk: equivalence of C simple and topo free} and Proposition \ref{prop: topo free permanant}. This implies that the action $\pi_1(\IG', v)\curvearrowright  X_{\IG'}$ is strongly faithful by Proposition \ref{prop: strongly faithful and topo free}. Therefore, the action $\pi_1(\IG, v)\curvearrowright X_\IG$ is also strongly faithful by Lemma \ref{lem: strongly faithful permanent} and thus the original action $\alpha: \pi_1(\IG)\curvearrowright \overline{\partial_\infty X_\IG}$ is topologically free by Propositions \ref{prop: strongly faithful and topo free}  and \ref{prop: topo free permanant}. Therefore, $\pi_1(\IG, v)$ is $C^*$-simple by Theorem \ref{thm: C simple equivalence} because $\overline{\partial_\infty X_\IG}$ is a topologically free $\pi_1(\IG, v)$-boundary. Moreover, the reduced crossed product $C(\overline{\partial_\infty X_\IG})\rtimes_r \pi_1(\IG, v)$ is a unital simple separable purely infinite $C^*$-algebra by Theorem \ref{thm: strong bdry action pure infinite}.     
\end{proof}

\begin{rmk}
    The subgraph $\Gamma'$ in Theorem \ref{thm: boundary action} could be chosen as a single edge $e$ with two vertices or one vertex, i.e. Examples \ref{example:edge of groups} or \ref{example:loop of groups}. These correspond to amalgamated free products or HNN extension of groups. The $C^*$-simplicity of amalgamated free products and HNN extension of groups have been studied thoroughly in the literature (see, e.g., \cite{H-P}, \cite{I-O} and \cite{B-I-O}). See Theorem \ref{thm: reduced graph C star simple} below.
\end{rmk}

\subsection{On reduced graph of groups} In this part, we mainly study the boundary actions of the reduced graph of groups in the sense of Definition \ref{defn: reduced graph}. As mentioned in Remark \ref{rmk: reduced advantage}, every graph of groups can be reduced to a reduced graph of groups and this process will not change the fundamental group. Therefore, it suffices to investigated the reduced graph of groups in many cases, for example, GBS groups in Section \ref{sec: GBS}.

The following appeared in \cite[Proposition 18]{H-P} and the proof may also follow from Remark \ref{rmk: infinite bdry} together with some routine arguments. We therefore omit the proof.

\begin{prop}\label{prop: linear}
Let $\IG=(\Gamma, G)$ be a reduced graph of groups. Then $\partial_\infty X_\IG$ is infinite if and only if $X_\IG$ is not isomorphic to a line. More concretely, it is equivalent to that neither of the following holds:
\begin{enumerate}[label=(\roman*)]
    \item $\IG=(\Gamma,G)$ is an edge of groups as in Example \ref{example:edge of groups} such that $[G_{o(e)}:G_e]=[G_{t(e)}:G_e]=2$.
    \item $\IG=(\Gamma,G)$ is a loop of groups as in Example \ref{example:loop of groups} such that both $\alpha_f$ and $\alpha_{\overline{f}}$ are surjective.
\end{enumerate}
\end{prop}

\begin{defn}\label{defn:ascending}
Let $\IG=(\Gamma, \CG)$ be a graph of groups. 
A loop $e\in E(\Gamma)$ is called \textit{ascending} if at least one of the two monomorphisms $\alpha_e$ or $\alpha_{\overline{e}}$ is surjective.
\end{defn}

It is worth comparing Definition \ref{defn:ascending} with Definition \ref{defn: ascending HNN}. Let $\IG=(\Gamma, \CG)$ be a loop $e$ with $o(e)=t(e)=v$ in Example \ref{example:loop of groups} and the fundamental group $\pi_1(\CG, v)$ is isomorphic to the HNN extension $G_v*_{G_e}$ as illustrated in Example \ref{example:HNN extension}.
It is straightforward to see $G_v*_{G_e}$ is ascending if and only if $e$ is ascending.

\begin{lem}\label{Lem: not loop repeatable}
Let $\IG=(\Gamma, \CG)$ be a reduced graph of groups consisting at least one edge.  Then there is a $\IG$-repeatable word $w$.
\end{lem}
\begin{proof}
Suppose there exists an $e\in E(\Gamma)$ such that $o(e)\neq t(e)$, i.e., $e$ is not a loop. Then, since $\IG=(\Gamma, \CG)$ is reduced, one has $|\Sigma_e|\geq 2$ and $|\Sigma_{\bar{e}}|\geq 2$. This allows to
define $w=geh\bar{e}$ where $g\in \Sigma_e\setminus \{1_o(e)\}$ and $h\in \Sigma_{\bar{e}}\setminus \{1_{t(e)}\}$.  It is direct to see $w$ is repeatable in the sense of Definition \ref{defn: repeatable}.

Otherwise, one has $o(f)=t(f)$ holds for any $f\in E(\Gamma)$. Because $\Gamma$ is connected by Definition \ref{defn: graph of groups}, there has to be only one vertex $v$ such that $o(f)=t(f)=v$ for any $f\in E(\Gamma)$. Let $e$ be one of these loops. Then the word $w=1_ve$ is also repeatable. 
\end{proof}

We are now in a position to present an example of ascending-loop graphs, demonstrating  the three conditions in Proposition \ref{prop: minimal 3 and strong bdry} are not equivalent in general. 

\begin{rmk}\label{rmk: ascending loop non minimal}
    Let $\IG=(\Gamma, \CG)$ be the graph of groups that is an ascending loop. Without loss of generality, denote this loop by $e$ such that $\Sigma_e=\{1_v\}$ where $o(e)=t(e)=v$. Define an infinite normalized word $\xi=1_ve1_ve\dots$ as a point in $\partial_\infty X_\IG$. Observe that any non-trivial element $\gamma$ in $\pi_1(\IG, v)$ can be represented by normalized words $w$ such that
 either $w=1e1e\dots 1eh$ or $w=g_1\bar{e}g_2\bar{e}\dots g_n\bar{e}g$. Then it is direct to see $\gamma\cdot \xi=\xi$ by Remark \ref{rmk: action on bdry of Bass-Serre tree}. We provide an explicit example for this calculation whose idea works for general $\gamma$. Let $\gamma=[1eh]$ for some non-trivial $h$. By definition, one has
   $\gamma\cdot \xi=N(1_vehe1_ve\dots)$.
   Then because $h\in \alpha_e(G_e)$, the normalization for infinite reduced words $1_vehe1_ve\dots$ in Remark \ref{rmk: action on bdry of Bass-Serre tree} is
   \[1_vehe1_ve\dots \Rightarrow 1_ve1_veh_1e1_ve\dots \Rightarrow 1_ve1_ve1_veh_2e\dots \Rightarrow \dots\]
   for some non-trivial proper   $h_1, h_2, \dots$ in $G_v$. This implies that $\gamma\cdot \xi=\xi$.

   However, the action $\pi_1(\IG, v)\curvearrowright X_\IG$ by automorphism is minimal by \cite[Proposition 17]{H-P}.  Thus the three conditions in Proposition \ref{prop: minimal 3 and strong bdry} are not equivalent in general.
\end{rmk}

The next result explains the equivalence between non-ascending loop graphs and flowness of normalized words to edges in Definition \ref{defn:flowness}. 
\begin{prop}\label{prop: no loop minimal}
Let $\IG=(\Gamma, \CG)$ be a reduced graph of groups. Then $\xi$ flows to $e$ for every infinite normalized word $\xi$ and edge $e\in E(\Gamma)$
if and only if $\IG$ is not an ascending loop.
\end{prop}
Before presenting the proof, let us make the following helpful remark.
\begin{rmk}\label{rmk: flowness reduced}
Let $\IG=(\Gamma, \CG)$ be a graph of groups.  To show $\xi$ flows to $e$ in the sense of Definition \ref{defn:flowness} for any infinite normalized word $\xi$ and $e\in E(\Gamma)$, it suffices to show $f$ flows to $e$ for any $f\in E(\Gamma)\setminus \{e\}$. Indeed, let $\xi$ be an infinite normalized word, and suppose $\xi$ is of the form $g_1eg_2e\dots$. Then $e$ has to be a loop. Define $\nu=1_{t(e)}=1_{o(e)}$ and this implies that $1_{o(e)}e1_{t(e)}e$ is a path word. Thus $\xi$ flows to $e$ by definition. Otherwise, there has to be an $f$ on $\xi$ that is not $e$. Then by our assumption that $f$ flows to $e$, one has $\xi$ flows to $e$.
\end{rmk}

\begin{proof}
If $\IG$ is an ascending loop, then Remark \ref{rmk: ascending loop non minimal} implies that $\pi_1(\IG, v)\curvearrowright \overline{\partial_\infty X_\IG}$ is not minimal because  $\pi_1(\IG, v)$ has a global fixed point. Then Proposition \ref{prop: minimal on bdry} implies that there exists an infinite normalized word $\xi$ and an $e\in E(\Gamma)$ such that $\xi$ does not flow to $e$.

Now suppose $\IG$ is not an ascending loop. It suffices to show $f$ flows to $f'$ for any $f\neq f'\in E(\Gamma)$ by Remark \ref{rmk: flowness reduced}. That is to find a word $\nu$ such that $1_{o(f')}f'\nu f$ is a path word.  Let $T$ be a maximal tree in $\Gamma$.
 Choose an edge path $c$ connecting $t(f')$ and $o(f)$ in $T$ with the minimum length. Suppose $c$ is of length $0$, which means $t(f')=o(f)$.
If $f\neq \bar{f'}$, define $\nu=1_{o(f)}$. Otherwise, necessarily one has $f=\bar{f'}$. In this case, if $|\Sigma_f|\geq 2$,  choose $h\in \Sigma_f\setminus \{1_{o(f)}\}$ and define $\nu=h$. Otherwise, since $\IG$ is reduced, $|\Sigma_f|=1$ implies that $f$ has to be an ascending loop by Remark \ref{rmk: reduced advantage}. However, by assumption, $\IG$ is not just an ascending loop. The connectivity of $\Gamma$ implies that there has to be an edge $e$ with $o(e)=o(f)$ and $e\notin \{f, f'\}$ in which $f'=\bar{f}$. If $e$ is still a loop, we define $\nu=1_{o(e)}e1_{o(e)}$. If $e$ is not a loop, we define $\nu=1_{o(e)}eg\bar{e}1_{o(e)}$ where $g\in \Sigma_{\bar{e}}\setminus \{1_{t(e)}\}$. In any case, it is direct to see $1_{o(f')}f'\nu f$ is a path word in the sense of Definition \ref{defn:flowness}.

Otherwise, the shortest path $c$ is of length $n$ with the form $(e_1,...,e_n)$. This implies $e_i\neq \bar{e}_{i+1}$ for any $i=1,\dots, n-1$ and these $e_i$ are not loops and thus in particular $|\Sigma_{e_1}|\geq 2$ and $|\Sigma_{\bar{e}_n}|\geq 2$ since $\CG$ is reduced.  Define $\nu=g_1e_11_{o(e_2)}e_2\dots 1_{o(e_n)}e_ng_n$ where $g_1\in \Sigma_{e_1}\setminus \{1_{o(e_1)}\}$ and $g_n\in \Sigma_{\bar{e}_n}\setminus \{1_{o(\bar{e}_n)}\}$. Then, it is straightforward to see that  $1_{o(f')}f'\nu f$ is a path word.
\end{proof}

Combining Lemma \ref{Lem: not loop repeatable}, Proposition \ref{prop: no loop minimal}, and  Theorem \ref{thm: boundary action}, we have the following result, recorded as Theorem \ref{thm: A} in Introduction.

\begin{thm}\label{thm: reduced graph C star simple}
    Let $\IG=(\Gamma, \CG)$ be a reduced graph of groups. Suppose 
    \begin{enumerate}[label=(\roman*)]
        \item $\IG$ contains a non-degenerated edge $e$ with $o(e)\neq t(e)$ such that $G_{o(e)}$ and $G_{t(e)}$ are amenable and the group $G_{o(e)}*_{G_e}G_{t(e)}$ is $C^*$-simple; 
        \item or $\IG$ contains a non-ascending loop $e$ with $o(e)=t(e)$ such that $G_{o(e)}$ is amenable and $G_{o(e)}*_{G_e}$ is $C^*$-simple. 
    \end{enumerate}
    Then $\pi_1(\IG, v)$ is $C^*$-simple and the crossed product $C(\overline{\partial_\infty X_\IG})\rtimes_r\pi_1(\IG, v)$ is a unital simple purely infinite $C^*$-algebra.
\end{thm}
\begin{proof}
    First note that $\IG$ is non-singular because $\IG$ is reduced by Remark \ref{rmk: reduced advantage}. Since $\Gamma$ contains such an edge $e$ in the statement, $\IG$ is not just an ascending loop.  Then Proposition \ref{prop: no loop minimal} implies that every $\IG$-infinite normalized words $\xi$ flows to $f$ for any $f\in E(\Gamma)$. Now define a subgraph $\Gamma'$ such that $E(\Gamma')=\{e\}$ and note that $\IG'=(\Gamma', G)$ is still non-singular by definition. In addition,  Lemma \ref{Lem: not loop repeatable} shows that there is a $\IG'$-repeatable word. Therefore, Theorem \ref{thm: boundary action} entails the conclusion.
\end{proof}

We are now going to present an application of Theorem \ref{thm: reduced graph C star simple}, where the non-ascending loop in (ii)  is given by the Baumslag-Solitar group $BS(p, q)$  as follows
\[BS(p,q)=\langle x,t\ |\ tx^pt^{-1}=x^q\rangle,\]
where $p,q\in \Z\setminus\{0\}$. Alternatively, $BS(p, q)$ is an HNN extension of the form $\Z*_\Z$, so it is  the fundamental group of the group of graphs $\CG=(\Gamma, G)$ such that $\Gamma$ consists of only a loop $e$ with $G_{o(e)}\simeq \Z$ and $G_e\simeq \Z$ such that $|\Sigma_e|=|q|$ and $|\Sigma_{\bar{e}}|=|p|$. It has been shown in \cite[Theorem 3]{H-P} that $BS(p, q)$ is $C^*$-simple if and only if $|p|\neq |q|$ and $\min\{|p|, |q|\}\geq 2$. The following corollary is thus a direct application of Theorem \ref{thm: reduced graph C star simple}.

\begin{cor}
     Let $\IG=(\Gamma, \CG)$ be a reduced graph of groups. Suppose $\CG$ contains a loop $e$ with $G_e$ and $G_{o(e)}$ are isomorphic to $\Z$ and  $\min\{|\Sigma_{e}|, |\Sigma_{\bar{e}}|\}\geq 2$ as well as $|\Sigma_e|\neq |\Sigma_{\bar{e}}|$. Then $\pi_1(\IG, v)$ is $C^*$-simple and the crossed product $C(\overline{\partial_\infty X_\IG})\rtimes_r\pi_1(\IG, v)$ is a unital simple separable purely infinite $C^*$-algebra.
    \end{cor}


\section{Graph of groups with an acylindrically hyperbolic vertex group and tubular groups}\label{sec: reduced graph}

\subsection{Acylindrically hyperbolic vertex groups}
In this subsection, we shall study the reduced graph of groups with one vertex group being acylindrically hyperbolic. Acylindrically hyperbolic groups introduced by Osin in \cite{Osin} form a very large class of groups with certain negative curvature,  including many known groups such as non-elementary hyperbolic groups, relative hyperbolic groups,  $\cat$ groups with rank-one elements, and many mapping class groups. Let $G$ be a group. A subgroup $H\leq G$ is said to be \textit{s-normal} if $|g^{-1}Hg\cap H|=\infty$ for any $g\in G$. It is shown in \cite[Lemma 7.1]{Osin} that acylindrical hyperbolicity is inherited to s-normal subgroups. For non-examples, all virtually solvable groups, e.g., $\Z^n$, are not acylindrically hyperbolic. Moreover, non-trivial Baumslag-Solitar groups are not acylindrically hyperbolic (\cite[Example 7.3]{Osin}).

A group $G$ is said to have Property $P_{naive}$ if for any finite $F\subset G\setminus \{1_G\}$, there exists an infinite order element $h\in G$ such that for any $g\in F$, the subgroup $\langle g, h\rangle$ of $G$ generated by $g$ and $h$  is isomorphic to the free product $\langle g\rangle *\langle h\rangle$. It was proven in \cite{A-Da} that an acylindrically hyperbolic group $G$ satisfies the property $P_{naive}$ if $G$ has no non-trivial finite normal subgroups.

The following elementary fact about free groups  will be  used later on.
\begin{lem}\label{lem: free vertex group}
    Let $F$ be a finite set of reduced words in $\F_n$ ($n\geq 2$) and $E$ a finite set of reduced words that do not end with an generator $a$ or its inverse $a^{-1}$. Then there exists a non-trivial reduced word $g\in \F_n\setminus E$  such that $g$ does not end with $a$ or $a^{-1}$, and 
    $wg\notin \langle a \rangle$ for any $w\in F$, and $g^{-1}wg\notin \langle a \rangle$ for any $w\in F\setminus \{1_{\F_n}\}$.
\end{lem}
\begin{proof}
    Denote by \[k=\max\{|d|: d \mbox{ is the power of the beginning or ending letter of a word } w\in E\cup F\}.\]
    Choose another generator $b\in \F_n$ other than $a$ and define $g=b^{k+1}ab^{k+1}$. Note $g$ does not end with $a$ or $a^{-1}$. By the choice of $k$, one first has $g\neq h$ for any $h\in E$.
    Then if $w=1_{\F_n}$, one has $wg=g\notin \langle a \rangle$ trivially. Otherwise, each $w\in F\setminus \{1_{\F_n}\}$ can be written as $c_1^{m_1}\dots c_l^{m_l}$ where each $c_i$ belongs to the set of generators of $\F_n$ and $c_i\neq c_{i+1}$ and thus $wg=c_1^{m_1}\dots c_l^{m_l}b^{k+1}ab^{k+1}$. If $c_l\neq b$, the word $wg$ is already reduced and thus not in $\langle a \rangle$. If $c_l=b$, since $k\geq |m_l|$, one has $k+1+m_l\neq 0$. In addition, note that $c_{l-1}\neq b$ and therefore $wg=c_1^{m_1}\dots c^{m_{l-1}}_{l-1}b^{k+1+m_l}ab^{k+1}\notin \langle a \rangle$.
Furthermore, the same argument also shows  $g^{-1}wg=(b^{-k-1}a^{-1}b^{-k-1})c_1^{m_1}\dots c_l^{m_l}(b^{k+1}ab^{k+1})\notin \langle a \rangle$.
\end{proof}

The following explicit criteria for strong faithfulness in the case that $\IG$ is not locally finite will be used in the proof of Theorem \ref{thm: reduced graph c star simple 2}.

\begin{prop}\label{prop: non locally finite topo free 1}
    Let $\IG=(\Gamma, \CG)$ be a graph of groups.  Suppose that there exists an edge $e$ such that $|\Sigma_e|=[G_{o(e)}: \alpha_e(G_e)]=\infty$ and $|\Sigma_{\bar{e}}|\geq 2$. Suppose for any finite $F\subset G_{o(e)}$, and finite $E\subset \Sigma_e$, there exists a $g\in \Sigma_e\setminus E$ such that $Fg\cap \alpha_e(G_e)=\emptyset$ and $g^{-1}(F\setminus \{1_{o(e)}\})g\cap \alpha_e(G_e)=\emptyset$. Then $\pi_1(\IG, v)\curvearrowright  X_\IG$ is strongly faithful.
    \end{prop}
    \begin{proof}
Let $e\in E(\Gamma)$ such that $|\Sigma_e|=[G_{o(e)}, \alpha_e(G_e)]=\infty$ and $|\Sigma_{\bar{e}}|\geq 2$. Choose the base vertex $v=o(e)$. Let $\gamma_1, \dots, \gamma_n\in \pi_1(\IG, v)$ be non-trivial elements. Our goal is to find a $\xi\in \partial_\infty X_\IG$ such that $\gamma_i\cdot \xi\neq \xi$ for any $1\leq i\leq n$. Here, all $\gamma_i$ can be represented by normalized words $w_i$ such that either $w_i=g_i$ for a $g_i\in G_v\setminus \{1_v\}$ or 
    \[w_i=g_{1, i}e_{1, i}\dots g_{m_i, i}e_{m_i, i}g_i\]
    in a normalized form such that $o(w_i)=o(e_{1, i})=v$ in the sense of Definition \ref{defn: normalized word}. 
    
    Form the index sets $I=\{1\leq i\leq n: e_{1,i}=e\}$ and $J=\{1\leq i\leq n: w_i=g_i\}$. Then $E:=\{g_{1, i}: i\in I\}$ and $F:=\{g_i: i\in J\}$ are finite subsets of $G_v$. Denote $M=\{w_i: i\in I\}$.
    Note that by definition, $M$ is disjoint from $F$. By assumption, there exists a $g\in \Sigma_e\setminus E$ such that $g^{-1}(F\setminus \{1_{o(e)}\})g\cap \alpha_e(G_e)=\emptyset$ and $Fg\cap \alpha_e(G_e)=\emptyset$.
Define a $\xi\in \partial_\infty X_\CG$ by
    \[\xi=geh\bar{e}geh\bar{e}\dots\]
    for an $h\in \Sigma_{\bar{e}}\setminus \{1_{o(\bar{e})}\}$. 
    Let $w_i\notin M\sqcup F$. One has
     \[\gamma_i\cdot \xi=[w_i]\cdot \xi=N(g_{1, i}e_{1, i}\dots g_{m_i, i}e_{m_i, i}(g_ig)eh\bar{e}\dots).\]
Note that $g_ig\notin \alpha_e(G_e)$ by assumption that $Fg\cap \alpha_e(G_e)=\emptyset$. Thus the infinite word 
\[g_{1, i}e_{1, i}\dots g_{m_i, i}e_{m_i, i}(g_ig)eh\bar{e}\dots\] is already reduced in the sense of Definition \ref{defn: infinite normalized word}.
Then there exists an $s_i\in \Sigma_e\setminus\{1_{o(e)}\}$ such that $g_ig\in s_i\alpha_e(G_e)$. The normalization process in Remark \ref{rmk: transfer to reduced and normalized word} and \ref{rmk: action on bdry of Bass-Serre tree} implies that for each $k\in \N$ there are $h_{k, i}\in \Sigma_{\bar{e}}\setminus \{1_{o(\bar{e})}\}$ and $s_{k,i}\in \Sigma_{e}\setminus \{1_o(e)\}$ for $k\geq 1$ such that 
     \[\gamma_i\cdot \xi=[w_i]\cdot \xi=g_{1, i}e_{1, i}\dots g_{m_i, i}e_{m_i, i}s_{i}eh_{1, i}\bar{e}s_{1,i}eh_{2, i}\bar{e}\dots\]    
    which is of the normalized form.
     This verifies that $\gamma_i\cdot \xi\neq \xi$ because $e$ is not equal to any $e_{1, i}$ for $1\leq i\leq n$. Now let $w_i=g_{1, i}e_{1, i}\dots g_{m_i, i}e_{m_i, i}g_i\in M$, i.e., $e_{1, i}=e$. One still has
    \[\gamma_i\cdot \xi=[w_i]\cdot \xi=g_{1, i}e_{1, i}\dots g_{m_i, i}e_{m_i, i}s_{1, i}eh_{1, i}\bar{e}s_{2,i}eh_{2, i}\bar{e}\dots\]
    for some proper $s_{k,i}\in \Sigma_{e}\setminus \{1_o(e)\}$ and $h_{k, i}\in \Sigma_{\bar{e}}\setminus \{1_o(\bar{e})\}$ where $k\geq 1$.    
    Thus one has $\gamma_i\cdot \xi\neq \xi$ because $g\neq g_{1, i}$ for any $i\in I$ by the choice of $g$. 
    
    The last case is $w_i\in F\setminus \{1_{o(e)}\}$, which means $w_i=g_i\in G_v$. Then our assumption on $g$ that $g^{-1}(F\setminus \{1_{o(e)}\})g\cap \alpha_e(G_e)=\emptyset$ implies that 
    $w_ig=h_is_i$
    for some $h_i\in \Sigma_{e}\setminus \{1_v, g\}$ and an $s_i\in \alpha_e(G_e)$. Then one has
    \[\gamma_i\cdot \xi=[w_i]\cdot \xi=N(g_igeh\bar{e}\dots)\]
    which is of the form
    \[\gamma_i\cdot \xi=h_ies'_i\bar{e}\dots\]
    for some $s'_i\in \Sigma_{\bar{e}}\setminus \{1_{o(\bar{e})}\}$. This implies that $\gamma_i\cdot \xi\neq \xi$ because $h_i\neq g$.

    In summary, we have shown that for $\IG=(\Gamma, \CG)$, the action $\pi_1(\IG, v)\curvearrowright X_\IG$ is strongly faithful where $v=o(e)$. 
    \end{proof}

We are ready to prove the main result of this subsection.    
\begin{thm}\label{thm: reduced graph c star simple 2}
    Let $\IG=(\Gamma, \CG)$ be a reduced graph of groups.  Suppose there exists an edge $e\in E(\Gamma)$ such that the vertex group $G_{o(e)}$ is an acylindrically hyperbolic group containing no non-trivial finite normal subgroup and the edge group $G_e\simeq \Z$. Then, setting $v=o(e)$, the group $\pi_1(\IG, v)$ is $C^*$-simple and the crossed product $C(\overline{\partial_\infty X_\IG})\rtimes_r \pi_1(\IG, v)$ is a unital simple separable purely infinite $C^*$-algebra.
\end{thm}
\begin{proof}
Identifying $G_e$ with $\Z$, we denote by $a=\alpha_e(1)\in G_v$.  Since $G_{v}$ is an acylindrically hyperbolic group containing no non-trivial finite normal subgroup, then $G_{v}$ satisfies the property $P_{naive}$ by \cite{A-Da}, which implies that there exists a $b\in G_{v}$ such that $H=\langle a, b\rangle\simeq \F_2$. Therefore, $\alpha_e$ actually maps $G_e$ into the subgroup $H\leq G_v$ such that $\alpha_e(m)=a^m$. We also identify $H$ by $\F_2$ directly, and define
\[R_1=\{w\in \F_2: w\mbox{ is a reduced word not ending with } a\mbox{ or }a^{-1}\}\]
to be the left coset representative set for $\F_2/\Z$. Then choose a left coset representative set $R_2$ containing $1_v$ for $G_v/\F_2$ and define $\Sigma_e=R_2\cdot R_1$.

Now, let $F\subset G_v$ and $E\subset \Sigma_e$ be finite sets. Write $F=\{r_1w_1,\dots, r_nw_n\}$ for some $r_i\in R_2$ and $w_i\in \F_2$. We denote by $F'=\{w_i: 1\leq i\leq n\}\subset \F_2$ and $E'=E\cap R_1$. Then for $F'$ and $E'$, Lemma \ref{lem: free vertex group} implies that there exists a $g\in R_1$ such that $g\in R_1\setminus E'$, and $wg\notin \langle a \rangle$ holds for any $w\in F'$ and $g^{-1}wg\notin \langle a \rangle$ holds for any $w\in F'\setminus \{1_{\F_n}\}$.

Since $g\in R_1\setminus E'$, one has $g\in \Sigma_e\setminus E$. Now, let $rw\in F$ with $r\in R_2$ and $w\in F'$. Then if $r=1_v$, then $rwg=wg\notin \langle a \rangle$. On the other hand, if $r\neq 1_v$, one even has $rwg\notin \F_2$. Therefore, one always has $Fg\cap \alpha_e(G_e)=\emptyset$ as $\alpha_e(G_e)=\langle a \rangle\leq \F_2$.

Then let $rw\in F\setminus \{1_v\}$ with $r\in R_2$ and $w\in F'$. If $r=1_v$, then one necessarily has $w\neq 1_{\F_n}$. This implies that $g^{-1}rwg=g^{-1}wg\notin \langle a\rangle=\alpha_e(G_e)$. Otherwise, if $r\neq 1_v$, then one will have $g^{-1}rwg\notin \F_2$ because otherwise, one has $r\in g\F_2g^{-1}w^{-1}=\F_2$, which is a contradiction to the fact $r\in R_2\setminus \{1_v\}$. Thus, $g^{-1}(F\setminus \{1_v\})g\cap \alpha_e(G_e)=\emptyset$ holds.

On the other hand, if $e$ is a loop, then $|\Sigma_{\bar{e}}|=[G_v, G_e]$ is still infinite. If $e$ is not a loop, one still has $|\Sigma_{\bar{e}}|\geq 2$ because $\IG$ is reduced. Then Lemma \ref{Lem: not loop repeatable} shows that there exists a $\IG$-repeatable word $w$ such that $o(w)=v$, which works as a hyperbolic isometry on the tree $X_\IG$ by Proposition \ref{prop: repeatable hyperbolic}. Then since $\IG$ is not an ascending loop, Proposition \ref{prop: no loop minimal} and 
\ref{prop: minimal on bdry} shows that the boundary action $\pi_1(\IG, v)\curvearrowright \partial_\infty X_\IG$ is minimal. Furthermore, Proposition  \ref{prop: linear} implies that $X_\IG$ is not isomorphic to a line. By the same argument in Theorem \ref{thm: boundary action}, the action $\pi_1(\IG, v)\curvearrowright X_\IG$ by automorphism is strongly hyperbolic minimal and the boundary action $\pi_1(\IG, v)\curvearrowright \partial_\infty X_\IG$ is a strong boundary action. Finally, Proposition \ref{prop: non locally finite topo free 1} shows that $\pi_1(\IG, v)\curvearrowright X_\IG$ is strongly faithful and therefore $\pi_1(\IG, v)\curvearrowright\overline{\partial_\infty X_\CG}$ is topologically free by Propositions \ref{prop: strongly faithful and topo free} and \ref{prop: topo free permanant}. Thus, the group $\pi_1(\IG, v)$ is $C^*$-simple the reduced crossed product $C(\overline{\partial_\infty X_\IG})\rtimes_r \pi_1(\IG, v)$ is a unital simple separable purely infinite $C^*$-algebra.
\end{proof}

\subsection{$C^*$-simplicity of tubular groups}
The class of \textit{tubular} groups introduced by Wise (see, e.g., \cite{Wise}) is an important class of groups in geometric group theory. These groups can be written as a fundamental group $\pi_1(\IG, v)$ for a graph of groups $\IG=(\Gamma, \CG)$ where $\Gamma$ is finite, and all vertex groups $G_v\simeq \Z\oplus \Z$ and all edge groups $G_e\simeq \Z$. For simplicity, we call such a graph of groups a \textit{tubular graph of groups}.

The following criterion on the $C^*$-simplicity of HNN extensions appeared in \cite{B-I-O}.

\begin{thm}\cite[Theorem 4.10]{B-I-O}\label{thm: c star simple HNN}
    Let $G*_H$ be a non-ascending HNN extension. Suppose for any finite set $F\subset H\setminus \{1_H\}$, there exists a $g\in G*_H$ such that $gFg^{-1}\cap H=\emptyset$.  Then $G*_H$ is $C^*$-simple.
\end{thm}

\begin{prop}\label{prop: tubular}
    Let $\IG=(\Gamma, \CG)$ be a graph of groups such that $\Gamma$ is a loop $e$ with $o(e)=t(e)=v$ such that $G_e\simeq \Z$ and $G_{o(e)}\simeq \Z\oplus \Z$. Denote by $(m_1, n_1)=\alpha_e(1)$ and $(m_2, n_2)=\alpha_{\bar{e}}(1)$. If $|m_1|\neq |m_2|$ or $|n_1|\neq |n_2|$, then $\pi_1(\IG, v)$ is $C^*$-simple.
\end{prop}
\begin{proof}
The group $\pi_1(\IG, v)$ is an HNN extension $G_v*_{\alpha_e(G_e)}\simeq (\Z\oplus \Z)*_{\Z}$.
By Theorem \ref{thm: c star simple HNN} ($e$ is non-ascending), it suffices to verify that for any finite $F\subset \alpha_e(G_e)\setminus \{1_v\}$ there exists a $\IG$-normalized word $y$ such that the normalized word $N(yhy^{-1})\notin \alpha_e(G_e)$ for any $h\in F$. Suppose $|m_1|\neq |m_2|$ or $|n_1|\neq |n_2|$. 

We look at the case that $m_1, n_1, m_2, n_2$ are all non-zero.
Denote by $d\geq 1$ the greatest common divisor of $|m_1|$ and $|m_2|$, and by $c\geq 1$ the greatest common divisor of $|n_1|$ and $|n_2|$. Then write $m_1=dl_1$ and $m_2=dl_2$ such that $l_1$ and $l_2$ are coprime. Similarly, we write $n_1=ck_1$ and $n_2=ck_2$ for coprime integers $k_1, k_2$. Note that $|l_1|\neq |l_2|$ or $|k_1|\neq |k_2|$ and they are all non-zero.

Now, for $x\in G_e\setminus \{1_e\}=\Z\setminus \{0\}$, one has $\alpha_e(x)=(m_1x, n_1x)=(dl_1x, ck_1x)$ and thus \[\bar{e}\alpha_e(x)e=\alpha_{\bar{e}}(x)=(m_2x, n_2x)=(dl_2x, ck_2x).\] 
Then if $(dl_2x, ck_2x)\in \alpha_e(G_e)$, then necessarily one has $(dl_2x, ck_2x)=(dl_1y, ck_1y)$ for some integer $y$. Therefore,  if $l_2/l_1\neq k_2/k_1$, choosing $g=1_ve$, then one has \[g^{-1}\alpha_e(x)g=\bar{e}\alpha_e(x)e=\alpha_{\bar{e}}(x)\in G_v\setminus\alpha_e(G_e).\] Otherwise, for $(dl_2x, ck_2x)$,  if $l_2x/l_1=k_2x/k_1$ is still an integer, one has that
\[\bar{e}1_v\bar{e}\alpha_e(x)e1_ve=\bar{e}(dl_1(l_2x/l_1), ck_1(k_2/k_1))e=(dl_2^2x/l_1, ck_2^2x/k_1).\]
Now, denote by $i_0\geq 0$ the maximal integer such that $l_2^{i_0}x/l^{i_0}_1=k_2^{i_0}x/k_1^{i_0}$ is an integer. Such an $i_0$ exists because $l^{i}_2, l^{i}_1$ and $k^{i}_2, k^{i}_1$ are also coprime pairs for any positive integer $i$. Then by induction, choose the path word $g=1_ve1_ve\dots 1_ve$ by repeating $1_ve$ for $i_0+1$ times, one has 
\[g^{-1}\alpha_e(x)g=(dl_2(l_2^{i_0}x/l_1^{i_0}), ck_2(k^{i_0}_2x/k^{i_0}_1)),\]
which implies that $g^{-1}\alpha_e(x)g\in G_v\setminus\alpha_e(G_e)$ by the maximality of $i_0$. 

Now, suppose one of $m_1, n_1, m_2, n_2$ is zero. For instance, suppose $m_1=0$. For $x\in G_e\setminus \{1_e\}=\Z\setminus \{0\}$, one still has $\alpha_e(x)=(0, n_1x)$ and 
$\bar{e}\alpha_e(x)e=\alpha_{\bar{e}}(x)=(m_2x, n_2x)$.  Thus, if $m_2\neq 0$, choosing $g=1_ve$, then 
\[g^{-1}\alpha_e(x)g=(m_2x, n_2x)\in G_v\setminus\alpha_e(G_e).\]
If $m_2=0$ as well, then since $\alpha_e, \alpha_{\bar{e}}$ are monomorphism. the integer $n_1, n_2$ necessarily are different non-zero integers. Define $c\geq 1$ to be the greatest common divisor for $|n_1|$ and $|n_2|$ and write $n_1=ck_1$ and $n_2=ck_2$ for coprime integers $k_1, k_2$. Then let $i_0\geq 0$ be the maximal integer such that $k^{i_0}_2x/k^{i_0}_1$ is still an integer and define $g=1_ve\dots 1_ve$ by repeating $1_ve$ for $i_0+1$ times. The same argument above shows that \[g^{-1}\alpha_e(x)g=(0, ck_2(k^{i_0}_2x/k^{i_0}_1))\in G_v\setminus \alpha_e(G_e)\]
The same method works for all other cases that $n_1, m_2, n_2=0$.

As a summary, we have demonstrated that for any $x\in G_e\setminus \{1_o\}$, there exists a $g_x=1_ve\dots 1_ve$ such that $g_x^{-1}\alpha_e(x)g_x\in G_v\setminus \alpha_e(G_e)$. Then for any $g=1_ve\dots 1_ve$ such that $\ell(g)> \ell(g_x)$, one has 
\[g^{-1}\alpha_e(x)g=\bar{e}1_v\dots 1_v\bar{e}(g^{-1}_x\alpha_e(x)g_x)e1_v\dots 1_ve\]
is a reduced word containing edge $e$ as $g_x^{-1}\alpha_e(x)g_x\in G_v\setminus \alpha_e(G_e)$, which implies $N(g^{-1}xg)\notin\alpha_e(G_e)$. Therefore, letting $F\subset \alpha_e(G_e)\setminus \{1_v\}$ and choosing a $g=1_ve\dots 1_ve$ with $\ell(g)$ is large enough, the normalized word $N(g^{-1}xg)$ is not contained in $\alpha_e(G_e)$.
Thus, the group $\pi_1(\IG, v)$ is $C^*$-simple.
\end{proof}

As a consequence of Proposition \ref{prop: tubular} and Theorem \ref{thm: reduced graph C star simple}, we have the following.
\begin{thm}\label{thm: tubular}
    Let $\IG=(\Gamma, \CG)$ be a tubular graph of groups. Suppose $\Gamma$ contains a loop $e$ with $o(e)=t(e)=v$. Denote by $(m_1, n_1)=\alpha_e(1)$ and $(m_2, n_2)=\alpha_{\bar{e}}(1)$. Suppose $|m_1|\neq |m_2|$ or $|n_1|\neq |n_2|$. Then the tubular group $\pi_1(\IG, v)$ is $C^*$-simple and the crossed product $A=C(\partial_\infty X_\IG)\rtimes_r \pi_1(\IG, v)$ is a unital Kirchberg algebra satisfying the UCT.
    \end{thm}
    \begin{proof}
    In a tubular graph $\IG=(\Gamma, \CG)$, since all vertex group $G_v\simeq \Z\oplus \Z$ and $G_e\simeq \Z$, all $|\Sigma_e|=\infty$ and therefore $\IG$ is a reduced graph.
    By Theorem \ref{thm: reduced graph C star simple} and Proposition \ref{prop: tubular}, the tubular group $\pi_1(\IG, v)$ is $C^*$-simple and the crossed product $C(\partial_\infty X_\IG)\rtimes_r \pi_1(\IG, v)$ is a unital simple purely infinite separable $C^*$-algebra. Moreover, all vertex groups $G_{v}\simeq \Z\oplus \Z$ are amenable, the boundary action  is topologically amenable by  \cite[Proposition 5.2.1, Lemma 5.2.6]{B-O} and thus $A$ is nuclear by Remark \ref{rmk: equivalence of dynamical properties} and satisfies the UCT by \cite{Tu}.  Therefore, $A$ is a unital Kirchberg algebra satisfying the UCT.    \end{proof}

We remark that not all tubular groups are $C^*$-simple. For example, the amalgamated free product $(\Z\oplus \Z)*_\Z(\Z\oplus \Z)$ has a normal subgroup isomorphic to $\Z$. However, one may apply Theorem \ref{thm: tubular} to a large class consisting of certain tubular groups whose graph is a 2-rose graph. A graph $\Gamma$ is called a \textit{rose graph} if $V(\Gamma)=\{v\}$ and $E(\Gamma)$ consists of several loops $e$ with $o(e)=t(e)=v$.

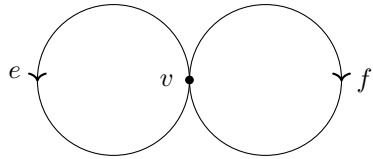
\begin{figure}[ht]
    \centering
\begin{tikzpicture}
      \node[label={left, yshift=0cm:}] at (1.7,2) {\small$v$};
      \node[label={left, yshift=0cm:}] at (-0.3,2.1) {\small$e$};
      \node[label={left, yshift=0cm:}] at (4.3,2) {\small$f$};
   
      \tikzset{enclosed/.style={draw, circle, inner sep=0pt, minimum size=.1cm, fill=black}}
      \node[enclosed, label={right, yshift=.2cm:}] at (2,2) {};

      \draw[->, thick] (0,2.02) -- (0,1.98);
      \draw[->, thick] (4,2.02) -- (4,1.98);
      \draw(0, 2) arc(-180:180:1);
      \draw(2, 2) arc(-180:180:1);

\end{tikzpicture}
    \caption{Tubular 2-Rose graph: $G_v\cong \mathbb{Z}\times \mathbb{Z}\cong \langle a,b\ |\ [a,b]=1 \rangle$, $G_e=\langle x \rangle\cong \mathbb{Z}$ where the monomorphism $\alpha_e: x\mapsto a^{m_1}b^{n_1}$ and $\alpha_{\bar{e}}: x\mapsto  a^{m_2}b^{n_2}$ and $G_f=\langle y\rangle \cong \mathbb{Z}$ where $\alpha_f: y\mapsto a^{k_1}b^{l_1}$ and $\alpha_{\bar{f}}: y\mapsto a^{k_2}b^{l_2}$.}
    \label{fig: 2-rose tubular}
\end{figure}

Denote by $\CC_{t,2}$ the following class of groups
\[G=\langle a,b, x, y\ |\  [a, b]=1, x^{-1}a^{m_1}b^{n_1}x=a^{m_2}b^{n_2}, y^{-1}a^{k_1}b^{l_1}y=a^{k_2}b^{l_2}\rangle,\]
in which $(m_i, n_i)\neq (0, 0)$ and $(k_i, l_i)\neq (0, 0)$ for any $i=1,2$.
From the definition, every $G\in \CC_{t, 2}$ is the fundamental group of the graph of groups in Figure \ref{fig: 2-rose tubular}. In particular, this class includes the following examples. 
\begin{eg}\label{eg: tubular}
    \begin{enumerate}[label=(\roman*)]
        \item Wise's non-Hopfian $\cat$ group $W$ if $m_1=1, n_1=0, m_2=n_2=2$ and $k_1=0, l_1=1, k_2=l_2=2$ in \cite{Wise1}. 
\item  Brady-Bridson group $\operatorname{BB}(p, r)$
where $0<p<r$ if $m_1=1, n_1=0, m_2=r, n_2=1$ and $k_1=p, l_1=0, k_2=r, l_2=-1$ in \cite{B-B}.
\item Wise's simple curve examples $G_d$ where $d\geq 2$ if  $m_1=d, n_1=1, m_2=1, n_2=1$ and $k_1=1, l_1=d, k_2=1, l_2=1$ in \cite{Wise}
\end{enumerate}
\end{eg}

According to the above discussion, the following is a direct consequence of Theorem \ref{thm: tubular}.

\begin{cor}\label{cor: wise BB}
The group $G\in \CC_{t, 2}$ is $C^*$-simple if $(|m_1|, |n_1|)\neq (|m_2|, |n_2|)$ or  $(|k_1|, |l_1|)\neq (|k_2|, |l_2|)$. In particular,
the Wise's non-Hopfian $\cat$ group $W$, Brady-Bridson group $\operatorname{BB}(p, r)$ for $0<p<r$, and Wise's simple curve examples $G_d$ for $d\geq 2$ are $C^*$-simple.
\end{cor}

To complement our discussion, we next show that many fundamental groups $\pi_1(\IG, v)$ in Proposition \ref{prop: tubular} are not acylindrically hyperbolic.

\begin{prop}\label{prop: non acy hyperbolic tubular}
 Let $\IG=(\Gamma, \CG)$ be a graph of groups such that $\Gamma$ is a loop $e$ with $o(e)=t(e)=v$ such that $G_e\simeq \Z$ and $G_{o(e)}\simeq \Z\oplus \Z$. Denote by $(m_1, n_1)=\alpha_e(1)$ and $(m_2, n_2)=\alpha_{\bar{e}}(1)$. Suppose
 \begin{enumerate}[label=(\roman*)]
     \item either $m_1, n_1, m_2, n_2$ are all non-zero and satisfying $m_1/m_2=n_1/n_2$;
     \item or $m_1=m_2=0$;
     \item or $n_1=n_2=0$.
 \end{enumerate}
Then $\pi_1(\IG, v)$ is not acylindrically hyperbolic.
\end{prop}
We remark that the groups $G=\pi_1(\IG, v)$ in Proposition \ref{prop: non acy hyperbolic tubular} always contain a Baumslag-Solitar group as $\alpha_e$ and $\alpha_{\bar{e}}$ maps $\Z$ to the same copy of $\Z$ in  $G_v$.
 \begin{proof}
     It suffices to show $G_v=\Z\oplus \Z$ is a s-normal subgroup of $\pi_1(\IG, v)$. So one needs to verify $w^{-1}(\Z\oplus \Z)w\cap (\Z\oplus \Z)$ is infinite for any normalized word 
     \[w=g_1e^{\epsilon_1}g_2e^{\epsilon_2}\dots g_ke^{\epsilon_k}g_{k+1}\]
     where $g_i\in (\Z\oplus \Z)$ and $\epsilon_i=1$ or $-1$ for $1\leq i\leq k$. Also, $e^{-1}$ is used to denote $\bar{e}$ for simplicity. In general, one has
     \[\bar{e}(m_1l, n_1l)e=(m_2l, n_2l)\]
     for any $l\in \Z$.

     Now suppose $m_1, n_1, m_2, n_2$ are all non-zero and $m_1/m_2=n_1/n_2$. For $w$ above with length $\ell(w)=k$, define integers $l_j=m^j_1n^j_1m^j_2n^j_2$ for all integers $j>k+1$. Then define \[h_j=(m_1l_j, n_1l_j)=(m^{j+1}_1n_1^{j}m^j_2n^j_2, m^{j}_1n_1^{j+1}m^j_2n^j_2)=(m_2d_j, n_2d_j)\]
     where $d_j=m^{j+1}_1n_1^{j}m^{j-1}_2n^j_2=m^j_1n^{j+1}_1m^j_2n^{j-1}_2$ as $m_1/m_2=n_1/n_2$.
     Then by induction on the length $\ell(w)=k$ of $w$ and the fact that $G_v=\Z\oplus \Z$ is abelian, one has
     \[w^{-1}h_jw=(p_j, q_j)\in \Z\oplus \Z\]
     in which $p_j
     =m^{j+a_k}_1n^{j+b_k}_1m^{j+c_k}_2n^{j+d_k}_2$ and $q_j=m^{j+a'_k}_1n^{j+b'_k}_1m^{j+c'_k}_2n^{j+d'_k}_2$ for some integers $a_k, b_k, c_k, d_k, a'_k, b'_k, c'_k, d'_k\in [-k, k]$; and there exists integers $x_j, y_j$ such that $p_j=m_1x_j=m_2y_j$ and $q_j=n_1x_j=n_2y_j$.
     This implies that $w^{-1}(\Z\oplus \Z)w\cap (\Z\oplus \Z)$ is infinite.

     If $m_1=m_2=0$ then $n_1, n_2$ are all non-zero because $\alpha_e$ and $\alpha_{\bar{e}}$ are monomophisms. Similarly, if $n_1=n_2=0$ then $m_1, m_2$ are all non-zero. In these cases, define $l_j=n^j_1n^j_2$ or $l_j=m^j_1m^j_2$ for all $j>k+1$. Then the same argument above shows that $w^{-1}(\Z\oplus \Z)w\cap (\Z\oplus \Z)$ is infintie.
\end{proof}   

\begin{rmk}\label{rmk: non ah tubular}
    We remark that many groups in Corollary \ref{cor: wise BB} are not acylindrically hyperbolic. For example, for the group $G\in \CC_{t, 2}$, if one additionally assumes $m_1, n_1, m_2, n_2, k_1, l_1, k_2, l_2$ are all non-zero and satisfy $m_1/m_2=n_1/n_2=k_1/l_1=k_2/l_2$. Then $G$ is not acylindrically hyperbolic. To this end, it suffices to show $w^{-1}(\Z\oplus \Z)w\cap (\Z\oplus \Z)$ is infinite for any normalized word 
     \[w=g_1x_1g_2x_2\dots g_kx_kg_{k+1}\]
     where $g_i\in (\Z\oplus \Z)$ and $x_i\in \{e, \bar{e}, f, \bar{f}\}$ for $1\leq i\leq k$. For this $w$ of length $\ell(w)=k$ and any $j>k+1$, define an integer $d_j=m^j_1n^j_1m^j_2n^j_2k_1^jl^j_1k_2^jl^j_2$ and $h_j=(m_1d_j, n_1d_j)$. Then the same argument in Proposition \ref{prop: non acy hyperbolic tubular} shows that $w^{-1}h_jw\in \Z\oplus \Z$ and therefore $\Z\oplus \Z$ is a s-normal group in $G$. 

     In addition, using the same idea, it is straightforward to see there exist other possible $m_1, n_1, m_2, n_2$ and $k_1, l_1, k_2, l_2$ satisfying the Proposition \ref{prop: non acy hyperbolic tubular} yielding non acylindrically hyperbolic groups in $\CC_{t, 2}$. But every such an example contains a Baumslag-Solitar group that is not isomorphic to $\Z\oplus \Z$. On the other hand, it seems from the literature \cite[Theorem 4.2]{Button3} that there is a characterization of (non-) acylindrically hyperbolicity for tubular groups, from which our Theorem \ref{thm: tubular} could yield more non-acylindrically hyperbolic $C^*$-simple groups as  lll such groups contain a $s$-normal subgroup $\Z$. On the other hand, Wise's groups $W$ and $G_d$ in Example \ref{eg: tubular} are indeed acylindrically hyperbolic (see \cite[Example 5.4 and 5.6]{Button1}). Therefore, their $C^*$-simplicity also follows from \cite{A-Da} once the ICC property for them is verified.
    \end{rmk}

\section{Outer automorphism groups of Baumslag-Solitar groups}\label{sec: Out(BS)}
In this section, we provide an application of Theorem \ref{thm: boundary action} to the $C^*$-simplicity of outer automorphism groups of Baumslag-Solitar groups. In the interesting cases, such a group admits a non-singular but not reduced graph of groups decomposition, for which our Theorem \ref{thm: boundary action} is applied to.

Recall $BS(p,q)=\langle x,t\ |\ tx^pt^{-1}=x^q\rangle$ where $p, q\in \Z\setminus \{0\}$. By interchanging $t\leftrightarrow t^{-1}$, one may always assume $1\leq p\leq |q|$. 
Moreover, it suffices to investigate the case $q=pn$ for $p, |n|>1$;  in other cases, the outer automorphism group $\out(BS(p, q))$ is known to be amenable and thus not $C^*$-simple. We record this fact in the following remark.


\begin{rmk}\label{rmk: amenable outer auto of BS}
It was proven in \cite[Proposition 5]{Co} that all $\aut(BS(1, q))$ are metabelian and thus amenable. This implies that $\out(BS(1, q))$ is also amenable. When $p$ does not divide $q$ properly, the following are known (see, e.g., \cite{G-H-M-R}, \cite{L} and \cite{C}):
\begin{itemize}
    \item[-] $\out(BS(p,q)) = \mathbb{Z}_{2|p-q|} \rtimes \mathbb{Z}_2$ if $p$ does not divide $q$.
    \item[-] $\out(BS(p,q)) = \mathbb{Z} \rtimes (\mathbb{Z}_2 \times \mathbb{Z}_2)$ if $p=q$.
    \item[-] $\out(BS(p,q)) = \mathbb{Z}_{2p} \rtimes \mathbb{Z}_2$ if $p = -q$.
\end{itemize}
\end{rmk}
In what follows, let us assume  $q=pn$ for $p, |n|>1$. It was shown in \cite[Section 4]{C} that $\out(BS(p,q))$ admits a graph of groups structure $\IG=(\Gamma,\mathcal{G})$, which means $\out(BS(p, np))\simeq \pi_1(\IG, v_0)$.  This was achieved by studying $\out(BS(p, q))$ on a certain tree $X_{p, q}$ and then applying Theorem \ref{thm fundamental theorem in Bass serre}. As a result, the corresponding graph of groups $\IG=(\Gamma, \CG)$ turns out to be a ray where $\Gamma$ is pictured in Figure \ref{fig: Out(BS)} and the vertex groups and edge groups are determined as the stabilizers of vertices and edges on $X_{p, q}$. See more details in \cite[Section 4]{C}.

\begin{figure}[ht]
    \centering
\begin{tikzpicture}
      \node[label={below, yshift=-0.3cm:}] at (0,0.5) {$v_0$};
      \node[label={below, yshift=-0.3cm:}] at (1.5,0.5) {$v_1$};
      \node[label={below, yshift=-0.3cm:}] at (3,0.5) {$v_2$};
      \node[label={below, yshift=-0.3cm:}] at (4.5,0.5) {$v_3$};
      \node[label={below, yshift=-0.3cm:}] at (6.25,0.5) {$\cdots$};
      \node[label={below, yshift=-0.3cm:}] at (0.8,-0.5) {$e_0$};
      \node[label={below, yshift=-0.3cm:}] at (2.3,-0.5) {$e_1$};
      \node[label={below, yshift=-0.3cm:}] at (3.8,-0.5) {$e_2$};

      \tikzset{enclosed/.style={draw, circle, inner sep=0pt, minimum size=.1cm, fill=black}}
     
      \node[enclosed, label={right, yshift=.2cm:}] at (0,0) {};
      \node[enclosed, label={right, yshift=.2cm:}] at (1.5,0) {};
      \node[enclosed, label={right, yshift=.2cm:}] at (3,0) {};      \node[enclosed, label={right, yshift=.2cm:}] at (4.5,0) {};

      \draw[->] (0, 0) -- (8, 0);
      \draw[->] (0, 0) -- (0.8, 0);
      \draw[->] (1.5, 0) -- (2.3, 0);
      \draw[->] (3, 0) -- (3.8, 0);

\end{tikzpicture}
    \caption{The graph of groups $\IG$ for $\out(BS(p,q))$.}
    \label{fig: Out(BS)}
\end{figure}
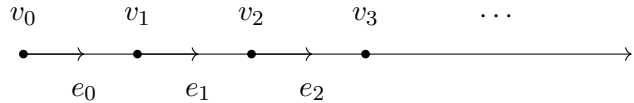

We record the known information on vertex and edge groups of $\IG$ shown in \cite[Section 4]{C}. First, $G_{v_0}$ is isomorphic to $\mathbb{Z}_{p|n-1|} \rtimes \mathbb{Z}_2$, generated by the following automorphisms $\psi$ and $\iota$ on $BS(p,q)=\langle x,t\ |\ tx^pt^{-1}=x^q\rangle$ defined by:
\begin{align*}
\psi: \hspace{0.5cm} & x \mapsto x \hspace{2cm} \iota: \hspace{0.5cm} x \mapsto x^{-1}\\
&t \mapsto xt \hspace{3cm}  t \mapsto t
\end{align*}
with the presentation $G_{v_0}=\langle \psi, \iota\  | \ \psi^{p|n-1|}=\iota^2=1,\ \iota\psi=\psi^{-1}\iota \rangle$. For $k \geq 1$,  the vertex group $G_{v_k}$ is isomorphic to $\mathbb{Z}_{n^k|n-1|} \rtimes \mathbb{Z}_2$ generated by the following automorphisms: 
\begin{align*}
\phi_{k}: \hspace{0.5cm} & x \mapsto x \hspace{2cm} \iota: \hspace{0.5cm} x \mapsto x^{-1}\\
&t \mapsto (t^{-k}x^pt^k)t \hspace{1.6cm}  t \mapsto t
\end{align*}
with the presentations
\[G_{v_k}=\langle \phi_k, \iota\ |\ \iota^2= \phi_k^{|n^k(n-1)|}=1, \ \iota\phi_k=\phi_k^{-1}\iota \rangle.\]
For edge groups, first $G_{e_0}$ is isomorphic to $\mathbb{Z}_{|n-1|}\rtimes \mathbb{Z}_2=\langle x_0, y_0\ |\ x_0^{|n-1|}=y^2_0=1, y_0x_0=x^{-1}_0y_0\rangle$. The two monomorphisms associated to $e_0$ are given by 
\begin{align*}
\alpha_{e_0}: \hspace{0.5cm} & x_0 \mapsto \psi^p \hspace{2cm} \alpha_{\bar{e}_0}: \hspace{0.5cm} x_0 \mapsto \phi_1^{|n|}\\
&y_0 \mapsto \iota \hspace{3.7cm}  y_0 \mapsto \iota
\end{align*}
These imply that $[G_{v_0}:\alpha_{e_0}(G_{e_0})]=|\Sigma_{e_0}|=p$
and $[G_{v_1}:\alpha_{\overline{e}_0}(G_{e_0})]=|\Sigma_{\bar{e}_0}|=|n|$.
Let $k\geq 1$. The edge group $G_{e_k}$ is isomorphic to $\mathbb{Z}_{|n^k(n-1)|}\rtimes \mathbb{Z}_2=\langle x_0, y_0\ |\ x_0^{|n^k(n-1)|}=y^2_0=1, y_0x_0=x^{-1}_0y_0\rangle$. Two monomorphisms associated to $e_k$ are given by
\begin{align*}
\alpha_{e_k}: \hspace{0.5cm} & x_0 \mapsto \phi_k \hspace{2cm} \alpha_{\bar{e}_k}: \hspace{0.5cm} x_0 \mapsto \phi_{k+1}^{|n|}\\
&y_0 \mapsto \iota \hspace{3.7cm}  y_0 \mapsto \iota
\end{align*}
This implies that $[G_{v_k}:\alpha_{e_k}(G_{e_k})]=|\Sigma_{e_k}|=1$ and $[G_{v_{k+1}}:\alpha_{\overline{e}_k}(G_{e_k})]=|\Sigma_{\bar{e}_k}|=|n|$.


\begin{rmk}\label{rmk: out bs not linear}
    We remark that the $\partial_\infty X_\IG$ is infinite. Indeed, since $|\Sigma_{e_0}|=p>1$ and $|\Sigma_{\bar{e}_0}|=|n|>1$ and $|\Sigma_{\bar{e}_k}|=|n|>1$,  Choose an $s_0\in \Sigma_{e_0}$. Then the family of infinite normalized words (with a period)
    \[1_{v_0}e_01_{v_1}e_1\dots1_{v_k}e_ks_k\bar{e}_k1_{v_k}\bar{e}_{k-1}\dots 1_{v_1}\bar{e}_0s_0e_01_{v_1}e_1\dots 1_{v_k}e_ks_k\bar{e}_k1_{v_k}\bar{e}_{k-1}\dots\]
among all $s_k\in \Sigma_{\bar{e}_k}$ and all $k\geq 1$ form a infinite family of elements in $\partial_\infty X_\IG$. In particular, the tree $X_\IG$ is not isomorphic to a line.
\end{rmk}

As described above, each edge $e_k$ is a collapsible edge for $k>0$ and thus the graph of groups $\IG$ is not reduced. Using Remark \ref{rmk: reduced tree} to collapse the edge $e_1$, the group $\out(BS(p, q))$ also admits the following graph of group decomposition.

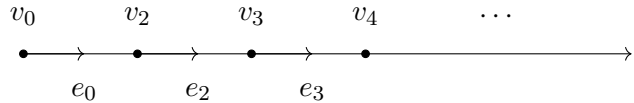
\begin{figure}[ht]
    \centering
\begin{tikzpicture}
      \node[label={below, yshift=-0.3cm:}] at (0,0.5) {$v_0$};
      \node[label={below, yshift=-0.3cm:}] at (1.5,0.5) {$v_2$};
      \node[label={below, yshift=-0.3cm:}] at (3,0.5) {$v_3$};
      \node[label={below, yshift=-0.3cm:}] at (4.5,0.5) {$v_4$};
      \node[label={below, yshift=-0.3cm:}] at (6.25,0.5) {$\cdots$};
      \node[label={below, yshift=-0.3cm:}] at (0.8,-0.5) {$e_0$};
      \node[label={below, yshift=-0.3cm:}] at (2.3,-0.5) {$e_2$};
      \node[label={below, yshift=-0.3cm:}] at (3.8,-0.5) {$e_3$};

      \tikzset{enclosed/.style={draw, circle, inner sep=0pt, minimum size=.1cm, fill=black}}
     
      \node[enclosed, label={right, yshift=.2cm:}] at (0,0) {};
      \node[enclosed, label={right, yshift=.2cm:}] at (1.5,0) {};
      \node[enclosed, label={right, yshift=.2cm:}] at (3,0) {};      \node[enclosed, label={right, yshift=.2cm:}] at (4.5,0) {};

      \draw[->] (0, 0) -- (8, 0);
      \draw[->] (0, 0) -- (0.8, 0);
      \draw[->] (1.5, 0) -- (2.3, 0);
      \draw[->] (3, 0) -- (3.8, 0);

\end{tikzpicture}
    \caption{The new graph of groups $\IG_1$ for $\out(BS(p,q))$.}
    \label{fig: OutBS1}
\end{figure}

Denote by $\IG_1=(\Gamma_1, \CG_1)$ the new graph of groups. Compared to the original graph of groups $\IG$, the new $\IG_1$ has the same vertex groups $G_{v_k}$  and $G_{e_k}$  and monomorphisms $\alpha_{e_k}$ and $\alpha_{\bar{e}_k}$ for $k=0, 2, 3\dots$ except the only changed monomorphism is that now $\alpha_{\bar{e}_0}: G_{e_0}\to G_{v_2}$ is determined by $x_0\mapsto \phi^{n^2}_2$ and $y_0\mapsto \iota$ and thus $|\Sigma_{\bar{e}_0}|=n^2>2$ as $|n|>1$. 

We next recall $C^*$-simplicity results for amalgamated free products in \cite{H-P}, \cite{I-O} and \cite{B-I-O}. The following was proven in \cite{I-O} for non-degenerated free products. See \cite[Proposition 19]{H-P} for other equivalent criteria. 

\begin{thm}\cite[Theorem 3.2]{I-O}\label{thm: c star simple amalgamated free product}
    Let $G=G_0*_H G_1$ be a non-degenerated amalgamated free product group. Suppose for any finite $F\subset H\setminus \{1_H\}$, there exists a $g\in G$ such that $g^{-1}Fg\cap H= \emptyset$. Then $G$ is $C^*$-simple.
\end{thm}

\begin{lem}\label{lem: edge important in out bs}
    Let $\IG_1=(\Gamma_1, \CG)$ be the graph of groups in Figure \ref{fig: OutBS1}. Suppose $n=2$. Then $G_{v_0}*_{G_{e_0}} G_{v_2}$ is $C^*$-simple.
\end{lem}
\begin{proof}
Suppose $n=2$. Then the group $G_{v_0}\simeq \Z_p\rtimes \Z_2$, the group $G_{v_2}=\Z_4\rtimes \Z_2$ and the group $G_{e_0}\simeq \Z_2$ equipped with monomorphisms $\alpha_{e_0}$ and $\alpha_{\bar{e}_0}$ are defined above. Since $|\Sigma_{e_0}|=p\geq 2$ and $|\Sigma_{\bar{e}_0}|=n^2=4$, the amalgamated free product $H=G_{v_0}*_{G_{e_0}} G_{v_2}$ is non-degenerated by Definition \ref{defn: non-degenerated}. Then, by Theorem \ref{thm: c star simple amalgamated free product}, it suffices to show that there exists $g\in H$ such that $\mathbb{Z}_2\cap g\mathbb{Z}_2g^{-1}=\{1\}$. Simply choose $g=\phi_2$. then one has
\[g\iota g^{-1}=\phi_2\iota \phi_2^{-1}=\phi_2 \phi_2\iota =\phi^2_2\iota \neq \iota.\]
This implies that $g\iota g^{-1}\notin \Z_2$
\end{proof}

\begin{lem}\label{lem: minimal out bs}
    Let $\IG=(\Gamma, \CG)$ be the original graph of groups for $\out(BS(p, np))$ as in Figure \ref{fig: Out(BS)}. Then  $\xi$ flows to $e$ for any infinite normalized word $\xi$ and $e\in E(\Gamma)$. The same statement also holds for the graph of groups $\IG_1=(\Gamma_1, \CG_1)$ as in Figure \ref{fig: OutBS1} for $\out(BS(p, np))$.
\end{lem}
\begin{proof}
We only prove the theorem for $\IG=(\Gamma, \CG)$ 
in Figure \ref{fig: Out(BS)} because the same arguments work for $\IG_1=(\Gamma_1, \CG_1)$ in Figure \ref{fig: OutBS1} as well. For simplicity, for $0\leq i<j$, we define path words 
    \[\nu_{i, j}=1_{v_{i+1}}e_{i+1}1_{v_{i+2}}e_{i+2}\dots 1_{v_{j-1}}e_{j-1}1_{v_{j}}\]
    and note that its inverse
    \[\gamma_{j, i}=\nu^{-1}_{i, j}=1_{v_j}\bar{e}_{j-1}1_{v_{j-2}}\dots 1_{v_{i+2}}\bar{e}_{i+1}1_{v_{i+1}}\]
    is still a path word.

    In addition, for $l\geq 1$ choose $h_l\in \Sigma_{\bar{e}_l}\setminus \{1_{v_{l+1}}\}$ and choose $h_0\in \Sigma_{\bar{e}_0}\setminus \{1_{v_{0}}\}$. Now for $l\geq 1$ define \[\mu_{l}=h_l\bar{e}_l1_{v_{l}}\bar{e}_{l-1}1_{v_{l-1}}\dots 1_{v_1}\bar{e}_0h_0.\]
   Now, it suffices to show $f$ flows to $e$ for any different $e, f\in E(\Gamma)$ in the sense of Definition \ref{defn:flowness} by Remark \ref{rmk: flowness reduced}.  There are four cases. 
    \begin{enumerate}[label=(\roman*)]
    \item Suppose $e=e_i$ and $f=e_j$. If $i<j$,  define $\nu=\nu_{i, j}$. Otherwise, if $j<i$, define
     $\nu=\mu_{i}e_0\nu_{0, j}$. 
    
    \item Suppose $e=\bar{e}_i, f=e_j$. Choose an $h_0\in \Sigma_{\bar{e}_0}\setminus \{1_{v_{0}}\}$. If $i>0$, we define $\nu=\mu_{i-1}e_0\nu_{0, j}$. Otherwise, if $i=0$, define $\nu=h_0e_0\nu_{0, j}$.

    \item Suppose $e=e_i, f=\bar{e}_j$. If $i<j$, choose a  $h_{j+1}\in \Sigma_{\bar{e}_j}\setminus \{1_{v_{j+1}}\}$ and define $\nu=\nu_{i, j+1}h_{j+1}$. 
    Otherwise, $i\geq j$ holds.  Choose an $h_{i+1}\in \Sigma_{\bar{e}_i}\setminus \{1_{v_{i+1}}\}$. If $i=j$, then define $\nu=h_{i+1}$ and if  $i>j$ then define $\nu=h_{i+1}\gamma_{i+1, j}$.

    \item Suppose $e=\bar{e}_i, f=\bar{e}_j$. If $0<i<j$, choose an  $h_{j+1}\in \Sigma_{\bar{e}_j}\setminus \{1_{v_{j+1}}\}$ and define $\nu=\mu_{i-1}e_0\nu_{0, j+1}h_{j+1}$. If $0=i<j$, define $\nu=h_0\nu_{0, j+1}h_{i+1}$ Otherwise, $i>j$ holds and define $\nu=\gamma_{i, j}$.
    \end{enumerate}
In any case, one has that $1_{o(e)}e\nu f$ is a path word, and thus by definition $f$ flows to $e$ as desired.
\end{proof}

Then we have the following result as an application of Theorem \ref{thm: boundary action}.
\begin{thm}\label{thm: out bs main}
Let $p>1$. The group $\out(BS(p, 2p))$ is $C^*$-simple.  Let $\IG=(\Gamma, \CG)$ and $\IG_1=(\Gamma_1, \CG_1)$ be  graphs of groups as in Figures \ref{fig: Out(BS)} and \ref{fig: OutBS1} for $\out(BS(p, 2p))$, respectively.  The crossed products  $C(\partial_\infty X_\IG)\rtimes_r \out(BS(p, 2p))$ and $C(\partial_\infty X_{\IG_1})\rtimes_r \out(BS(p, 2p))$ from  $\alpha: \out(BS(p, 2p))\curvearrowright \partial_\infty X_\IG$ and $\beta: \out(BS(p, 2p))\curvearrowright \partial_\infty X_{\IG_1}$ are unital Kirchberg algebras satisfying the UCT.  
\end{thm}
\begin{proof}
    First, note that the graphs of groups $\IG$ and $\IG_1$ are locally finite and non-singular.  Remark \ref{rmk: reduced tree} shows that $\out(BS(p, 2p))\simeq \pi_1(\IG, v_0)\simeq \pi_1(\IG', v_0)$. Now, we work in $\Gamma_1$ and we denote by $\Gamma'_1$ to be the subgraph in $\Gamma_1$ consisting of $e_0$ together with its vertices $v_0$ and $v_2$. Write $\IG'_1=(\Gamma'_1, \CG'_1)$.  Recall $|\Sigma_{e_0}|=p>1$ and $|\Sigma_{\bar{e}_0}|=2^2=4$ so that one may choose $g_0\in \Sigma_{e_0}\setminus \{1_{v_0}\}$ and $g_1\in \Sigma_{\bar{e}_0}\setminus \{1_{v_2}\}$. Then define $w=g_0e_0g_1\bar{e}_0$, which is a repeatable word. In addition,  the existence of $e_0$ in $\Gamma'_1$ implies that $\partial_\infty X_{\IG'_1}$ is infinite by Remark \ref{rmk: infinite bdry} and thus $X_{\IG'_1}$ is not isomorphic to a line. Consequently, $X_{\IG_1}$ is not isomorphic to a line as well. Moreover, Lemma \ref{lem: minimal out bs} shows that $\xi$ flows to any $f$ for any $\IG_1$-infinite normalized word $\xi$ and $f\in E(\Gamma_1)$.
    Finally, Lemma \ref{lem: edge important in out bs} implies that $\pi_1(\IG'_1, v_0)= G_{v_0}*_{G_{e_0}} G_{v_2}$ is $C^*$-simple and $G_{v_0}$ and $G_{v_1}$ are finite groups and thus amenable. Then Theorem \ref{thm: boundary action} implies that the  $\out(BS(p, 2p))\simeq \pi_1(\IG_1, v_0)$ is $C^*$-simple and the crossed product $C^*$-algebra $A=C(\partial_\infty X_{\IG_1})\rtimes_r \out(BS(p, 2p))$ is a unital simple purely infinite separable $C^*$-algebra. On the other hand, since all $\IG_1$-vertex groups $G_{v_k}$ are finite, the action $\beta$ is actually amenable by  \cite[Proposition 5.2.1, Lemma 5.2.6]{B-O} and thus $A$ is nuclear by Remark \ref{rmk: equivalence of dynamical properties} and satisfies the UCT by \cite{Tu}.  Therefore, $A$ is a unital Kirchberg algebra satisfying the UCT.

    Now, we work in the graph of groups $\IG=(\Gamma, \CG)$. First, $\partial_\infty X_\IG$ is not isomorphic to a line by Remark \ref{rmk: out bs not linear}. Choose a $r\in \Sigma_{\bar{e}_0}\setminus \{1_{v_1}\}$ and the word $w'=g_0e_0r\bar{e}_0$ is a repeatable word. Then Lemma \ref{lem: minimal out bs} shows that $\xi$ flows to any $f$ for any $\IG$-infinite normalized word $\xi$ and $f\in E(\Gamma)$. Moreover, since $\pi_1(\IG, v_0)$ is shown to be $C^*$-simple and all vertex groups in $\IG$ are finite, the boundary action $\alpha$ is a topologically amenable topologically free strong boundary action and thus the crossed product $C(\partial_\infty X_\IG)\rtimes_r \out(BS(p, 2p))$ is a unital Kirchberg algebra satisfying the UCT by Theorem \ref{thm: boundary action}, Remark \ref{rmk: equivalence of C simple and topo free} and \cite{Tu}.
    \end{proof}

We remark that in the case $q=np$ and $n\neq 2$, Theorem \ref{thm: out bs main} is not true. This is mainly because the amalgamated free product $H=G_{v_0}*_{G_{e_0}} G_{v_2}$ can be written as 
\[(\Z_{p|n-1|}\rtimes \Z_2)*_{\Z_{|n-1|}\rtimes \Z_2} (\Z_{|n(n-1)|}\rtimes \Z_2)=(\Z_{p|n-1|}*_{\Z_{|n-1|}}\Z_{|n(n-1)|})\rtimes \Z_2\]
is not $C^*$-simple. Indeed, since $K=\Z_{p|n-1|}*_{\Z_{|n-1|}}\Z_{|n(n-1)|}$ contains an abelian normal subgroup $\Z_{|n-1|}$, the group $K$, as a finite index subgroup of $H$, which is not $C^*$-simple. Therefore $H$ is not $C^*$-simple by  \cite[Proposition 19(iii)]{Harpe}. Actually, the same argument shows that if $n\neq 2$, the group $\out(BS(p, np))$ is not $C^*$-simple using the following explicit presentation provided in \cite{C}.

\begin{thm}\cite[Theorem 4.4]{C}\label{thm: presentation of out bs}
    Let $q=pn$ such that $p, |n|\in \Z_{\geq 1}$ and write $BS(p,q)=\langle x,t\ |\ tx^pt^{-1}=x^q\rangle$. Then the automorphism group $\aut(BS(p, q))$ is generated by automorphisms $c_x, c_t, \psi, \iota$ and $\phi_k$ for $k\geq 1$ subjects to following relations
    \begin{align*}
        &c_tc^p_xc^{-1}_t=c^q_x, \ \   \iota^{-1}=\iota,\ \ 
        \iota c_x\iota=c_x^{-1}, \ \  \iota c_t\iota=c_t,\ \  \iota\psi\iota=\psi^{-1},\ \  \iota\phi_k\iota=\phi^{-1}_k\text{ for }k\geq 1,\\
        &\psi^p=\phi^n_1,\ \ \psi^{p(n-1)}=c_x^{-p},\ \ \psi c_x\psi^{-1}=c_x,\ \ \psi c_t\psi^{-1}=c_xc_t, \ \ \phi_kc_x\phi^{-1}_k=c_x\text{ for }k\geq 1,\\
        &\phi_kc_t\phi_k^{-1}=c_t^{-k}c_x^pc_t^{k+1} \text{ for }k\geq 1, \ \ \phi^n_{k+1}=\phi_k \text{ for }k\geq 1,
    \end{align*}
  where $c_x$ and $c_t$ are conjugation automorphisms by generators $x, t$ of $BS(p, q)$,  and  $\psi, \iota, \phi_k$ ($k\geq 1$) are outer automorphisms of $BS(p, q)$ defined in \cite[Subsection 4.3]{C}. Then the outer automorphism group $\out(BS(p, q))$ is generated by the image of $\iota, \psi$ and $\phi_k$ for $k\geq 1$ and has the form
   \[\out(BS(p, q))=((\mathbb{Z}_{|p(n-1)|})\ast_{\mathbb{Z}_{|n-1|}}(\mathbb{Z}[\frac{1}{|n|}]/|n(n-1)|\mathbb{Z}))\rtimes \mathbb{Z}_2.\]
   \end{thm}

Using Theorem \ref{thm: presentation of out bs}, we have the following complete characterization of $C^*$-simplicity.

\begin{thm}\label{thm: out bs C simple}
   $\out(BS(p, q))$ is $C^*$-simple if and only if $q=2p$ and $p>1$. 
\end{thm}
\begin{proof}
If $q=2p$ and $p>1$, then Theorem  \ref{thm: out bs main} show that $\out(BS(p, q))$ is $C^*$-simple.    Now for the converse, suppose $\out(BS(p, q))$ is $C^*$-simple. Then this implies that $q=np$ for some $n\in \Z$ and $p>1$ because otherwise, $\out(BS(p, q))$ is amenable by Remark \ref{rmk: amenable outer auto of BS}. Then Theorem \ref{thm: presentation of out bs} shows 
    \[\out(BS(p, q))=((\mathbb{Z}_{|p(n-1)|})\ast_{\mathbb{Z}_{|n-1|}}(\mathbb{Z}[\frac{1}{|n|}]/|n(n-1)|\mathbb{Z}))\rtimes \mathbb{Z}_2.\]
    We write $N=(\mathbb{Z}_{|p(n-1)|})\ast_{\mathbb{Z}_{|n-1|}}(\mathbb{Z}[\frac{1}{|n|}]/|n(n-1)\mathbb{Z}|)$ for simplicity.  If $n\neq 2$, then $N$ is an amalgamated free product of two abelian groups. This implies that $\Z_{|n-1|}$ is a non-trivial normal abelian subgroup of $N$. On the other hand, since $N$ is a normal subgroup of $\out(BS(p, q))$ with finite index, the group $N$ is also $C^*$-simple by \cite[Proposition 19(iii)]{Harpe}. But this is a contradiction to the fact that the amenable radical of $N$ is non-trivial. Therefore, $n=2$ holds necessarily.
    \end{proof}

\begin{rmk}\label{rmk: out bs acy hyperbolic}
   Here is an alternative way to show the $C^*$-simplicity of $\out(BS(p, 2p))$. It follows from \cite[Theorem 2.1]{M-O}  that $\out(BS(p, 2p))$ is acylindrically hyperbolic once one has verified the minimality of the topological action $\out(BS(p, 2p))\curvearrowright \partial_\infty X_{\IG}$ as well as the action $\out(BS(p, 2p))\curvearrowright  X_{\IG}$ by automorphism by Lemma \ref{lem: minimal out bs}, Proposition \ref{prop: minimal on bdry} and Proposition \ref{prop: minimal 1}.  Therefore, the $C^*$-simplicity of $\out(BS(p, 2p))$ also follows from \cite{A-Da} if one may verify that it has the ICC property.
\end{rmk}




\section{$n$-dimensional Generalized Baumslag-Solitar groups}\label{sec: GBS}
This section is aiming to study $\text{GBS}_n$ graphs and $\text{GBS}_n$ groups. For simplicity, we call a graph of groups $\IG=(\Gamma, \CG)$ a \textit{$GBS_n$ graph of groups} if all vertex and edge groups are $\Z^n$. In general, we do not require that $\Gamma$ is finite for a $\text{GBS}_n$ graphs.   An $n$-\textit{generalized Baumslag-Solitar group} ($\text{GBS}_n$ group) is a fundamental group of a finite $\text{GBS}_n$ graph.  Note that $\text{GBS}_1$ group are nothing but usual GBS groups as we will address next.

\subsection{1-dimensional case}
The main theorem in this subsection, i.e., Theorem \ref{thm: GBS C star simple} on $C^*$-simplicity of (finitely generated) GBS groups have been proven in \cite[Proposition 9.1]{M-V} by a straightforward geometric analysis of actions on the tree together with Theorem \ref{thm: GBS topological freeness} established in \cite{B-M-P-S-T} on the topological freeness of the boundary action. On the other hand,  we will show Theorem \ref{thm: GBS C star simple} is another consequence of Theorem \ref{thm: boundary action} together with Theorem \ref{thm: GBS topological freeness}. Moreover, we will also demonstrate applications of Theorem \ref{thm: boundary action} to infinitely generated GBS groups.

A GBS group is said to be \textit{non-elementary} if it is not isomorphic to one of the following groups: $\mathbb{Z}$, $\mathbb{Z}^2$ or the Klein bottle group $BS(1, -1)=\Z\rtimes \Z$.

Let $\IG=(\Gamma, \CG)$ be a GBS graph of groups.  Since every group in $\mathcal{G}$ is isomorphic to $\mathbb{Z}$, then each inclusion map $\alpha_e: G_e \hookrightarrow G_{o(e)}$ is given by multiplication by a non-zero integer, which we denote by $\lambda(e)$. This data can be reflected by adding labels on the edges in $\Gamma$ by defining a function $\lambda: E(\Gamma)\rightarrow \mathbb{Z}\setminus\{0\}$, which is called the labeling function. Given a choice of generators of $G_e$ and $G_{t(e)}$, the inclusion map $\alpha_e: G_e \hookrightarrow G_{o(e)}$ is multiplication by $\lambda(e)$. By definition, for GBS graphs of groups, $|\Sigma_e|=|\lambda(e)|$ holds for any $e\in E(\Gamma)$. In particular, GBS graphs of groups are all locally finite.


Let $G$ be a non-elementary GBS group. It is well-known that there is a well-defined homomorphism $\Delta: G\rightarrow \mathbb{Q}^{\times}$, which is called the \textit{modular homomorphism} (see, e.g., \cite[Section 2]{L2}). In the context of GBS graphs of groups $\IG=(\Gamma, \CG)$, such modular homomorphism can be interpreted as $\Delta: \pi_1(\IG, v)\to \Q^{\times}$ by
\[w=g_1e_1\dots g_ne_ng\mapsto \prod_{i=1}^n \frac{\lambda(\overline{e}_i)}{\lambda(e_i)}\]
where $w\in \pi_1(\IG, v)$ is not necessarily in the normalized form. In addition, to be compatible with the definition of GBS graphs above, we also do not require the graph $\Gamma$ to be finite in the definition of the modular homomorphism. See also \cite[Section 7.1]{B-M-P-S-T}. The fundamental group $\pi_1(\IG, v)$ is \textit{unimodular} if $\Delta(\pi_1(\IG, v))\subseteq \{1,-1\}$. 


\begin{rmk}\label{rmk: unimordular characterization}
    Let $G$ be a non-elementary GBS group. The modular homomorphism $\Delta$ on $G$ does not depend on the choice of the graph of groups $\IG=(\Gamma, \CG)$ representing $G$. See \cite[Section 2]{L2}. The following are equivalent. See \cite[Section 2]{L2} and \cite[Proposition 4.1]{L2}.
    \begin{enumerate}[label=(\roman*)]
        \item $G$ is unimodular.
        \item $G$ has a non-trivial center.
        \item $G$ has a normal infinite cyclic group.
        \item $G$ is virtually $\F_n\times \Z$ for some $n\geq 2$.
        \item $G=\F_n\rtimes \Z$ where the action $\Z\curvearrowright \F_n$ is an outer automorphism of finite order. 
    \end{enumerate}
\end{rmk}

If the underlying graph $\Gamma$ in a non-singular GBS graph of groups $\IG=(\Gamma, \CG)$ is finite, one has the following nice characterization of topological freeness of boundary actions.

\begin{thm}\cite[Corollary 7.11]{B-M-P-S-T}\label{thm: GBS topological freeness}
Let $\IG=(\Gamma, \CG)$ be a non-singular GBS graph of groups such that $\Gamma$ is finite. Then the action $\alpha: \pi_1(\IG, v)\curvearrowright \partial_\infty X_\IG$ is topologically free if and only if $\pi_1(\IG, v)$ is not unimodular.
\end{thm}
We remark that  the finiteness of $\Gamma$ is necessary in Theorem \ref{thm: GBS topological freeness}. See Proposition \ref{prop: infinite GBS topo free} below.

Recall that elementary GBS groups are all virtually abelian. Remark \ref{rmk: unimordular characterization} implies that unimodular GBS groups have a non-trivial center. In addition, all BS groups $BS(1, n)$ are solvable and the infinite dihedral group $D_\infty$ is virtually $\Z$. Thus, it is direct to see that these groups are not $C^*$-simple. 

\begin{rmk}\label{rmk: gbs not linear}
    Let $\IG=(\Gamma, \CG)$ be the reduced GBS graph described in Proposition \ref{prop: linear} such that $X_\CG$ is isomorphic to a line. Observe that the fundamental group $\pi_1(\IG, v)$ is either of the form $\langle a, b|a^2=b^2\rangle$ or of the form $BS(1, 1)=\Z^2$ or $BS(1, -1)=\Z\rtimes \Z$. In addition, by defining $x=b^{-1}a$ by $t=b$, then the group  $\langle a, b|a^2=b^2\rangle=\langle x, t| t^{-1}xt=x^{-1}\rangle$ is also isomorphic to $BS(1, -1)$.    Thus, in the case that $X_\IG$ is isomorphic to a line, then $\pi_1(\IG, v)$ is elementary.
\end{rmk}

The following theorem was first proven in \cite[Proposition 9.1]{M-V} and we recover this in our framework.

\begin{thm}\label{thm: GBS C star simple}
    Let $G$ be a (finitely generated) GBS group. Then $G$ is $C^*$-simple if and only if $G$ is non-elementary, non-virtually $\F_m\times \Z$ ($m\geq 2$), and not isomorphic to $BS(1, n)$ for any $n\in \Z\setminus \{0\}$. 
\end{thm}
\begin{proof}
It suffices to show if a (finitely generated) GBS group $G$ satisfies the conditions above, then it is $C^*$-simple as the converse is clear.
    Let $G$ be such a group. Choose a GBS graph of groups $\IG=(\Gamma, \CG)$ such that $\Gamma$ is finite. Then Remark \ref{rmk: reduced tree} implies that one may assume $\Gamma$ is reduced and thus non-singular as the reduction of a GBS graph is still a GBS graph. In addition, the reduced GBS graph $\IG$ cannot be a vertex only without any edges because otherwise, $G\simeq \Z$ is elementary. Moreover, by Remark \ref{rmk: gbs not linear}, the tree $X_\CG$ is not isomorphic to a line because $G$ is not elementary. 
    Furthermore, Lemma \ref{Lem: not loop repeatable} shows that there exists a repeatable $\IG$-path word $w$. Denote by $v=o(w)$.   
    Since $G$ is not isomorphic to $BS(1, n)$, the graph of groups $\IG$ is not just an ascending loop only. Then Proposition \ref{prop: no loop minimal} implies that $\xi$ flows to $e$ for any $\IG$-infinite normalized word $\xi$ and edge $e\in E(\Gamma)$.  Then Theorem \ref{thm: boundary action} shows that $\beta: G=\pi_1(\IG, v)\curvearrowright \partial_\infty X_\IG$ is a strong boundary action and thus a $G$-boundary action by Remark \ref{rmk: strong bdry is bdry}.  Then, Remark \ref{rmk: unimordular characterization} entails that $G=\pi_1(\IG, v)$ is not unimodular and then Theorem \ref{thm: GBS topological freeness} shows that  $\beta$ is a topologically free.  Therefore, the action $\beta: G=\pi_1(\IG, v)\curvearrowright \partial_\infty X_\IG$ is a topological free $G$-boundary action. Therefore $G$ is $C^*$-simple.
\end{proof}

\begin{cor}\label{cor: boundary action of c star simple GBS}
    Let $G$ be GBS group that is  non-elementary, non-virtually $\F_m\times \Z$ ($m\geq 2$), and is not isomorphic to $BS(1, n)$ for any $n\in \Z\setminus \{0\}$. Suppose $\IG=(\Gamma, \CG)$ is a reduced GBS graph of groups representing $G$. Then $C(\partial_\infty X_\IG)\rtimes_r G$ is a unital Kirchberg algebra satisfying the UCT. 
\end{cor}
\begin{proof}
    In Theorem \ref{thm: GBS C star simple}, we have shown that $G\curvearrowright \partial_\infty X_\IG$ is a topological free strong boundary action. Moreover, since all vertex and edge stabilizers are isomorphic to $\Z$, the action is (topological) amenable. Then the $C^*$-algebra $A=C(\partial_\infty X_\IG)\rtimes_r G$ is a unital Kirchberg algebra satisfying the UCT by Theorem \ref{thm: boundary action}, Remark \ref{rmk: equivalence of dynamical properties} and \cite{Tu}. 
\end{proof}

To end this subsection, we provide new $C^*$-simple groups from the GBS graph of groups $\IG=(\Gamma, \CG)$ with infinite underlying graph  $\Gamma$, in which case, $\pi_1(\IG, v)$ is infinitely generated.

\begin{prop}\label{prop: infinite GBS topo free}
    Let $\IG=(\Gamma, \CG)$ be a graph of groups such that $|\Sigma_e|\geq 2$ for any $e\in E(\Gamma)$. Suppose there exists a vertex $v\in V(\Gamma)$ emitting infinitely many edges, i.e. there are infinitely many different $e_i\in E(\Gamma)$ for $i\in \N$ such that $o(e_i)=v$. Suppose for any finite set  $F\subset G_v\setminus \{1_v\}$ and finite $E\subset \{e\in E(\Gamma):o(e)=v\}$, there exists an edge $f$ with $o(f)=v$ such that $f\notin E$ and $g\notin \alpha_f(G_f)$ for any $g\in F$. 
    Then $\pi_1(\IG, v)\curvearrowright \partial_\infty X_\IG$ is strongly faithful.
    \end{prop}
\begin{proof}
Choose the $v$ as our base vertex. Let $\gamma_1, \dots, \gamma_n\in \pi_1(\IG, v)$ be non-trivial elements such that every $\gamma_i=[w_i]$, where $w_i$ is a normalized word. Then either $w_i=g_i$ for some $g_i\in G_v\setminus \{1_v\}$ or 
   \[w_i= g_{1, i}e_{1, i}\dots g_{m_i, i}e_{m_i, i}g_i\]  
  such that $o(w_i)=o(e_{1, i})=t(e_{m_i,i})=v$.  Our goal is to find a $\xi\in \partial_\infty X_\CG$ such that $\gamma_i\cdot \xi\neq \xi$ for any $1\leq i\leq n$.  Define $F=\{g_i: 1\leq i\leq n\}$ and $E=\{e_{1, i}, \bar{e}_{m_i, i}: 1\leq i\leq n\}$. Then by assumption, there exists an  edge $f$ with $o(f)=v$ such that $f\notin E$ and $g\notin \alpha_f(G_f)$ for any $g\in F$.

  Now define an infinite normalized word $\xi \in \partial_\infty X_\IG$ of the form
    \[\xi=1_vfh\bar{f}sfh\bar{f}sf\dots\]
    such that $s\in \Sigma_f\setminus \{1_{v}\}$ and $h\in \Sigma_{\bar{f}}\setminus \{1_{t(f)}\}$. Then for any $1\leq i\leq n$, if $w_i=g_i\neq 1_v$, one has
    \[\gamma_i\cdot \xi=N(g_ifh\bar{f}sf\dots).\]
    Since $g_ifh\bar{f}sf\dots$ is already reduced, the $\gamma_i\cdot \xi$  is of the form
    \[r_ifh_{1, i}\bar{f}s_{1, i}f\dots\]
    as an infinite normalized word for some $r_i\in \Sigma_f\setminus \{1_v\}$ such that $r_i\alpha_f(G_f)=g_i\alpha_f(G_f)\neq \alpha_f(G_f)$ and proper $h_{k, i}\in \Sigma_{\bar{f}}\setminus \{1_{t(f)}\}$ and $s_{k, i}\in \Sigma_f\setminus \{1_{o(f)}\}$ by the choice of $f$.  This implies that $\gamma_i\cdot \xi\neq \xi$ because $r_i\neq 1_v$.    
    
    Now, for $\gamma_i=g_{1, i}e_{1, i}\dots g_{m_i, i}e_{m_i, i}g_i$. Note that
    \[\gamma_i\cdot \xi=N(g_{1, i}e_{1, i}\dots g_{m_i, i}e_{m_i, i}g_ifh\bar{f}\dots).\]
Because $f\neq \bar{e}_{m_i, i}$, the word $g_{1, i}e_{1, i}\dots g_{m_i, i}e_{m_i, i}g_ifh\bar{f}\dots$ is an infinite reduced word in the sense of Definition \ref{defn: infinite normalized word}. Therefore,  one has
    \[\gamma_i\cdot \xi=g_{1, i}e_{1, i}\dots g_{m_i, i}e_{m_i, i}r_ifh_{1, i}\bar{f}s_{1,i}fh_{2, i}\bar{f}\dots\]
     for some $r_i\in \Sigma_f$  and proper $h_{k, i}\in \Sigma_{\bar{f}}\setminus \{1_{t(f)}\}$ and $s_{k, i}\in \Sigma_f\setminus \{1_{o(f)}\}$ for $k\in \N$  This verifies that $\gamma_i\cdot \xi\neq \xi$ because $f$ is not equal to any $e_{1, i}$ for $1\leq i\leq n$.
\end{proof}

\subsubsection*{Octopus GBS graph groups}
The following GBS graph consisting of one vertex $v$ emitting infinitely many edges $e_1, e_2\dots, $ such that $|\Sigma_{e_i}|=i+1\geq 2$ and $|\Sigma_{\bar{e}_i}|=2$, is named by the \textit{Octopus} graph. 

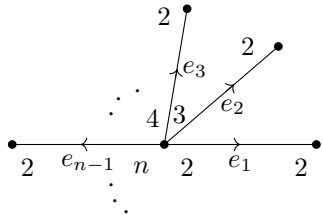
\begin{figure}[ht]
    \centering
\begin{tikzpicture}
      \node[label={above, yshift=0cm:}] at (3.3,2.4) {$\cdot$};
      \node[label={above, yshift=0cm:}] at (3.4,2.6) {$\cdot$};
      \node[label={above, yshift=0cm:}] at (3.65,2.7) {$\cdot$};      \node[label={above, yshift=0cm:}] at (3.3,1.45) {$\cdot$};
      \node[label={above, yshift=0cm:}] at (3.35,1.25) {$\cdot$};
      \node[label={above, yshift=0cm:}] at (3.5,1.1) {$\cdot$};

      \tikzset{enclosed/.style={draw, circle, inner sep=0pt, minimum size=.1cm, fill=black}}
      
      \node[enclosed,black, label={right, yshift=.2cm:}] at (2,2) {};
     
      \draw(2,2) -- (4,2) node[midway, right] {};

      \node[enclosed,black, label={right, yshift=.2cm:}] at (4,2) {};
      \node[enclosed,black, label={right, yshift=.2cm:}] at (5.5,3.3) {};
      \node[enclosed,black, label={right, yshift=.2cm:}] at (6,2) {};
      \node[enclosed,black, label={right, yshift=.2cm:}] at (4.3,3.8) {};

      \draw(4,2) -- (5.5,3.3) node[midway, right] {};
      \draw(4,2) -- (4.3,3.8) node[midway, right] {};
      \draw(4,2) -- (6,2) node[midway, right] {};

      \node[label={above, yshift=0cm:}] at (3.7,1.7) {\small$n$};
      \node[label={above, yshift=0cm:}] at (4.2,2.4) {\small$3$};
      \node[label={above, yshift=0cm:}] at (2.2,1.7) {\small$2$};

      \node[label={above, yshift=0cm:}] at (5,1.75) {\small$e_1$};
      \node[label={above, yshift=0cm:}] at (4.9,2.5) {\small$e_2$};
      \node[label={above, yshift=0cm:}] at (4.4,3) {\small$e_3$};
      \node[label={above, yshift=0cm:}] at (3,1.75) {\small$e_{n-1}$};

      \node[label={above, yshift=0cm:}] at (5.1,3.3) {\small$2$};
      \node[label={above, yshift=0cm:}] at (4.3,1.7) {\small$2$};
      \node[label={above, yshift=0cm:}] at (5.8,1.7) {\small$2$};
      \node[label={above, yshift=0cm:}] at (3.85,2.35) {\small$4$};
      \node[label={above, yshift=0cm:}] at (4,3.7) {\small$2$};

      \draw[->] (4.9, 2) -- (5, 2);
      \draw[->] (4.81, 2.7) -- (4.92, 2.8);
      \draw[->] (4.15, 2.9) -- (4.17, 3);
      \draw[->] (3, 2) -- (2.9, 2);

\end{tikzpicture}
    \caption{Octopus graph}
    \label{fig: oct}
\end{figure}

\begin{prop}\label{prop: non locally finite GBS}
    Let $\IG=(\Gamma, \CG)$ be a GBS octopus graph of groups as in Figure \ref{fig: oct} above. Then $\pi_1(\IG, v)\curvearrowright \partial_\infty X_\IG$ is a topological free strong boundary action. As consequences, $\pi_1(\IG, v)$ is $C^*$-simple and $C(\overline{\partial_\infty X_\IG})\rtimes_r \pi_1(\IG, v)$ is a unital Kirchberg algebra satisfying the UCT.
\end{prop}
\begin{proof}
    First, the graph of groups $\IG$ is reduced. Then Propositions \ref{prop: linear}, \ref{prop: no loop minimal} and Lemma \ref{Lem: not loop repeatable} implies that 
    \begin{enumerate}[label=(\roman*)]
        \item $X_\IG$ is not isomorphic to a line;
        \item  there exists a repeatable word $w\in \pi_1(\IG, v)$ with $o(w)=v$;
        \item $\xi$ flows to $f$ holds for any infinite normalized word $\xi$ and edge $f\in E(\Gamma)$.
        \end{enumerate}
Then Theorem \ref{thm: boundary action} implies that $\pi_1(\IG, v)\curvearrowright\overline{\partial_\infty X_\IG}$ is a strong boundary action.  On the other hand, let $F\subset G_v\setminus\{1_v\}\simeq \Z\setminus \{0\}$ and $E\subset \{f\in E(\Gamma): o(f)=v\}$ be finite set. Then choose a large enough prime number $p\notin F$ such that $e_{p-1}\notin E$. Define $f=e_{p-1}$ and then $\alpha_f(G_f)=p\Z$ and therefore $m\notin \alpha_f(G_f)$ for any $m\in F$. Thus, Proposition \ref{prop: infinite GBS topo free} shows that $\pi_1(\IG, v)\curvearrowright X_\IG$ is strongly faithful and thus $\pi_1(\IG, v)\curvearrowright \overline{\partial_\infty X_\IG}$ is topologically free by Proposition \ref{prop: strongly faithful and topo free}.  Therefore, $\pi_1(\IG, v)$ is $C^*$-simple by Theorem \ref{thm: C simple equivalence} because $\overline{\partial_\infty X_\IG}$ is a topologically free $\pi_1(\IG, v)$-boundary action. Moreover, the reduced crossed product $C(\overline{\partial_\infty X_\IG})\rtimes_r \pi_1(\IG, v)$ is a unital simple purely infinite $C^*$-algebra by Theorem \ref{thm: strong bdry action pure infinite}.     
\end{proof}

Note that $\pi_1(\IG, v)$ in Proposition \ref{prop: non locally finite GBS} is also unimodular. Therefore, Theorem \ref{thm: GBS topological freeness} does not hold if $\Gamma$ is infinite.

\subsection{n-dimensional case}\label{subsec: gbsn}
In this subsection, we study $n$-dimensional GBS groups. Denote by $v_i=(0,\dots, 0, 1, 0,\dots, 0)$ the $i$-th generator for $\Z^n$ for $i=1,\dots, n$ in this subsection. Let $e$ be an edge in a $\text{GBS}_n$ graph. For monomorphisms $\alpha_e$ and $\alpha_{\bar{e}}$, denote by $\alpha_e(v_i)=(k_{i, 1}, k_{i, 2}, \dots, k_{i, n})\in \Z^n$ and $\alpha_{\bar{e}}(v_i)=(l_{i, 1},\dots, l_{i, n})\in \Z^n$ for $i=1,\dots, n$. We denote by $A_{e}=[k_{i, j}]$ and $A_{\bar{e}}=[l_{i, j}]$ the $n\times n$ matrices assigned to $\alpha_e$ and $\alpha_{\bar{e}}$, respectively.  Note that $A_e, A_{\bar{e}}\in \operatorname{GL}(n, \Q)$ as $\alpha_e$ and $\alpha_{\bar{e}}$ are also $\Z$-module monomorphisms. In this subsection, we also denote $\vec{0}$ for the neutral element in $\Z^n$.

\begin{lem}\label{lem: hnn GBSn}
    Let $\IG=(\Gamma, \CG)$ be a graph of groups such that $\Gamma$ is a non-ascending loop $e$ with $o(e)=t(e)=v$ such that $G_e\simeq \Z^n$ and $G_v\simeq \Z^n$.   Suppose for any  $x\in \Z^n\setminus \{\vec{0}\}$, there exits an $m\geq 1$ such that $(A_e^{-1}A_{\bar{e}})^m x\notin \Z^n$.     
 Then $\pi_1(\IG, v)$ is $C^*$-simple. 
\end{lem}
\begin{proof}
    In light of Theorem \ref{thm: c star simple HNN}, it suffices to show for any $F\subset \alpha_e(G_e)\setminus \{1_v\}=\Z^n\setminus \{\vec{0}\}$, there exits a $\IG$-normalized word 
    $y=g_1e^{\epsilon_1}g_2e^{\epsilon_2}\dots g_me^{\epsilon_m}g_{m+1}$ 
such that $N(yxy^{-1})\notin \alpha_e(G_e)$ for any $x\in F$, where $g_i\in G_v$, $\epsilon_i\in \{1, -1\}$ for $i=1,\dots, m+1$.
    
Write $x\in G_e$ as a column vector in $\Z^n$. Then $\alpha_e(x)=A_ex$ and $\alpha_{\bar{e}}(x)=A_{\bar{e}}x$. This implies that
\[e^{-1}A_exe=A_{\bar{e}}x.\]
Moreover, we claim that for $k\geq 1$, one has $e^{-k}A_exe^k= A_{\bar{e}}(A^{-1}_eA_e)^{k-1}x$ whenever $e^{-l}A_exe^l\in \alpha_e(G_e)$ holds for any $l<k$.

Indeed, by induction, suppose $e^{-l}A_exe^l=A_{\bar{e}}(A^{-1}_eA_{\bar{e}})^{l-1}x\in \alpha_e(G_e)$ holds for $l\geq 1$. Then there exists a $u\in G_e$ such that $A_{\bar{e}}(A^{-1}_eA_{\bar{e}})^{l-1}x=A_eu$, which implies that $u=(A^{-1}_eA_{\bar{e}})^lx$. Therefore, one has
\[e^{-l-1}A_exe^{l+1}=e^{-1}A_eue=A_{\bar{e}}u=A_{\bar{e}}(A^{-1}_eA_{\bar{e}})^lx.\]
This finishes the induction.

Then by the assumption, for and $x\in \Z^n\setminus \{\vec{0}\}$, choose the smallest $m_x\geq 1$ such that  $(A_e^{-1}A_{\bar{e}})^m x\notin \Z^n$. This implies the element
\[g_x\coloneqq e^{-{m_x}}A_exe^{m_x}=A_{\bar{e}}(A^{-1}_eA_{\bar{e}})^{m_x-1}x\in G_v\setminus \alpha_e(G_e).\]
Then for any $m>m_x$, one has 
\[e^{-m}A_exe^m=e^{-(m-m_x)}g_xe^{m-m_x}=1_v\bar{e}\dots1_v\bar{e}g_xe1_v\dots 1_ve\]
is a reduced word that is not in $\alpha_e(G_e)$ as $g_x\in G_v\setminus \alpha_e(G_e)$. Therefore, Let $F\subset \alpha_e(G_e)\setminus \{1_v\}$, choose $y=e^m$ for a large enough $m$ and the above implies that $N(yxy^{-1})\notin \alpha_e(G_e)$ for any $x\in F$.
Thus, the group $\pi_1(\IG, v)$ is $C^*$-simple.
\end{proof}

We note that Lemma \ref{lem: hnn GBSn} has generalized the ``if'' part for a non-solvable  $BS(k, l)$ in \cite[Theorem 3 (iii)]{H-P}. Indeed, for such $BS(k, l)$, the matrix $A_e=k$ and $A_{\bar{e}}=l$ are integers that not equal $1$ or $-1$. Then if $|k|\neq |l|$, without loss of generality, one may assume $|k|>|l|$, which implies $A^{-1}_eA_{\bar{e}}=l/k$. Then for any $x\in \Z\setminus \{0\}$, one may always choose a large $m\geq 1$ such that $(l/k)^mx\notin \Z$.

Then as a direct application of Theorem \ref{thm: reduced graph C star simple}, we have the following for $\text{GBS}_n$ groups.

\begin{thm}\label{thm: main GBSn}
    Let $\IG=(\Gamma, \CG)$ be a $\text{GBS}_n$ graph of groups containing a non-ascending loop $e$ such that for any $x\in \Z^n\setminus \{\vec{0}\}$, there exits an $m\geq 1$ such that $(A_e^{-1}A_{\bar{e}})^m x\notin \Z^n$.     
 Then $\pi_1(\IG, v)$ is $C^*$-simple. Morevoer, the crossed product $C(\overline{\partial_\infty X_\IG})\rtimes_r\pi_1(\IG, v)$ is a unital Kirchberg $C^*$-algebra satisfying the UCT.
 \end{thm}
 \begin{proof}
     In light of Lemma \ref{lem: hnn GBSn} and Theorem \ref{thm: reduced graph C star simple}, it is left to show $C(\overline{\partial_\infty X_\IG})\rtimes_r\pi_1(\IG, v)$ is nulcear. But this follows from Remark \ref{rmk: equivalence of dynamical properties} as the boundary action $\pi_1(\IG, v)\curvearrowright\partial_\infty X_{\IG}$ is topological amenable by  \cite[Proposition 5.2.1, Lemma 5.2.6]{B-O} since all vertex groups are $\Z^n$.
     \end{proof}

\begin{eg}\label{eg: gbsn diagonal}
  We provide some elementary concrete $\text{GBS}_n$ examples satisfying Theorem \ref{thm: main GBSn}.  Let $e$ be a loop.  Suppose $A_e=\diag(k_1,\dots, k_n)$ and $A_{\bar{e}}=\diag(l_1,\dots,l_n)$ are matrices in $\text{GL}(n, \Q)$. Then $A^{-1}_eA_{\bar{e}}=\diag(l_1/k_1,\dots, l_n/k_n)$ and thus for any $m\geq 1$ and $x=(x_1,\dots, x_n)\in \Z^n$ one has 
  \[(A^{-1}_eA_{\bar{e}})^mx=(l^m_1x_1/k^m_1,\dots, l^m_nx_m/k^m_n).\]
Therefore, if there exists an $i\leq n$ such that $|k_i|\neq |l_i|$, then for any finite $x\in \Z^n$, there exists a large enough $m$ such that $(A^{-1}_eA_{\bar{e}})^mx\notin \Z^n$ for any $x\in F$. Therefore, in this case, the $\text{GBS}_n$ graph of groups $\IG$ satisfies Theorem \ref{thm: main GBSn}.
\end{eg}

We now provide more complicated examples.

\begin{prop}\label{prop: hnn special gbs2}
    Let $\IG=(\Gamma, \CG)$ be a $\text{GBS}_2$ graph of groups such that $\Gamma$ is a non-ascending loop $e$ with $o(e)=t(e)=v$ and $G_v, G_e$ are all isomorphic to $\Z^2$. Suppose the matrix $M=A^{-1}_eA_{\bar{e}}$ is a unitary of the form
    \[\begin{bmatrix}
        \cos\theta & \sin\theta\\
        -\sin\theta & \cos\theta
    \end{bmatrix}
    \]
    in which $\theta$ is irrational. Then $\pi_1(\IG, v)$ is $C^*$-simple.
\end{prop}
\begin{proof}
    We apply Lemma \ref{lem: hnn GBSn}. By induction, for the $M$ and an integer $k\geq 1$, one has 
     \[M^k=\begin{bmatrix}
        \cos k\theta & \sin k\theta\\
        -\sin k\theta & \cos k\theta
    \end{bmatrix}.
    \]
 Now, for $x=(y, z)\in \Z^2$, this implies that 
 \[M^kx=\begin{bmatrix}
       z\sin(k\theta)+y\cos(k\theta)\\
       z\cos(k\theta)-y\sin(k\theta)
   \end{bmatrix}\]
Now, note that the irrational rotation by $\theta$ on the unit circle $\T$ is minimal and thus for this $\Z$-minimal action, it is a standard fact that its any half-orbit $e^{in\theta}$, for $n\in \N$, is still dense in $\T$. 
   
Fix an $0<\epsilon<1/4$. For any $x=(y, z)\in \Z^2\setminus \{(0, 0)\}$, choose an $m\in \N$ large enough such that $\sin(m\theta), \cos(m\theta)>0$,  $|\sin(m\theta)y|<\epsilon$, $|\sin(m\theta)z|<\epsilon$, $|\cos(m\theta)z-z|<\epsilon$ and $|\cos(m\theta)y-y|<\epsilon$. If one of $y, z$ is $0$, then it is direct to see $M^kx\notin \Z^2$. Suppose $yz>0$ holds. Then, either
\[z>\cos(m\theta)z-\sin(m\theta)y>z-2\epsilon\]
when $y, z>0$, or
\[z+2\epsilon>\cos(m\theta)z-\sin(m\theta)y>z\]
when $y,z<0$.
This implies that $\cos(m\theta)z-\sin(m\theta)y\notin \Z$.
For the case $yz<0$, the similar argument also shows that $z\sin(n\theta)+y\cos(n\theta)\notin \Z$.
Therefore, in any case, one has $M^mx\notin \Z^2$, which has verified the condition in Lemma \ref{lem: hnn GBSn}. Thus $\pi_1(\IG, v)$ is $C^*$-simple.
\end{proof}

\begin{rmk}\label{rmk: LM group}
    An example satisfying Proposition \ref{prop: hnn special gbs2} is the famous \textit{Leary-Minasyan} group 
\[G_{P}=\langle t, a, b| [a, b], ta^2b^{-1}t^{-1}=a^2b, tab^2t^{-1}=a^{-1}b^2 \rangle\]
introduced in \cite{LM}. By the presentation of $G_{P}$, the group $G_{P}=\pi_1(\IG, v)$, where $\IG$ consists of one loop $e$ with the relation
   \[ta^5t^{-1}=a^3b^4\text{ and } tb^5t^{-1}=a^{-4}b^3,\] 
   which implies that $M=A^{-1}_eA_{\bar{e}}=\begin{bmatrix}
       3/5 & -4/5\\
       4/5 & 3/5
   \end{bmatrix}$ satisfying the assumption of Proposition \ref{prop: hnn special gbs2}. Therefore,  Leary-Minasyan group $G_{P}$ is $C^*$-simple.
   \end{rmk}

As a direct corollary of Theorem \ref{thm: main GBSn}, Proposition \ref{prop: hnn special gbs2} and Remark \ref{rmk: LM group}, we have the following.

\begin{cor}\label{cor: gbs2 including LM}
Let $\IG=(\Gamma, \CG)$ be a $\text{GBS}_2$ graph of groups containing a non-ascending loop $e$ such that   $M=A^{-1}_eA_{\bar{e}}$ is a unitary in $\operatorname{GL}(2, \Q)$ of the form
    \[\begin{bmatrix}
        \cos\theta & \sin\theta\\
        -\sin\theta & \cos\theta
    \end{bmatrix}
    \]
    in which $\theta$ is irrational (e.g., the loop $e$ as a subgraph yielding Leary-Minasyan group $G_p$ ).  Then $\pi_1(\IG, v)$ is $C^*$-simple. Morevoer, the $C*$-algebra $C(\overline{\partial_\infty X_\IG})\rtimes_r\pi_1(\IG, v)$ is a unital Kirchberg $C^*$-algebra satisfying the UCT.
\end{cor}

\begin{rmk}\label{rmk: non ah gbs}
    We  note that every (finitely generated) $\text{GBS}_n$ group $G$ is not acylindrically hyperbolic because any vertex group $\Z^n$ can be verified to be s-normal in $G$ and this seems to be a standard fact (see \cite[Section 4.2]{Button2}). Moreover, the same approach also shows that  the  GBS octopus graph of groups in Proposition \ref{prop: non locally finite GBS} is not acylindrically hyperbolic. Moreover, the Leary-Minasyan group $G_P$ is shown to be not virtually hierarchically hyperbolic  by \cite[Theorem 5.3]{Button2} as well. 
    \end{rmk}

\section{Acknowledgements}
The authors would like to thank Petr Naryshkin, Tron Omland, and Jack Spielberg for helpful comments. In addition, the authors are very grateful to Ashot Minasyan and Motiejus Valiunas for kindly communicating with them that $C^*$-simplicity of one-relator groups and $\text{GBS}_1$ groups have been proven in \cite{M-V}, bringing their attention to \cite{Button3}, and other useful comments on acylindrical hyperbolic groups after the first version of the paper is posted on arXiv. Moreover, the authors would like to thank the anonymous referee for very helpful comments on the paper.  W. Y. is partially supported by National Key R \& D Program of China (SQ2020YFA070059) and  National Natural Science Foundation of China (No. 12131009 and No.12326601).



\end{document}